\documentclass[10pt]{amsart}

\usepackage[foot]{amsaddr}
\usepackage[margin=2.5cm,bottom=2.8cm,top=2cm,footskip=1.5cm]{geometry}

\usepackage{graphicx}

\usepackage[english]{babel}

\usepackage[latin1]{inputenc}

\usepackage{hyperref}
\hypersetup{colorlinks=true}

\usepackage{color}

\usepackage{amsmath,amsfonts,amsthm,amssymb}

\usepackage{epsfig}
\usepackage{amscd}
\usepackage{mathtools}
\mathtoolsset{showonlyrefs}

\numberwithin{figure}{section}
\theoremstyle{plain}
\newtheorem{thm}{Theorem}[section]

\newtheorem{prop}[thm]{Proposition}

\newtheorem{cor}[thm]{Corollary}
\newtheorem{lemma}[thm]{Lemma}
\numberwithin{equation}{section}

\theoremstyle{remark}

\newtheorem{rmq}[thm]{Remark}

\newtheorem{dfn}[thm]{Definition}

\renewcommand\Re{\mathrm{Re}\,} \renewcommand\Im{\mathrm{Im}\,}
\newcommand\R{{\mathbb R}} \newcommand\N{{\mathbb N}}
 
\newcommand\C{{\mathbb C}}

\renewcommand{\Im}{  \text{Im}   }
\renewcommand{\Re}{  \text{Re}   }

\def\e{\varepsilon} \def\cdotv{\raise 2pt\hbox{,}}

\renewcommand\Re{\mathrm{Re}\,} \renewcommand\Im{\mathrm{Im}\,}

\makeatletter
\def\@tvsp{\mathchoice{{}\mkern-4.5mu}{{}\mkern-4.5mu}{{}\mkern-2.5mu}{}}
\def\ltrivert{\left|\@tvsp\left|\@tvsp\left|}
\def\rtrivert{\right|\@tvsp\right|\@tvsp\right|}
\makeatother

\def\e{\varepsilon} \def\cdotv{\raise 2pt\hbox{,}}

\newcommand{\del}{\partial}
\newcommand{\ii}{{\rm i}}

\begin{document}

\parindent=0pt

\title[Dispersion for wave and Schr\"{o}dinger outside a ball and counterexamples]{Dispersion for the wave and Schr\"{o}dinger equations outside a ball and counterexamples}

\author{Oana Ivanovici${}^{1}$}
\address{${}^{1}$Sorbonne Université, CNRS, LJLL, F-75005 Paris, France} 
\email{oana.ivanovici@math.cnrs.fr}

 \thanks{Oana Ivanovici was supported by ERC grant ANADEL 757 996.} 
\maketitle 

 \date{\today}

\maketitle

\begin{abstract}
We consider the wave equation with Dirichlet boundary conditions in the exterior of the unit ball $B_{d}(0,1)$ of $\mathbb{R}^d$. For $d=3$, we obtain a global in time parametrix and derive sharp dispersive estimates, matching the $\mathbb{R}^{3}$ case, for all frequencies (low and high). For $d\geq 4$, we provide an explicit solution at large frequency $1/h$, $h\in (0,1)$, with a smoothed Dirac data at a point at distance $h^{-1/3}$ from the origin in $\R^d$ whose decay rate exhibits $h^{-(d-3)/3}$ loss with respect to the boundary less case, that occurs at observation points around the mirror image of the source with respect to the center of the ball (at the Poisson-Arago spot). Similar counterexample are obtained for the Schrödinger flow. 

Moreover, we generalize these counterexamples, first announced in \cite{ildispext}, to the case of the wave and Schrödinger equations outside cylindrical domains of the form $B_{d_1}(0,1)\times \mathbb{R}^{d_2}$ in $\mathbb{R}^d$ with $d=d_1+d_2$ and $d_1\geq 4$, for which we construct solutions, as done \cite{IaIv23} for $d_1=2$, $d_2=1$, whose decay rates exhibit a $h^{-(d_1-3)/3}$ loss with respect to the boundary less case (at observation points around the mirror image of the source with respect to the origin). 
\end{abstract}

\section{Introduction}
We consider the linear wave
equation on an exterior domain $\Omega\subset \mathbb{R}^{d}$ with smooth boundary; we let $\Delta_\Omega$ denote the Laplacian with constant coefficients and Dirichlet boundary conditions on $\Omega$ :
\begin{equation} \label{WE} 
\left\{ \begin{array}{l}
   (\partial^2_t-
 \Delta_\Omega) u(\cdot,t)=0 \;\; \text{ in } \Omega, \quad
 u(\cdot,t)\Big|_{\partial\Omega}=0, \\ 
 u(\cdot,0) = u_0, \; \partial_t u(\cdot,0)=u_1.
 \end{array} \right.
 \end{equation}
We also address the Schrödinger equation
\begin{equation} \label{SchE} 
 (i\partial_t+\Delta_\Omega) v(\cdot,t)=0,  \;\; \text{ in } \Omega, \quad 
 v(\cdot,t)\Big|_{\partial\Omega} = 0,
\quad  v(\cdot,0) = v_0. 
 \end{equation}
Heuristically, dispersion relates to how waves spread out with time, while retaining their energy: it quantifies decay for waves' amplitude. In $\mathbb{R}^d$, the half-wave propagator $e^{\pm it\sqrt{-\Delta_{\R^d}}}$ can be computed explicitly, yielding the dispersion estimate, for  $\chi\in C_{0}^\infty
((1/2,2))$ and $D_{t}=-i\partial_{t}$,
\begin{equation}\label{disprd}
\|\chi(hD_t)e^{\pm it\sqrt{-\Delta_{\R^d}}}\|_{L^1(\mathbb{R}^d)\rightarrow L^{\infty}(\mathbb{R}^d)}\leq C(d)h^{-d}\min\{1,(h/|t|)^{\frac{d-1}{2}}\}\,.
\end{equation}
For the Schr\"odinger propagator, dispersion follows at once from its explicit Gaussian kernel:
\begin{equation}\label{dispschrodrd}
\|e^{\pm it\Delta_{\R^d}}\|_{L^1(\mathbb{R}^d)\rightarrow L^{\infty}(\mathbb{R}^d)}\leq C(d)|t|^{-d/2}\,.
\end{equation}
These estimates, together with the energy conservation, allow to obtain the whole set of Strichartz estimates (although the endpoints are more delicate, see \cite{keta98}). 
Strichartz estimates for ${\mathbb R}^d$ and manifolds without boundary have been understood for some time (see \cite{stri77},\cite{give85}, \cite{give95}, \cite{lev90}, \cite{ls95} for $\mathbb{R}^d$ and also \cite{sm98}, \cite{tat02} for low regularity metrics). Even though the boundary-less case has been well understood, obtaining results for the case of manifolds {\it with} boundary has been surprisingly elusive.

Our aim in the present paper is to prove dispersive bounds for the wave and the Schrödinger equations outside a ball in $\mathbb{R}^d$, as announced in \cite{ildispext} : for $d=3$ we prove that both linear flows satisfy the same dispersion bounds as in $\mathbb{R}^3$. In higher dimensions $d\geq 4$ we prove that these estimates cannot hold as in $\mathbb{R}^d$ as losses do appear at the Poisson-Arago spot. We then generalize this loss to exterior domains of the form $\R^d\setminus B_{d_1}(0,1)\times \mathbb{R}^{d_2}$ where $d=d_1+d_2$, $d_1 \geq 4$ (notice that, although our construction applies for $d_1=2,3$ as well, this case is uninteresting for our purposes as it is very similar to the one of \cite{IaIv23} where the dispersive bounds do hold as in the free case).

While several positive results on dispersive effects on exterior domains have been established since the mid-90's, the question about whether or not dispersion did hold remained completely open, even for the exterior of a ball. We recall that (local in time) Strichartz estimates were proved to hold as in $\mathbb{R}^{d}$ for the wave equation in \cite{smso95} and for the Schr\"odinger equation in \cite{doi08}, with arguments that did not require the dispersion estimate to be known (both estimates are known to extend globally in time using local energy decay or local smoothing estimates). Since there is no obvious concentration of energy, like in the case of a generic non-trapping obstacle, where concave portions of the boundary can act as mirrors and refocus wave packets, one would expect dispersive estimates to hold outside strictly convex obstacles (for spherically symmetric functions this was proved, outside a sphere in \cite{lismza12}). Our main  positive result is the following:

\begin{thm}\label{thmdisp3D}
Let $\Theta=B_3(0,1)\subset\mathbb{R}^3$ be the unit ball in $\mathbb{R}^3$ and set $\Omega=\mathbb{R}^3\setminus \Theta$. Let $\Delta_{\Omega}$ denote the Laplace operator in $\Omega$ with Dirichlet boundary condition and let $\chi\in C^{\infty}_0((\frac 12,2))$.  
\begin{enumerate}
\item Dispersion holds for the wave flow in $\Omega$ like in $\mathbb{R}^3$:
 \begin{equation}\label{dispomega3}
\|\chi(hD_t)e^{\pm it\sqrt{-\Delta_\Omega}}\|_{L^1(\Omega)\rightarrow L^{\infty}(\Omega)}\lesssim h^{-3}\min\{1,{h}/{|t|}\}.
\end{equation}
\item Dispersion holds for the classical Schr\"odinger flow in $\Omega$  like in $\mathbb{R}^3$: 
 \begin{equation}\label{dispschrodomega3}
\|e^{\pm it\Delta_\Omega}\|_{L^1(\Omega)\rightarrow L^{\infty}(\Omega)}\lesssim |t|^{-3/2}.
\end{equation}
\end{enumerate}
\end{thm}
In Theorem \ref{thmdisp3D} and in the remaining of the paper, $A\lesssim B$ means that there exists a constant $C$ such that $A\leq CB$ and this constant may change from line to line but is independent of all parameters. It will be explicit when (very occasionally) needed. Similarly, $A\sim B$ means both $A\lesssim B$ and $B\lesssim A$.
\begin{rmq}
Having the full dispersion in $3D$ we can immediately obtain the endpoint Strichartz estimates, following \cite{keta98}. Theorem \ref{thmdisp3D} also helps in dealing with the non-linear Schr\"odinger equation (see e.g. the recent work \cite{NLSscatteringExt}, which makes use of our dispersion estimate to reprove the main result from \cite{kivizh12}, on global well-posedness for the defocusing energy critical NLS in $3D$.)
\end{rmq}
\begin{rmq}
We claim that Theorem \ref{thmdisp3D} still holds with $\Omega$ replaced by the exterior of an obstacle with smooth, strictly geodesically concave boundary. By the last condition we mean that the second fundamental form on the boundary is positive definite. The general case of Theorem \ref{thmdisp3D} requires arguments that may be seen as perturbative of those, mostly explicit, used in the exterior of a ball and will be dealt with elsewhere.
\end{rmq}

A loss in dispersion may be related to a cluster point: such clusters occur when optical rays, sent along different directions, are no longer diverging from each other. If a point source  illuminates a ball, by Huygens's principle every point of the obstacle acts as a new point source: diffraction by the obstacle deviates light on the boundary which arrives at the center of the shadow behind the obstacle in phase and constructively interferes. This results in a bright spot at the shadow's center. Therefore, our intuition tells us that a location where dispersion fails should be the Poisson-Arago spot (which really is part of the line joining the source and the center of the obstacle: in physical experiments, one may place a screen somewhat symmetrically to the source of light, hence the choice of wording): it turns out that, indeed, that region is brighter than the illuminated regime when $d\geq 4$.

Before stating our counterexamples to dispersion, let us describe the domain we consider (which may be either the exterior of a ball or the exterior of a cylindrical obstacle in $\R^d$).
Let $d\geq 4$ and consider a cylinder in $\mathbb{R}^d$ of the form $B_{d_1}(0,1)\times \mathbb{R}^{d_2}$ where $d_{1,2}$ are such that $d_1+d_2=d$ and $d_1\geq 3$, $d_2\geq 0$ (notice that, although our construction applies for $d_1=2$ as well, this case is uninteresting for our purposes as it is very similar to the one of \cite{IaIv23} where the dispersive bounds did hold as in the free case, as the diffractive effects in the shadow are weaker). Here $B_{d_1}(0,1)$ denotes the unit ball of center $0\in \mathbb{R}^{d_1}$. We define $\Omega_{d_1,d_2}:=\Big(\mathbb{R}^d\setminus B_{d_1}(0,1)\Big)\times \mathbb{R}^{d_2}$, whose boundary is of the form $\mathbb{S}^{d_1-1}\times \{z=(z_1,...,z_{d_2})\}$. Taking spherical  coordinates 
\begin{equation}\label{polarcoord}
\left\{ \begin{array}{l}
x_{d_1}=r\cos\varphi,\quad x_{d_1-1}=r\sin\varphi \times  \cos \omega_1, \\
x_k=r\sin\varphi \times \Big(\Pi_{j=1}^{d_1-k-1}\sin\omega_j\Big) \cos \omega_{d_1-k},  2\leq k\leq d_1-2, \quad
x_1=r\sin \varphi \times \Pi_{j=1}^{d_1-2}\sin\omega_j ,
 \end{array} \right.
\end{equation}
we obtain 
\begin{equation}\label{omega}
\Omega_{d_1,d_2}
=\{(r,\varphi,\omega, z), r\geq 1, \varphi\in [0,\pi], \omega_j\in [0,\pi], \forall 1\leq j\leq d_1-3, \omega_{d_1-2}\in [0,2\pi), z\in \mathbb{R}^{d_2}\}.
\end{equation}
In these coordinates, the usual Laplace operator takes the form 
\begin{equation}\label{laplaballrS2}
\Delta=\frac{1}{r^{d_1-1}}\frac{\partial}{\partial r}\Big(r^{d_1-1}\frac{\partial}{\partial r}  \Big)+\frac{1}{r^2}\Delta_{\mathbb{S}^{d_1-1}}+\Delta_{\mathbb{R}^{d_2}},
\end{equation}
where $\Delta_{\mathbb{R}^{d_2}}=\sum_{k=1}^{d_2}\partial^2_{z_k}$ is the usual Laplace operator in $\mathbb{R}^{d_2}$ acting on $z$, $\Delta_{\mathbb{S}^{n}}$ denotes the Laplace operator on $\mathbb{S}^{n}$ and 
$\Delta_{\mathbb{S}^{d_1-1}}=\frac{1}{(\sin\varphi)^{d_1-2}}\frac{\partial}{\partial\varphi}\Big((\sin\varphi)^{d_1-2}\frac{\partial}{\partial\varphi}\Big)+\frac{1}{(\sin\varphi)^2}\Delta_{\mathbb{S}^{d_1-2}}$. 

\begin{tabular}{cl}
 \parbox[b][3mm][t]{0.75\textwidth}{ \vskip-2.4cm {When $d_2=0$, $d_1=d$, the domain $\Omega_{d,0}=\mathbb{R}^d\setminus B_d(0,1)$ is the exterior of the sphere. In dimension $d=d_1=3$, $d_2=0$, the coordinates are shown in the picture, with $\varphi\in [0,\pi]$ and $\omega\in [0,2\pi)$. In higher dimension $d=d_1>3$, $\varphi$ remains unchanged and $\omega=(\omega_1,...,\omega_{d_1-2})$ with $ \omega_j\in [0,\pi], \forall 1\leq j\leq d_1-3, \omega_{d_1-2}\in [0,2\pi)$. When $d_2>0$, the "obstacle" is the product of the $d_1$-dimensional ball with $\R^{d_2}$.}}
 &
 \includegraphics[width=0.16\textwidth]{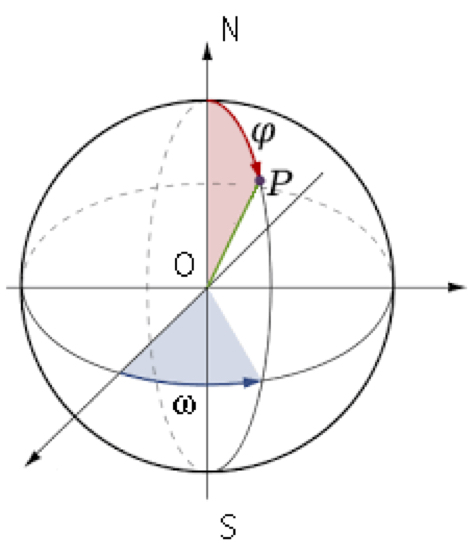}
\end{tabular}
%

\begin{thm}\label{thmCE}
Let $d\geq 4$, $d_{1,2}\in \N$ with $d=d_1+d_2$. Let $\Delta_{\Omega_{d_1,d_2}}$ be the Laplace operator in $\Omega_{d_1,d_2}$ with Dirichlet boundary condition, that is \eqref{laplaballrS2} with Dirichlet condition at $r=1$. Let $Q_{\pm}(s)\in \Omega_{d_1,d_2}\cap \R^{d_1}\times \{0\}^{d_2}$ be two points of $\Omega_{d_1,d_2}$, symmetric with respect to the center of the ball $B_{d_1}(0,1)\times \{0_{\mathbb{R}^{d_2}}\}$, at distance $s>1$ from $\{0\}\in \mathbb{R}^d$. Let $\chi\in C^{\infty}_0((\frac 12,2))$ be a smooth cut-off equal to $1$ near $1$, $0\leq \chi\leq 1$.
\vskip2mm

For the wave operator, there exists $0<h_0<1$ such that, for all $0<h\leq h_0$ and for $s\sim h^{-1/3}$ 
the following estimates hold at $t\sim 2(\sqrt{s^2-1}+\arcsin \frac 1s)\sim 2 h^{-1/3}$
\[
\Big|\chi(hD_t)e^{it\sqrt{-\Delta_{\Omega_{d_1,d_2}}}}(\delta_{Q_+(s)})\Big| (Q_-(s))\sim h^{-d}\Big(\frac ht\Big)^{\frac{d-1}{2}} h^{-\frac{d_1-3}{3}}.
\]

For the Schr\"odinger operator, let $0<h\leq h_0^2$, $s\sim h^{-1/6}$, $t\sim h^{1/3}$. Then the following holds
\[
\Big|\chi(hD_t)e^{it\Delta_{\Omega_{d_1,d_2}}}(\delta_{Q_+(s)})\Big| (Q_-(s))\sim h^{-\frac d6-\frac{d_1-3}{6}}, \quad h^{-\frac d6}\sim t^{-\frac d2}.
\]
For $d_1\geq 4$, these estimates contradict the usual (flat) ones \eqref{dispschrodrd}.
\end{thm}

\begin{rmq}
For $d=d_1\geq 4$, Theorem \ref{thmCE} provides a first example of a domain on which global in time Strichartz estimates do hold like in $\mathbb{R}^d$ (see \cite{smso95}, \cite{doi08}) while dispersion fails.
\end{rmq}
\begin{rmq}
The proof of Theorem \ref{thmCE} is based on an explicit parametrix construction : when $d_2=0$, such a parametrix is obtained in the proof of Theorem \ref{thmdisp3D} in $3D$ and then generalized to higher dimensions. When $d_2\geq 1$, it is a generalization of the one from \cite{IaIv23}, where $d_1=2$ and $d_2=1$, where sharp dispersive estimates for the wave equation (as in Theorem \ref{thmdisp3D}) have been proved in the domain $\Omega_{2,1}$ with Laplace operator \eqref{laplaballrS2}.
\end{rmq}

\begin{rmq}\label{rmqellipse}
In physical experiments the Arago spot is very sensitive to small-scale deviations from the ideal circular cross-section: if there is no "source" point whose apparent contour (defined as the boundary of the set of points that can be viewed from the source) may produce constructive interferences, then the odds of obtaining a bright spot in the shadow region drastically decrease. If the cross-section of the obstacle deviates from a circle, the shape of the Poisson-Arago spot changes and becomes a caustic. In particular, if the object has an ellipsoidal cross-section, the Poisson-Arago spot has the shape of an evolute (i.e. the locus of the center of curvature while moving along the ellipsoid), for which the losses are less important. More generalized counterexamples outside an ellipsoid in $\R^d$ will be provided in a forthcoming  work (more generally, we will show that even generic strictly convex obstacles in higher dimensions provide losses in dispersion).
\end{rmq}
The Poisson-Arago spot is an optical phenomenon where a bright point appears at the center of a circular object's shadow. 
This phenomenon was pivotal in scientifically establishing light's wave nature and has since become a standard demonstration in undergraduate physics courses.
In his 1819 report to the french Academy of Sciences, Arago mentions: {\it "L'un de vos commissaires, M.Poisson, avait déduit des intégrales rapportées par l'auteur, le résultat singulier que le centre de l'ombre d'un écran circulaire opaque devait, lorsque les rayons y pénétraient sous des incidences peu obliques, être aussi éclairé que si l'écran n'existait pas. Cette conséquence a été soumise à l'épreuve d'une expérience directe, et l'observation a parfaitement confirmé le calcul"}.
%
Arago later noted that Delisle had previously observed similar concentric bright and dark rings in a ball's shadow, see \cite{Delisle1715} and also \cite{Maraldi1723}.
%

Wave scattering depends heavily on wavelength, with diffraction causing shadow edges to break into fringes. Maraldi observed in \cite{Maraldi1723} that light circulates more easily around smaller objects: {\it "la lumiere plus grande au milieu des boules plus petites, fait voir qu'elle circule en plus grande abondance et plus facilement autour des petites boules qu'autour des grandes"} and \cite[Fig.8]{Maraldi1723} shows light at the center of a ball's shadow.

The study of wave scattering has a rich history, starting with Fresnel's work on diffraction and involving numerous researchers who contributed to understanding this phenomenon.
H\"ormander \cite{horFIO} made geometrical optics a branch of mathematics, providing powerful tools and clarifying relevant concepts: "a wealth of others ideas" from geometrical optics (referred to in the introduction to \cite{horFIO}) was later on exploited by Melrose, Taylor, Andersson, Eskin, Sj\"ostrand and Ivrii in dealing with propagation of singularities for mixed problems.
Of particular importance in recent works is the Melrose and Taylor parametrix for the diffractive Dirichlet problem which gives the form of the solution to \eqref{WE} near diffractive points (see also Zworski \cite{zw90}).

The complexity of diffracted waves, particularly near shadow boundaries, was highlighted by researchers like Keller, who developed a geometrical theory of diffraction. This theory explains how rays "creep" along obstacle boundaries following the shortest possible path. After Keller's work \cite{keller}, it had been conjectured that the decreasing rate of the intensity of light in the shadow region had to be $\exp(-C\tau^{1/3})$, where $\tau$ is the frequency and where the constant $C$ depends on the geometry of the geodesic flow on the boundary of the obstacle. Assuming an analytic boundary, this was proved by Lebeau \cite{le84} and then by Hargé and Lebeau \cite{hale94} for  $C^{\infty}$ boundary.

As we recalled earlier, while dispersion had remained unknown, global in time Strichartz estimates have been known to hold like in the flat case (in every dimension).  Indeed, on a manifold for which the boundary is everywhere strictly geodesically concave (no multiply-reflected rays, no gliding rays - such as the complement in $\mathbb{R}^d$ of a strictly convex obstacle), the Melrose and Taylor parametrix was used in \cite{smso95} to prove that Strichartz estimates for the wave equation do hold as in $\mathbb{R}^d$ (except for the endpoints). Later on, Ivanovici obtained in \cite{doi08} similar results for the classical Schr\"odinger flow with Dirichlet boundary condition, for which an additional difficulty, related to the infinite speed of propagation of the flow, had to be overcome:  first, sharp scale-invariant Strichartz for the semi-classical Schr\"odinger equation (e.g. on a time interval of size the wavelength for the classical equation) are proved on compact manifolds with strictly concave boundaries of dimension $d\geq 2$ (an example of which is provided by the so-called Sina\"i billiard, e.g. a punctured torus), then combined with local smoothing estimates as in \cite{bgt04}. Therefore, to obtain Strichartz in semi-classical time in \cite{doi08} we bypassed dispersion: we side-stepped this issue by taking advantage, as in \cite{smso95}, of the $L^2$ continuity of certain operators to reduce consideration to operators like those on a manifold without boundary: we stress out that this approach is very unlikely to work when one is interested in obtaining dispersion. 
\\ 

The paper is organized as follows: in section \ref{secgeneral} we construct a parametrix for the wave equation in the high-frequency case and when the source point is not too close to the boundary, to be used in the proofs of Theorems \ref{thmCE} and \ref{thmdisp3D}. When the observation point is located near a boundary point through which passes a glancing ray, the Melrose and Taylor parametrix applies directly; when the observation point is far, 
we use Kirchhoff's integral representation formula from section \ref{secgeneral}. In the third section we prove Theorem \ref{thmCE} in the case of the wave equation, by performing explicit computations at the Poisson-Arago spot using our parametrix - that involves Kirchhoff's integral representation formula - obtained in section \ref{secgeneral}, first when the obstacle is a ball ($d_2=0$) and then in the case of a cylindrical obstacle in $\R^{d_1+d_2}$.
The fourth section is devoted to the proof of Theorem \ref{thmdisp3D} for the wave equation. This proof is split into several  parts involving very different arguments : sections \ref{secffb} and \ref{secofoc} deal with the high frequency regime, when the source point is not too close to the boundary, and is based entirely on arguments of section \ref{secgeneral}. When both the source and the observation points are close to the boundary, a parametrix is obtained in section \ref{secctb} using spherical harmonics. In section \ref{secHelm} we deal with the low frequency regime, using classical results on the exterior Dirichlet problem for the Helmholtz equation. 
In the fifth section we explain how one can derive sharp dispersion bounds for the Schrödinger flow using the Kanaï transform and taking advantage of the previous results for waves. The Appendix contains some useful properties of the Airy, Bessel and Hankel functions. 

\subsection*{Acknowledgments} The author would like to thank Gilles Lebeau for helpful and constructive discussions on this problem between 2015 and 2018, when a first sketch of the proof of Theorem \ref{thmCE} (involving the exterior of a ball) has been announced and published as a collaboration in the CRAS note \cite{ildispext}, work on which this paper is based. She would also like to thank Fabrice Planchon for discussions on the generalization of the Poisson counterexample of \cite{ildispext} to a larger class of obstacles, including the cylindrical ones as in Theorem \ref{thmCE}.

\section{Construction of a global parametrix for the wave equation in $\Omega_{d_1,d_2}$ in the high-frequency case, when dist$(Q_0,\del \Omega_{d_1,d_2})$, dist$(Q_0,\del \Omega_{d_1,d_2})>c>0$}\label{secgeneral}

We start with the general form of a parametrix for the wave flow inside $\Omega_{d_1,d_2}=(\mathbb{R}^{d_1}\setminus B_{d_1}(0,1))\times \mathbb{R}^{d_2}$ for any $d_1\geq 3$, $d_2\geq 0$ : this construction will be particularly useful in Section \ref{secdispext3D}, in order to prove Theorem \ref{thmdisp3D} (for $d=d_1=3$) when both the source and the observation points stay away from a fixed, small neighborhood of $\partial\Omega_{3,0}$, as well as in Section \ref{secCE}, in order to construct counterexamples at the Poisson spot when $d_1\geq 4$, for both the exterior of a ball $d_2=0$ and the exterior of a cylinder $d_2\geq 1$. 

Consider the equation \eqref{WE} with initial data $(\delta_{Q_0},0)$, where $Q_0$ is an arbitrary point in $\Omega_{d_1,d_2}$
\begin{equation} \label{WEOC} 
\left\{ \begin{array}{l}
   (\partial^2_t-
 \Delta_{\Omega_{d_1,d_2}}) u=0  \;\; \text{ in } \Omega_{d_1,d_2}\times \mathbb{R}_t, \\ 
 u|_{t=0}=\delta_{Q_0}, \; \partial_t u|_{t=0}= 0, \quad u|_{\partial\Omega_{d_1,d_2}}=0. 
 \end{array} \right.
 \end{equation}
Here $\Delta_{\Omega_{d_1,d_2}}$ is the Laplace operator $\Delta$ from \eqref{laplaballrS2} with Dirichlet boundary condition on $\partial\Omega_{d_1,d_2}$.
By finite speed of propagation, for any time smaller than the distance (in $\mathbb{R}^d$) from $Q_0$ to the boundary of $\Omega_{d_1,d_2}$ (i.e. such as $0<t< d(Q_0,\partial\Omega_{d_1,d_2})$), the solution to \eqref{WEOC} in $\Omega_{d_1,d_2}$ equals the free wave in $\mathbb{R}^d$, denoted $u_{free}(Q,Q_0,t)$, given by
\begin{equation}\label{ufree}
u_{free}(Q,Q_0,t):=\frac{1}{(2\pi)^d}\int e^{i(Q-Q_0)\xi}\cos(t|\xi|)d\xi.
\end{equation}
Let 
$w_{in}(Q,Q_0,\tau):=\widehat{1_{t>0}u_{free}}(Q,Q_0,\tau)$ denote the Fourier transform in time of $u_{free}(Q,Q_0,t)|_{t>0}$, then
\begin{equation}\label{vfreeball}
w_{in}(Q,Q_0,\tau)=\tau^{\frac{d-1}{2}}\frac{e^{-i\tau|Q-Q_0|}}{|Q-Q_0|^{\frac{d-1}{2}}}\Sigma_d(\tau |Q-Q_0|),
\end{equation}
where $\Sigma_3=\frac{i}{4\pi} $ and for $d\geq 4$ and $V:=\tau|Q-Q_0|\gg 1$, 
\begin{equation}\label{vfreeform}
\Sigma_d(V)\sim_{V^{-1}} \sum_{j\geq 0}\Sigma_{j,d} V^{-j}, \quad \Sigma_{0,d}\neq 0.
\end{equation}
In \eqref{vfreeform} and throughout the paper, we use the following classical notion of asymptotic expansion: a function $f(w)$ admits an asymptotic expansion for $v\rightarrow 0$ when there exists a (unique) sequence $(c_{n})_{n}$ such that, for any $n$, $\lim_{v\rightarrow 0} v^{-(n+1)}(f(v)-\sum_{0}^{n} c_{j} v^{j})=c_{n+1}$. We denote $f(v)\sim_{v} \sum_{n} c_{n} v^{n}$.

We first chose our source point $Q_0$ in $ \Omega_{d_1,d_2}$.
Denote by $N$ and $S$ the North and the South poles of $B_{d_1}(0,1)$, respectively : a point $Q_N(r)$ belongs to $ON$ if, in the coordinates $(r,\varphi,\omega,z)$ of \eqref{omega}, it has $\varphi=0$, $z=0$. A point $Q_S(r)$ belongs to $OS$ if $\varphi=\pi$, $z=0$; in both cases $r$ denotes the distance from $Q_{N,S}(r)$ to $B_{d_1}(0,1)$. We let $s>1+c$ for some small constant $c>0$ and let $Q_{0}=:Q_N(s)\in ON$ be our source point throughout the paper (when we will search for counterexamples in Theorem \ref{thmCE}, we will consider the observation point to be $Q_S(s)\in OS$, also with vanishing $z$ coordinate when $d_2\geq 1$). 

We introduce the distance between two points in $\Omega_{d_1,d_2}$ in the $(r,\varphi,\omega,z)$ coordinates as follows : let $Q$ be an arbitrary point of $\Omega_{d_1,d_2}$ with $z_Q=z=(z_1,...,z_{d_2})\in \mathbb{R}^{d_2}$, then
\[
Q:=(r\sin\varphi \times \Pi_{j=1}^{d_1-2} \sin\omega_j,... ,r\sin\varphi \times( \Pi_{j=1}^{d_1-k}  \sin\omega_j)\times \cos \omega_{d_1-k},... , r\sin\varphi \times \cos \omega_1, r\cos \varphi, z).
\]
and the (Euclidean) distance between $Q$ and $Q_0=Q_N(s)$ reads as follows
\begin{equation}\label{tildePhi}
\tilde\phi(r,\varphi,z,s):=|Q-Q_0|=\sqrt{r^2-2sr\cos \varphi+s^2+|z|^2}.
\end{equation}
We change our coordinates  $(r,\varphi)$ into $(x,y)$ defined by $x=r-1$, $y=\frac{\pi}{2}-\varphi$, in order to make the boundary flat: as such, the boundary $\{r=1\}$ becomes $\{x=0\}$ and the domain $(\Omega_{d_1,d_2},\Delta)$ reads as 
\begin{equation}\label{omega2}
\Omega_{d_1,d_2}=\{x\geq 0, y\in [-\pi/2,\pi/2], \omega_j\in [0,\pi], \forall 1\leq j\leq d_1-3, \omega_{d_1-2}\in[0,2\pi),z\in \mathbb{R}^{d_2}\}
\end{equation}
\begin{equation}\label{laplacyl}
\Delta=\frac{\partial^2}{\partial x^2}+\frac{(d_1-1)}{(1+x)}\frac{\partial}{\partial x}+\frac{1}{(1+x)^2}\Big(\frac{\partial^2}{\partial y^2}+(d_1-2)\tan y\frac{\partial}{\partial y}+\frac{1}{(\cos y)^2}\Delta_{\mathbb{S}^{d_1-2}}\Big) +\Delta_{\mathbb{R}^{d_2}}.
\end{equation}
In the coordinates $(x,y)$, this distance \eqref{tildePhi} satisfies $\phi(x,y,z,s):=\tilde\phi(1+x,\frac{\pi}{2}-y,z,s)$, where we have 
\begin{equation}\label{Phi}
\phi(x,y,z,s)=\sqrt{(1+x)^2-2s(1+x)\sin y+s^2+|z|^2}.
\end{equation}
We use both types of coordinates : to describe the geometry it is more convenient to work with $(r, \varphi)$, but to provide an explicit form of a parametrix to the wave equation it is more useful to work with $(x,y)$.\\

For $Q\in \Omega_{d_1,d_2}$, we write $w_{in}$ from \eqref{vfreeball} using its Fourier transform, modulo $O(\tau^{-\infty})$, as follows
\begin{equation}\label{winform}
w_{in}(Q,Q_0,\tau)=\Big(\frac{\tau}{2\pi}\Big)^{1+d_2}\tau^{\frac{d-1}{2}}\int \frac{\Sigma_d(\tau\phi(x,\tilde y,\tilde z,s))}{\phi(x,\tilde y,\tilde z,s)^{\frac{d-1}{2}}}e^{i\tau((y-\tilde y)\alpha+(z-\tilde z)\gamma)}e^{-i\tau\phi(x,\tilde y,\tilde z,s)} d\alpha d\gamma d\tilde y d\tilde z.
\end{equation}
The forward free wave is given by
\begin{equation}\label{ufree}
u^+_{free}:=\frac{1}{2\pi}\int e^{i\tau t}w_{in}(Q,Q_0,\tau) d\tau.
\end{equation}
We denote by $u$ the solution to the Dirichlet wave equation in $\Omega_{d_1,d_2}$ whose incoming part (before reflection) equals $u^+_{free}$ from \eqref{ufree}. To construct it explicitly, we define its extension to $\mathbb{R}^d$ as follows
\begin{equation}\label{defUbar}
\underline{u}(Q,Q_0,t):=
\left\{ \begin{array}{l}
u(Q,Q_0,t), \text{ if } Q\in\Omega_{d_1,d_2},\\
0, \text{ if } Q\in \overline{B}_{d_1}(0,1)\times \R^{d_2}.
 \end{array} \right.
\end{equation}
By construction, $\underline{u}(Q,Q_0,t)$ vanishes inside $B_{d_1}(0,1)\times \R^{d_2}$.
Using the classical Duhamel formula and setting $u^+:=1_{t>0}u$, $\underline{u}$ reads as follows
 \begin{equation}\label{Gbarform}
 \underline{u}|_{t>0} =u^+_{free}-u^{\#},\quad \text{ where we set } u^{\#}(Q,Q_0,t): 
 =\square^{-1}\Big((\partial_xu^+)|_{\partial\Omega_{d_1,d_2}}\Big),
 \end{equation}
 where $\del_x u$ is the normal derivative of $u$ and where, for $F$ such that supp$(F)\subset \{t'\geq 0\}$ and with $\mathcal{F}^{-1}_{\xi}$ denoting the inverse Fourier transform in space, we have
\begin{equation}\label{FRform}
\square^{-1}F(t)=\int_{-\infty}^t \mathcal{F}^{-1}_{\xi}\Big(\frac{\sin (t|\xi|)}{|\xi|}\Big)(t-t')*F(t')dt'. 
\end{equation}

Fix $0\leq h_0<1$ small and let $h\in (0,h_0)$. Let $\chi\in C^{\infty}_0((\frac 12, 2))$ be a smooth cutoff equal to $1$ on $[\frac 34, \frac 32]$ and such that $0\leq \chi\leq 1$.
As we are interested in obtaining and evaluating $\chi(hD_t)\underline{u}(Q,Q_0,t)$, let
\begin{equation}\label{udiezh}
u^+_{free,h}:=\chi(hD_t)u^+_{free}, \quad u^{\#}_{h}:=\chi(hD_t)u^{\#}(Q,Q_0,t).
\end{equation}
As the free wave flow $u^+_{free,h}$ satisfies the usual dispersive estimates, we are reduced to study $u^{\#}_{h}(Q,Q_0,t)$ of \eqref{udiezh}, with $Q_0=Q_N(s)$. Using \eqref{Gbarform}, \eqref{FRform} and classical results, $u^{\#}_{h}(Q,Q_0,t)$ reads as 
\begin{prop}\label{propudiezh}
If $Q\in\Omega_{d_1,d_2}$ is such that $\tau\text{dist}(Q,\Omega_{d_1,d_2})\gg 1$, we have
\begin{multline}\label{vhform}
u^{\#}_{h}(Q,Q_0,t)=\chi(hD_t)u^{\#}(Q,Q_0,t)=\int e^{it\tau}\chi(h\tau)
 \int_{P\in\partial\Omega_{d_1,d_2}}\mathcal{F}(\partial_x u^+|_{\partial\Omega_{d_1,d_2}})(P,Q_0,\tau)\\
 \times \Big(\frac{\tau^{\frac{d-3}{2}}}{|Q-P|^{\frac{d-1}{2}}}\Sigma_d(\tau|Q-P|)e^{-i\tau|Q-P|}
+O((\tau |Q-P|)^{-\infty})\Big)d\sigma(P) d\tau,
\end{multline}
where $\Sigma_3:=\frac{i}{4\pi}$ and for $d=d_1+d_2\geq 4$, $\Sigma_d$ is given in \eqref{vfreeform} and where $\mathcal{F}(\partial_x u^+|_{\partial\Omega_{d_1,d_2}})(P,Q_0,\tau)$ denotes the Fourier transform in time of $\partial_x u^+|_{\partial\Omega_{d_1,d_2}}(P,Q_0,t)$ for $P\in\partial\Omega_{d_1,d_2}$.
\end{prop}
\begin{rmq}
Recall that when $d_2\geq 1$, $\partial\Omega_{d_1,d_2}=\mathbb{S}^{d_1-1}\times \mathbb{R}^{d_2}$, while when $d_2=0$, $\partial\Omega_{d_1,d_2}=\mathbb{S}^{d-1}$. Also, when $d\geq 4$, $\Sigma_d$ is given in \eqref{vfreeform} and as $|Q-P|\geq r-1>c$, it follows that $O((\tau |Q-P|)^{-\infty})=O(h^{\infty})$.
\end{rmq}

It follows from Proposition \ref{propudiezh} that, in order to evaluate the $L^{\infty}$ norm of $u^{\#}_{h}$, we are reduced to compute $\mathcal{F}(\partial_{x}u^+|_{\mathbb{S}^{d_1-1}\times \mathbb{R}^{d_2}})$. In order to do that we have to obtain $u^{+}(Q,Q_0,t)$ for $Q$ near $\mathbb{S}^{d_1-1}\times \mathbb{R}^{d_2}$. As rays sent from $Q_0$ in the direction of the obstacle may be either transverse or tangent to its boundary, to construct the reflected wave we separate two regimes (transverse and glancing), where the behaviour of the wave will be different. To do that, we first define the apparent contour from $Q_0$.
\begin{dfn}
For a source point $Q_0=Q_N(s)$, which has $(r,\varphi, \omega,z)$-coordinates $(s,0,\omega,0_{\mathbb{R}^{d_2}})$, we define its apparent contour $\mathcal{C}^{d_1,d_2}_{Q_0}$ as  the set of points $P\in\mathbb{S}^{d_1-1}\times \mathbb{R}^{d_2}$ such that the ray $Q_0P$ is tangent to $\mathbb{S}^{d_1-1}\times \mathbb{R}^{d_2}$: in other words, with $\tilde\phi$ defined in \eqref{Phi}, we have
\[
\mathcal{C}^{d_1,d_2}_{Q_0}:=\{P\in\mathbb{S}^{d_1-1}\times \mathbb{R}^{d_2} \text{ with coordinates } (1,\theta,z) \text{ such that } \partial_r \tilde\phi(1,\varphi,z,s)=0\}.
\]
As $\partial_r\tilde\phi=(r-s\cos\varphi)/\tilde\phi$ and $\del_r\tilde \phi|_{r=1}=0$ cancels when $\cos\varphi=\frac 1s$, we find 
\[
\mathcal{C}_{Q_0}^{d_1,d_2\geq 1}:=\{ P=(1,\arccos (1/s),\omega,z), z\in \mathbb{R}^{d_2}\}, \quad \mathcal{C}_{Q_0}^{d_1,0}:=\{ P=(1,\arccos (1/s), \omega)\}.
\]
In the coordinates $(x,y,\omega, z)$ we therefore have $\mathcal{C}^{d_1,d_2}_{Q_0}=\{P=(0,\arcsin(1/s),\omega,z)\}$.
We define
\[
\varphi_*:=\arccos(1/s)=\frac{\pi}{2}-\arcsin(1/s), \quad y_*(s):=\arcsin(1/s).
\]
\end{dfn}

We need to bound $u^{\#}_{h}(Q,Q_0,t)$ defined by \eqref{vhform}. For a point $P$ on the boundary with coordinates $(0,y_P,\omega_P,z_P)$, the integrand in \eqref{vhform} behaves differently if $y_P$ is outside a small neighborhood of $y_*(s)$ or if it is near $y_*(s)$. We separate the integration contour in  two parts, according to these two cases.

\begin{dfn}\label{dfnkappa}
Fix $0<\e\leq 1$ and let $\kappa_{\e}\in C^{\infty}_0((-\e,\e))$ be a smooth function equal to $1$ on $[-\e/2, \e/2]$ and such that $0\leq\kappa_{\e}\leq 1$ everywhere. This cut-off will be often used throughout the paper with different values of $\epsilon$ (that will be, each time, a small positive number independent of the other parameters).
\end{dfn}
Let $\kappa_{\e_0}$ like in Definition \ref{dfnkappa} for some fixed, sufficiently small constant $0<\e_0<1/2$ independent of $h$, and split the integral over $P$ in \eqref{vhform} in two parts using $1=\kappa_{\e_0}(y-y_*(s))+(1-\kappa_{\e_0}(y-y_*(s)))$. 
For a point $P$ on the boundary with coordinates $(0,y_P,\omega_P,z_P)$ and for $\kappa\in \{\kappa_{\e_0},1-\kappa_{\e_0}\}$, we let 
\begin{multline}\label{vhformnewchi}
u^{\#}_{h,\kappa}(Q,Q_0,t):=\int e^{it\tau}\chi(h\tau)
 \int_{P\in\mathbb{S}^{d_1-1}\times \mathbb{R}^{d_2}}\mathcal{F}(\partial_x u^+|_{\mathbb{S}^{d_1-1}\times\mathbb{R}^{d_2}})(P,Q_0,\tau)\kappa(y_P-y_*(s))\\
 \times \Big(\frac{\tau^{\frac{d-3}{2}}}{|Q-P|^{\frac{d-1}{2}}}\Sigma_d(\tau|Q-P|)e^{-i\tau|Q-P|}
+O((\tau |Q-P|)^{-\infty})\Big)d\sigma(P) d\tau,
\end{multline}
then $u^{\#}_{h,\kappa_{\e_0}}+u^{\#}_{h,1-\kappa_{\e_0}}=u^{\#}_{h}$. In the next section we deal with $u^{\#}_{h,1-\kappa_{\e_0}}$, the "transverse" part.

\subsection{The transverse part of $u^+$}
We show that $u^{\#}_{h,1-\kappa_{\e_0}}$ satisfies the usual dispersive bounds.
\begin{lemma}\label{lemtrans}
Let $Q\in \Omega_{d_1,d_2}$ 
then there exists $C=C_{\e_0}$ such that
\begin{equation}\label{sumtransv}
|u^{\#}_{h,1-\kappa_{\e_0}}(Q,Q_0,t)|
\lesssim \frac{C_{\e_0}}{h^d}\Big(\frac{h}{t}\Big)^{(d-1)/2}.
\end{equation}
\end{lemma}
\begin{proof}
As in the integral defining $u^{\#}_{h,1-\kappa_{\e_0}}(Q,Q_0,t)$ the boundary variable $P$ stays away from a small, fixed neighborhood of $\mathcal{C}^{d_1,d_2}_{Q_0}$ (that is, from the set $\{y=y_*(s)=\arcsin (1/s)\}$), it means that the ray $Q_0P$ is transverse to the boundary. From classical results, for such $P$, the phase function of $\mathcal{F}(\partial_x u^+|_{\mathbb{S}^{d_1-1}\times \mathbb{R}^{d_2}})(P,Q_0,\tau)$ equals $-i\tau |P-Q_0|$ and therefore the phase function of \eqref{sumtransv} equals $i\tau(t-|P-Q|-|P-Q_0|)$. The critical points of $P\rightarrow |P-Q|+|P-Q_0|$ satisfy, for some $\lambda\in \mathbb{R}\setminus \{0\}$,
$ \Big(\frac{P-Q}{\vert P-Q\vert}+\frac{P-Q_0}{\vert P-Q_0\vert}\Big)=\lambda \vec{\nu}_P$
with $\vec{\nu}_P$ the unit normal to $\partial\Omega_{d_1,d_2}$ pointing towards $\Omega_{d_1,d_2}$. 
Let $\Omega^+_{Q_0}\subset \Omega_{d_1,d_2}$ be the open set of points $\tilde Q\in\Omega_{d_1,d_2}$ such that the segment $[Q_0,\tilde Q]$ is contained in $\Omega_{d_1,d_2}$. 
For an observation point $Q\notin \Omega^+_{Q_0}$, the restriction of this phase to $\partial\Omega^+_{Q_0}:=\Omega^+_{Q_0}\cap \partial\Omega_{d_1,d_2}$ has a unique critical point at $P(Q,Q_0)\in \partial\Omega^+_{Q_0}$ which is the intersection of the segment $[Q_0,Q]$ with $\partial\Omega^+_{Q_0}$.  As both $Q_0$ and $Q$ have $z=0\in \mathbb{R}^{d_2}$ coordinate, the same will hold for $P(Q,Q_0)$.
 As $P(Q,Q_0)$ stays away from a small, fixed neighborhood of $\mathcal{C}_{Q_0}$ (due to the support of $1-\kappa_\e$), one has a lower bound on $\nabla^2_P(|Q-P|+|P-Q_0|)$ which is moreover uniform with respect to $Q,Q_0$. 
The stationary phase applies to $u^{\#}_{h,1-\kappa_{\e_0}}(Q,Q_0,t)$ and yields usual dispersive bounds of $\mathbb{R}^d$.
\end{proof}

We are left with $u^{\#}_{h,\kappa_{\e_0}}$, the most delicate part. For $P$ near $\mathcal{C}^{d_1,d_2}_{Q_0}$, i.e. on the support of $\kappa_{\e_0}(y_P-y_*(s))$, we must construct $\mathcal{F}(\partial_x u^+|_{\mathbb{S}^{d_1-1}\times\mathbb{R}^{d_2}})(P,Q_0,\tau)$. This is the goal of the remaining of this section.

\subsection{The glancing part of $u^{+}$} \label{sectgl} This part follows closely the construction of \cite{IaIv23} (where $d_1=2$, $d_2=1$).
We obtain $u^+(\cdot,Q_0,t)$ near the boundary using the Melrose and Taylor parametrix and then we compute the trace on the boundary of its normal derivative $\partial_{x}u^+$ in order to obtain the "glancing part" of $u^{\#}$ from formula \eqref{Gbarform}. The next result applies near rays which are tangent to $\partial\Omega_{d_1,d_2}$.
\begin{prop}\label{MT}\cite{zw90}
Microlocally near a glancing point of second order contact with the boundary there exist smooth phase functions $\iota(x,y,z,\alpha,\gamma)$ and $\zeta(x,y,z,\alpha,\gamma)$ such that $\iota\pm\frac 23 (-\zeta)^{3/2}$ satisfy the eikonal equation and there exist symbols $a,\,b$ satisfying appropriate transport equations such that, for any parameters $\alpha,\gamma$ in a conic neighborhood of a glancing direction and for $\tau\gg 1$ large enough,
\begin{equation}\label{Gtau}
G_{\tau}(x,y,z,\alpha,\gamma):=e^{\ii\tau\iota(x,y,z,\alpha,\gamma)}\Big(aA_+(\tau^{2/3}\zeta)+b\tau^{-1/3}A'_+(\tau^{2/3}\zeta)\Big)A^{-1}_+(\tau^{2/3}\zeta_0)
\end{equation} 
satisfies 
$(\tau^2+\Delta)G_{\tau}=e^{\ii\tau\iota(x,y,z,\alpha,\gamma)}\Big(a_{\infty}A_+(\tau^{2/3}\zeta)+b_{\infty}\tau^{-1/3}A'_+(\tau^{2/3}\zeta)\Big)A^{-1}_+(\tau^{2/3}\zeta_0)$,
where the symbols verify $a_{\infty}$, $b_{\infty}\in O(\tau^{-\infty})$ and where we set $\zeta_0=\zeta|_{x=0}$. Moreover, the following properties hold:
\begin{itemize}
\item $\iota$ and $\zeta$ are homogeneous of degree $0$ and $-1/3$ and satisfy $\langle d\iota,d\iota\rangle-\zeta\langle d\zeta,d\zeta\rangle=1$,
$\langle d\iota,d\zeta\rangle=0$,
where $\langle\cdot,\cdot\rangle$ is the polarization of $p$; the phase $\zeta_0$ is independent of $y,z$ so that $\zeta_0(\alpha,\gamma)$ vanishes at a glancing direction; the diffractive condition means that $\partial_x\zeta|_{x=0}<0$ near a glancing point;
\item the symbols $a(x,y,z,\alpha,\gamma)$ and $b(x,y,z,\alpha,\gamma)$ belong to the class $\mathcal{S}^0_{(1,0)}$  and satisfy the appropriate transport equations. Moreover $a|_{x=0}$ is elliptic at the glancing point with essential support included in a small, conic neighborhood of it, while $b|_{x=0}=0$.
\end{itemize}
\end{prop}
In the following we will explicitly compute the phase functions and the main contributions of the symbols. Notice that, as we consider here the exterior of a ball or of a cylinder (i.e. of a model convex obstacle), we can obtain the explicit form of the (glancing part of the) outgoing wave directly, using the eikonal and the transport equations below (so without using the Melrose and Taylor parametrix). However, the fact that the symbol $b$ vanishes on the boundary it is not obvious and it is particularly useful. As Proposition \ref{MT} ensures this property, it is very convenient to make use of it.\\

\paragraph{\bf The eikonal equation} The functions $\iota$ and $\zeta$ from Proposition \ref{MT} solve the system of equations
\begin{equation}\label{MT-para}
\left\{\begin{aligned}
&(\partial_x\iota)^2+\frac{(\partial_y\iota)^2}{(1+x)^2}+|\nabla_z\iota|^2-\zeta\Big((\partial_x\zeta)^2+\frac{(\partial_y\zeta)^2}{(1+x)^2}+|\nabla_z\zeta|^2\Big)=1,\\
&\partial_x\iota \partial_x\zeta+\frac{\partial_y\iota \partial_y\zeta}{(1+x)^2}+<\nabla_z\iota, \nabla_z\zeta>=0.
\end{aligned}\right.
\end{equation}
The system \eqref{MT-para} admits the pair of solutions 
\begin{equation}\label{soleikball}
\iota(y,z,\alpha,\gamma)=y\alpha+<z,\gamma>,   \quad \zeta(x,\alpha,\gamma)=\alpha^{2/3}\tilde\zeta((1+x)\sqrt{1-|\gamma|^2}/\alpha),
\end{equation}
where one should think of $\alpha, \gamma$ as the dual variables of $y,z$ and
where for $\rho:=(1+x)\frac{\sqrt{1-|\gamma|^2}}{\alpha}$, $\tilde \zeta$ is the (unique) solution to
$\frac{1}{\rho^2}-\tilde{\zeta}(\rho)[\tilde{\zeta}'(\rho)]^2=1$, $\tilde\zeta(1)=0$.

\begin{lemma}\label{lemzeta}
The equation $-\tilde{\zeta}(\partial_\rho\tilde{\zeta})^2+1/\rho^2=1$, $\tilde{\zeta}(1)=0$ has a unique solution of the form
\begin{equation}\label{tildezet>}
\frac23 (-\tilde{\zeta}(\rho))^{3/2}=\int_{1}^{\rho}\frac{\sqrt{w^2-1}}{w}dw=\sqrt{\rho^2-1}-\arccos\left(\frac{1}{\rho}\right),
\end{equation}
if $\rho>1$, while for $\rho<1$ we have 
\begin{equation}\label{tildezet<}
\frac23\tilde{\zeta}(\rho)^{3/2}=\int_{\rho}^1\frac{\sqrt{1-w^2}}{w}dw=\log [(1+\sqrt{1-\rho^2})/\rho]-\sqrt{1-\rho^2}.
\end{equation}
We note that at $\rho=1$ we have $\tilde{\zeta}=0$ and $\lim_{\rho\rightarrow1} \frac{(-\tilde \zeta)(\rho)}{\rho-1}=2^{1/3}$.
\end{lemma} 
\begin{rmq}
When $d_2=0$, there is no $z$ variable, and in this case we have simpler formulas
\begin{equation}\label{soleikballd}
\iota(y,z,\alpha,\gamma)=y\alpha,   \quad \zeta(x,\alpha)=\alpha^{2/3}\tilde\zeta((1+x)/\alpha),
\end{equation}
where $\rho:=(1+x)\frac{1}{\alpha}$ and $\tilde \zeta$ is like in Lemma \ref{lemzeta}.
\end{rmq}

\paragraph{\bf The transport equations } Let $\Delta$ as in \eqref{laplacyl}. We look for symbols $a,b$ such that $(\tau^2+\Delta)G_{\tau}\in O_{C^{\infty}}(\tau^{-\infty})$, with $G_{\tau} $ defined in \eqref{Gtau} and with $\iota$ and $\zeta$ given in Proposition \ref{MT}. We introduce the following notation
\[
v(a,b;\tau):=e^{i\tau \iota}\Big(aA_+(\tau^{2/3}\zeta)+b\tau^{-1/3}A'_+(\tau^{2/3}\zeta)\Big).
\]
\begin{lemma}\label{lemeqtransp}
We have $
(\tau^2+\Delta)v(a,b;\tau)=v(A,B;\tau)$
,  with $
\left(
\begin{array}{c}
A   \\
B   \\
\end{array}
\right)=(\tau\mathcal{M}+\Delta)
\left(
\begin{array}{c}
a   \\
b   \\
\end{array}
\right)
$
\hskip 3mm  and
\[
\mathcal{M}=
\left(
\begin{array}{cc}
 \frac{i\alpha}{(1+x)^2}(2\partial_y+(d_1-2)\tan y) +2i<\gamma, \nabla_z>&  -\Big(\zeta \partial_x\zeta(2\partial_x+\frac{d_1-1}{1+x})+(\partial_x\zeta)^2+\zeta \partial^2_x\zeta\Big)    \\
2\partial_x\zeta\partial_x+\partial^2_x\zeta+\frac{d_1-1}{1+x}\partial_x\zeta  &   \frac{i\alpha}{(1+x)^2}(2\partial_y+(d_1-2)\tan y) +2i<\gamma, \nabla_z> \\  
\end{array}
\right)
\]
\end{lemma}
The lemma follows by explicit computations. When $d_2=0$, take $\gamma=0$ everywhere in $\mathcal{M}$.\\

We set $a\sim_{\tau^{-1}} \sum_{k\geq 0} a_{k}(i\tau)^{-k}$, $b\sim_{\tau^{-1}} \sum_{k\geq 0} b_{k}(i\tau)^{-k}$. Using Lemma \ref{lemeqtransp} we obtain that $a_k,b_k$ satisfy the following transport equations  
\begin{equation}\label{treq}
\mathcal{M}\left(
\begin{array}{c}
a_{k+1}   \\
b_{k+1}   \\
\end{array}
\right)+\Delta \left(
\begin{array}{c}
a_{k}   \\
b_{k}   \\
\end{array}
\right) =0, \quad \mathcal{M} \left(
\begin{array}{c}
a_{0}   \\
b_{0}   \\
\end{array}
\right)=0.
\end{equation}
From Proposition \ref{MT} we know that we may find $(a,b)$ such that, $a|_{x=0}$ is elliptic and $b|_{x=0}=0$ (the fact that $b_0,b_1$ can be taken to vanish on the boundary follows easily from direct computations, but this is less obvious for all $b_k$). The main contribution $a_0$ of $a$ is obtained below.

\begin{lemma}\label{proptransp}
The second equation in \eqref{treq}, $ \mathcal{M} \left(
\begin{array}{c}
a_{0}   \\
b_{0}   \\
\end{array}
\right)=0$, has a solution of the form 
\[
a_{0}=(\cos y)^{(d_1-2)/2} \xi((1+x)/\alpha), \quad b_{0}=0,
\]
where, for $\tilde\zeta$ obtained in Lemma \ref{lemzeta}, the function $\xi $ satisfies 
$2\partial_{\rho}\tilde \zeta\partial_{\rho}\xi+\partial^2_{\rho}\tilde\zeta+\frac{d_1-1}{\rho} (\partial_{\rho}\tilde\zeta)  \xi=0$, $\xi(1)=1$. In particular, $\xi$ is elliptic for $(1+x)/\alpha$ near $1$.
\end{lemma}
Our goal in this section is to describe the solution $u^+$ to the wave equation with Dirichlet boundary condition, whose incoming part is $u^+_{free}$, microlocally near a bicharacteristic tangent to the boundary. In the next lemma, which follows from Proposition \ref{MT}, we introduce an operator which provides the form of the outgoing wave ; as its trace on the boundary is an elliptic FIO, it will allow to recover the outgoing solution with incoming part $u^+_{free}$. Recall first some notations :
\begin{dfn}\label{dfndistrib}
For an open set $U\subset\R^n$, we denote by $\mathcal{D}(U)$ the space of test functions $C^{\infty}_0(U)$ (smooth functions with compact support) and by $\mathcal{D}'(U)$ the space of distributions on $U$ (the topological dual of $\mathcal{D}(U)$). We denote by $\mathcal{E}(U)$ the (Fréchet) topological vector space $C^{\infty}(U)$ of smooth functions with the family of semi-norms $C^{\infty}(U)\ni g \rightarrow \sup_{K\subset U, K \text{compact}} |\partial^{j} g |\in [0,\infty)$ and by $\mathcal{E}'(U):=(\mathcal{E}(U))^*$ its dual, which is the space of distributions with compact support.
\end{dfn}

\begin{lemma}\label{corMtau} 
Define an operator $M_{\tau}: \mathcal{E}'(\mathbb{R}^{1+d_2})\rightarrow  \mathcal{D}'(\mathbb{R}^{2+d_2})$, with $\mathcal{E}'(\mathbb{R}^{1+d_2})$ as in Definition \ref{dfndistrib}
\begin{equation*}
M_{\tau}(f)(x,y,z):=\left(\frac{\tau}{2\pi}\right)^{1+d_2}\int G_{\tau}(x,y,z,\alpha,\gamma)\hat{f}(\tau\alpha,\tau\gamma)d\alpha d\gamma.
\end{equation*}
Then, near the glancing region, we have $(\tau^2+\Delta)M_{\tau}(f)\in O(\tau^{-\infty})
$ (up to the boundary) for all $f\in \mathcal{E}'(\mathbb{R}^{1+d_2})$.
Moreover, the restriction to the boundary $M_{\tau}(f)|_{\partial\Omega_{d_1,d_2}}=:J_{\tau}(f)$ defined by
\begin{equation*}
J_{\tau}(f)(y,z)=\left(\frac{\tau}{2\pi}\right)^{1+d_2}\int e^{\ii\tau(\iota(y,z,\alpha,\gamma)-\tilde{y}\alpha-\tilde{z}\gamma)}a(0,y,z,\alpha,\gamma,\tau)f(\tilde{y},\tilde{z})d\alpha d\gamma d\tilde{y} d\tilde{z},
\end{equation*}
has a microlocal inverse $J^{-1}_{\tau}$ as $a(x,y,z,\alpha,\gamma,\tau)$ is the elliptic symbol from Proposition \ref{MT}.
\end{lemma}
\begin{proof}
The first statement is obvious as $G_{\tau}$, defined in \eqref{Gtau}, satisfies $(\tau^2+\Delta)G_{\tau}\in O(\tau^{-\infty})$. Moreover, as $b$ can be chosen to vanish on the boundary, it follows that 
\[
G_{\tau}|_{x=0}=e^{i\tau\iota(0,y,z,\alpha,\gamma)} a(0,y,z,\alpha,\gamma).
\]
As $a|_{x=0}$ is elliptic, then so is $G_{\tau}|_{x=0}$. Hence, the restriction to the boundary of the operator $M_{\tau}$, denoted $J_{\tau}$, is an elliptic Fourier integral operator. 
\end{proof}

The operator $M_{\tau}\circ J^{-1}_{\tau}$ applied to the restriction to the boundary of the free wave yields the form of the outgoing wave near the glancing region that matches the boundary trace $w_{in}|_{\partial\Omega_{d_1,d_2}}$ of the incoming wave (at least as long as $|z|\ll s$, which means that one considers directions within a cone around the $ON$ axis) :

\begin{prop}\label{struttura}
For all $Q=(x,y,\omega,z)$ near the glancing region with $|z|\ll s$, i.e. with $x$ near $0$ and $y$ on the support of $\kappa_{\e_0}(y-y_*(s))$ for some small $\e_0>0$, $u^+(Q,Q_0,t)$ is independent of $\omega$ and we have 
\begin{equation}\label{uplusjgl}
u^+(Q,Q_0,t)=\frac{1}{(2\pi)^2}\int e^{\ii t\tau}\big(w_{in}(x,y,z,\tau)-M_{\tau}(J_{\tau}^{-1}(w_{in}|_{\partial\Omega_{d_1,d_2}})(x,y,z))\big)d\tau.
\end{equation}

\end{prop}
As mentioned above, we wish to determine the form of $\mathcal{F}(\partial_{x} u^+)(P,Q_0,\tau)$ for $P$ near $\mathcal{C}^{d_1,d_2}_{Q_0}$ which will be inserted in formula \eqref{vhform} in order to obtain dispersive bounds.
To do that, we need to compute $M_{\tau}(J_{\tau}^{-1}(w_{in}|_{\partial\Omega_{d_1,d_2}}))$. For that we write  $w_{in}$, that satisfies $(\tau^2+\Delta)w_{in}\in O(\tau^{-\infty})$ near the glancing region, as an oscillating integral with phase functions $\iota,\zeta$ and symbols depending on $a,b$ from Proposition \ref{MT}. We introduce an operator  $T_{\tau}: \mathcal{E}'(\mathbb{R}^{1+d_2})\rightarrow  \mathcal{D}'(\mathbb{R}^{2+d_2})$ for $F\in \mathcal{E}'(\mathbb{R}^{1+d_2})$
\begin{equation}\label{T-tau}
T_{\tau}(F)(x,y,z)=\left(\frac{\tau}{2\pi}\right)^{1+d_2}\int{e^{\ii\tau(y\alpha+<z,\gamma>)}}\big(aA(\tau^{2/3}\zeta)+ b\tau^{-1/3}A'(\tau^{2/3}\zeta)\big) \hat{F}(\tau\alpha,\tau\gamma)d\alpha d\gamma.
\end{equation}
According to \cite[Lemma A.2]{SmSo94Duke}, $T_{\tau}$ is an elliptic FIO near the glancing regime and $(\tau^2+\Delta)T_{\tau}(F) \in O(\tau^{-\infty})$. Moreover, for every solution to $(\tau^2+\Delta)w\in O_{C^{\infty}}(\tau^{-\infty})$ there exists a unique $F$ such that, microlocally near a glancing point, $w=T_{\tau}F+O(\tau^{-\infty})$. Applying  \cite[Lemma A.2]{SmSo94Duke} to $w_{in}$ gives
\begin{lemma}\label{lemF}
There exists a unique $F_{\tau}\in \mathcal{E}'(\mathbb{R}^{1+d_2})$, essentially supported near $y_*(s)$, such that,  for $Q=(x,y,\omega,z)$ near $\mathcal{C}^{d_1,d_2}_{Q_0}$ (and $|z|\ll s$ if $d_2>0$), $w_{in}(Q,Q_0,\tau)$ may be written as follows modulo $O(\tau^{-\infty})$ terms
\begin{equation}\label{windfn}
w_{in}(Q,Q_0,\tau)=\left(\frac{\tau}{2\pi}\right)^{1+d_2}\int  e^{i\tau(y\alpha+<z,\gamma>)} \Big(aA(\tau^{2/3}\zeta)+b\tau^{-1/3}A'(\tau^{2/3}\zeta)\Big)\hat{F}_{\tau}(\tau\alpha,\tau\gamma) d\alpha d\gamma,
\end{equation}
where $\zeta$ is given in \eqref{soleikball}, and $a,b$ are as in Proposition \ref{MT} with $a$ elliptic with main contribution $a_0$ given in Lemma \ref{proptransp} and with $b|_{x=0}=0$. When $d_2=0$, there are no $z,\gamma$ variables in \eqref{windfn}. 
\end{lemma}
\begin{rmq}\label{rmq|z|}
The condition on the size of $|z|$ in Lemma \ref{lemF} is necessary in order to apply the stationary phase with respect to $\gamma$ and corresponds to $|\gamma|\ll 1$, i.e. to initial directions in a cone around the $ON$ axis. 
Under the assumptions of Theorem \ref{thmCE}, this condition simply holds by finite speed of propagation.
Notice that, for $|\gamma|$ close to $1$, which corresponds to initial directions almost parallel to the $Oz$ axis, a parametrix for the wave equation has been given in \cite{IaIv23} for $d_2=1$ in terms of a spectral sum in cylindrical coordinates.

\end{rmq}

The precise form of $F_{\tau}$ is given in Lemma \ref{lemmaF} below. Let us first state a corollary of Lemma \ref{lemF} which provides the form of $\mathcal{F}(\partial_{x}u^+|_{\mathbb{S}^{d_1-1}\times \mathbb{R}^{d_2}})$ in terms of $\hat{F}_{\tau}$, that will allow to obtain the parametrix from \eqref{vhform}.
\begin{prop}\label{struttura}
Let $F_{\tau}\in \mathcal{E}'(\mathbb{R}^{1+d_2})$ as in Lemma \ref{lemF} satisfying $w_{in}(\cdot,\tau)=T_{\tau}(F_{\tau})+O(\tau^{-\infty})$ for $(x,y)$ near $(0,y_*(s))$. Then for all $Q=(x,y,\omega,z)$ near the glancing region with $|z|\ll s$, we have

\begin{multline}\label{MtauJ-1tau}
M_{\tau}(J_{\tau}^{-1}(w_{{in}}|_{\partial\Omega_{d_1,d_2}}))(x,y,z)=\left(\frac{\tau}{2\pi}\right)^{1+d_2}\int e^{\ii\tau(y\alpha+<z,\gamma>)}
\left(a A_{+}(\tau^{2/3}\zeta)+b\tau^{-1/3}A'_{+}(\tau^{2/3}\zeta)\right) \\
\times \frac{A(\tau^{2/3}\zeta_0)}{A_{+}(\tau^{2/3}\zeta_0)}\hat{F}_{\tau}(\tau\alpha,\tau\gamma)d\alpha d\gamma +O(\tau^{-\infty}).
\end{multline}
\end{prop}
\begin{proof}
The proposition follows immediately from Lemma \ref{lemF} as, replacing $w_{in}|_{\partial\Omega_{d_1,d_2}}$ by $T_{\tau}(F_{\tau})|_{\partial\Omega_{d_1,d_2}}$, the symbol $b$ vanishes and yields $T_{\tau}|_{\partial\Omega_{d_1,d_2}}=J_{\tau}\circ Q_{\tau}$ modulo $O(\tau^{-\infty})$ terms, where we set
\[
\hat{Q}_{\tau}(F):=A(\tau^{2/3}\zeta_0(\alpha,\gamma))\hat{F}_{\tau}(\tau\alpha,\tau\gamma), \text{ with } \zeta_0(\alpha,\gamma)=\zeta(0,\alpha,\gamma).
\] 
Then $J^{-1}_{\tau}\circ T_{\tau}(F_{\tau})|_{\partial\Omega_{d_1,d_2}}=Q_{\tau}(F_{\tau})+O(\tau^{-\infty})$.
Replacing in $M_{\tau}(J_{\tau}^{-1}(w_{in}|_{\partial\Omega_{d_1,d_2}}))$ gives \eqref{MtauJ-1tau}.
\end{proof}
We can now obtain the form of $u^+$ in \eqref{uplusjgl} near $\mathcal{C}^{d_1,d_2}_{Q_0}$ using \eqref{windfn} to express $w_{in}$ and \eqref{MtauJ-1tau} to express $M_{\tau}(J_{\tau}^{-1}(w_{in}|_{\partial\Omega_{d_1,d_2}}))$. Taking (the trace of) the normal derivative of $u^+$ yields $\mathcal{F}(\partial_x u^+)|_{\del\Omega_{d_1,d_2}}$.

\begin{prop}\label{propw0gl}
For $P=(0,y,\omega,z)\in\mathbb{S}^{d_1-1}\times \mathbb{R}^{d_2}$ near $\mathcal{C}^{d_1,d_2}_{Q_0}$ and $|z|\ll s$ when $d_2>0$, the Fourier transform in time of $\partial_{x}u^+$, denoted $\mathcal{F}(\partial_x u^+)(P,Q_0,\tau)$, is independent of $\omega$ and  reads as follows
\begin{equation}\label{u_xomega}
\mathcal{F}(\partial_x u^+)(y,z,Q_0,\tau)=\left(\frac{\tau}{2\pi}\right)^{1+d_2}\tau^{2/3}\int e^{\ii\tau(y\alpha+<z,\gamma>)}\frac{i e^{-i\pi/3} b_{\del}(y,z,\alpha,\gamma,\tau)}{A_+(\tau^{2/3}\zeta_0)} \hat{F}_{\tau}(\tau\alpha,\tau\gamma)d\alpha d\gamma,
\end{equation}
where $\zeta_0(\alpha,\gamma)=\alpha^{2/3}\tilde\zeta(\sqrt{1-\gamma^2}/\alpha)$ with $\tilde\zeta$ defined in Lemma \ref{lemzeta} ($\zeta_0(\alpha,\gamma)=\alpha^{2/3}\tilde\zeta(1/\alpha)$ if $d_2=0$) and where the symbol $b_{\partial}$ is elliptic, essentially supported for $\tilde \alpha=\frac{\alpha}{\sqrt{1-|\gamma|^2}}$ near $1$ and reads as
\begin{equation}\label{bjbord}
b_{\partial}(y,z,\alpha,\gamma,\tau):=a(0,y,z,\alpha,\gamma,\tau)\tilde\alpha^{-1/3}(\partial_\rho\tilde\zeta)(\frac{1+x}{\tilde\alpha})+\tau^{-1}\partial_x b(0,y,z,\alpha,\gamma,\tau).
\end{equation}
\end{prop}
\begin{proof}
We compute the normal derivatives of \eqref{windfn} and \eqref{MtauJ-1tau} : one derivative which falls on $A(\tau^{2/3}\zeta)$ (or $A_{+}(\tau^{2/3}\zeta)$) yields $\tau^{2/3}\partial_x\zeta=(\tau\sqrt{1-\gamma^2})^{2/3}\tilde\alpha^{-1/3}\partial_x\tilde \zeta$. Taking the difference yields \eqref{u_xomega} with $(A'-A_{+}'\frac{A}{A_{+}})$ (instead of $ie^{-i\pi/3}/A_+$) and using the Wronskian relation $A'(z)A_{+}(z)-A'_{+}(z)A(z)=\ii e^{-\ii \pi/3}$ allows to conclude.
 As $a|_{x=0}$ is elliptic near $\mathcal{C}^{d_1,d_2}_{Q_0}$ and $|\partial_x\tilde\zeta(\frac{1}{\tilde\alpha})|\sim 2^{1/3}$, $b_{\partial}$ must have the same properties as $a|_{x=0}$. 
\end{proof} 

The next lemma, proved in section \ref{seclemF} below, provides the explicit form of $\hat{F}_{\tau}(\tau\alpha,\tau\gamma)$.
\begin{lemma}\label{lemmaF} 
The Fourier transform of the function $F_{\tau}$ satisfying \eqref{windfn} has the following form :
\begin{itemize}
\item when $d_2=0$, $\hat{F_{\tau}}(\tau\alpha)$ is essentially supported for $|1-\alpha|$ small, and equals, modulo $O(\tau^{-\infty})$, 
\begin{equation}\label{FourierformFball}
\hat{F_{\tau}}(\tau\alpha)=
\frac{\tau^{\frac 23+\frac{d-3}{2}}}{\sqrt{s^2-1}^{\frac{d-1}{2}}}e^{-i\tau\Gamma_0(y_*(s),\alpha,s)}
 \kappa_{\e_1}(1-\alpha)f(\alpha,s,\tau),
\end{equation} 
for some small $\e_1$ (where the cut-off $\kappa_{\e_1}$ is as in Definition \ref{dfnkappa}), where $f$ is an elliptic asymptotic expansion with main contribution $\Sigma_{0,d}$ (defined in \eqref{vfreeform}) and where the phase function is given by 
\begin{equation}\label{defGamm}
\Gamma_0(y_*(s),\alpha,s)=y_*(s)\alpha+\sqrt{s^2-1}+\frac{(1-\alpha)^2}{2\sqrt{s^2-1}}(1+O(1-\alpha)), \text{ for } \alpha \text{ on the support of } \kappa_{\e_1}(1-\alpha).
\end{equation}

\item when $d_2\geq 1$, $\hat{F_{\tau}}(\tau\alpha,\tau\gamma)$ is supported for $|1-\frac{\alpha}{\sqrt{1-|\gamma|^2}}|$, $|\gamma|$ small and equals, modulo $O(\tau^{-\infty})$,
\begin{equation}\label{FourierformF}
\hat{F_{\tau}}(\tau\alpha,\tau\gamma)=
\frac{\tau^{\frac 23+\frac{d_1-3}{2}}}{\sqrt{s^2-1}^{\frac{d_1-1}{2}}}e^{-i\tau\sqrt{1-|\gamma|^2}\Gamma_0(y_*(s), \tilde \alpha,s)}
 \kappa_{\e_1}(1-\tilde\alpha)f(\tilde\alpha,\gamma,s,\tau),\quad \tilde\alpha=\frac{\alpha}{\sqrt{1-|\gamma|^2}},
\end{equation} 
for some small $\e_1$, where $f$ is an elliptic symbol with main contribution $\Sigma_{0,d}$ (from \eqref{vfreeform}) and $\Gamma_0(y_*(s), \tilde\alpha,s)$ is the given in \eqref{defGamm} for $\tilde\alpha (:=\alpha/\sqrt{1-|\gamma|^2})$ on the support of $\kappa_{\e_1}(1-\tilde\alpha)$.
\end{itemize}
\end{lemma}

\begin{rmq}\label{rmkball}
In the exterior of a cylinder in $\mathbb{R}^3$, we have obtained \eqref{FourierformF} in \cite[Lemma 2.9]{IaIv23} for $d_1=2$ and $d_2=1$. In the exterior of a ball in $\mathbb{R}^d$, we have announced \eqref{FourierformFball}  in \cite{ildispext} for $d_2=0$, without details. 
\end{rmq}

Replacing the explicit form of $\hat{F_{\tau}}$ from (\eqref{FourierformFball}, or, more general) \eqref{FourierformF} in \eqref{u_xomega} yields
\begin{cor}\label{corw0gl}
For $P=(0,y,\omega,z)\in\mathbb{S}^{d_1-1}\times \mathbb{R}^{d_2}$ near $\mathcal{C}^{d_1,d_2}_{Q_0}$ and with $|z|\ll s$, the following holds
\begin{itemize}
\item when $d_1=d$ and $d_2=0$, $\mathcal{F}(\partial_x u^+)(y,Q_0,\tau)$ becomes 
\begin{equation}\label{u_xomega23}
\mathcal{F}(\partial_x u^+)(y,Q_0,\tau)=\frac{\tau}{2\pi}
\frac{\tau^{\frac{1}{3}+\frac{d-1}{2}}}{\sqrt{s^2-1}^{\frac{d-1}{2}}}\int e^{i\tau (y\alpha-\Gamma_0(y_*(s),\alpha,s))}   
\kappa_{\e_1}(1-\alpha)\frac{i e^{-i\pi/3}b_{\partial}(y,\alpha,\tau)}{A_+(\tau^{2/3}\zeta_0(\alpha))}f(\alpha,s,\tau)d\alpha.
\end{equation}

\item when $d_2\geq 1$, $\mathcal{F}(\partial_x u^+)(y,z,Q_0,\tau)$ becomes 
\begin{multline}\label{u_xomega2}
\mathcal{F}(\partial_x u^+)(y,z,Q_0,\tau)=\left(\frac{\tau}{2\pi}\right)^{1+d_2}
\frac{\tau^{\frac{1}{3}+\frac{d_1-1}{2}}}{\sqrt{s^2-1}^{\frac{d_1-1}{2}} }\int e^{i\tau (y\alpha+<z,\gamma>-\sqrt{1-|\gamma|^2}\Gamma_0(y_*,(s)\frac{\alpha}{\sqrt{1-|\gamma|^2}},s))}\\\times 
\kappa_{\e_1}\Big(1-\frac{\alpha}{\sqrt{1-|\gamma|^2}}\Big)\frac{i e^{-i\pi/3}b_{\partial}(y,z,\alpha,\gamma,\tau)}{A_+(\tau^{2/3}\zeta_0(\alpha,\gamma))}  f\Big(\frac{\alpha}{\sqrt{1-|\gamma|^2}},\gamma,s,\tau\Big)d\alpha d\gamma.
\end{multline}
\end{itemize}

\end{cor}

\subsubsection{Proof of Lemma \ref{lemmaF} for $d_2=0$}\label{seclemF}
In this section we prove Lemma \ref{lemmaF} in the case $d_2=0$, when the obstacle is a ball in $\R^{d_1}$, and in the next section we explain how to deal with the additional variable $z$ in $\mathbb{R}^{d_2}$ when $d_2\geq 1$. We need to obtain $F_{\tau}$ such that $T_{\tau}(F_{\tau})=w_{in}(\cdot,Q_0,\tau)+O(\tau^{-\infty})$ near a glancing point on $\del\Omega_{d_1,0}$, where $\Omega_{d_1,0}=\mathbb{R}^{d_1}\setminus B_{d_1}(0,1)$ and where $Q_0=Q_N(s)$, hence we need to "invert" in some way $T_{\tau}$ near a glancing point. 
Using the integral form of the Airy function and its derivative, the operator $T_{\tau}: \mathcal{E}'(\mathbb{R})\rightarrow \mathcal{D}'(\mathbb{R}^{2})$ defined in \eqref{T-tau} (where we now take $d_2=0$) reads as
\[
T_{\tau}(F)(x,y)=\tau^{1/3}\frac{\tau}{2\pi}\int  e^{i\tau(\iota(y,\alpha)-\tilde y\alpha+\sigma^3/3+\sigma\zeta(x,\alpha))} (a+b\sigma/i)F(\tilde y)d\sigma d\tilde y d\alpha,
\]
where the phase functions $\iota$ and $\zeta$ are defined in \eqref{soleikballd}, $\iota(y,\alpha)=y\alpha$ and $\zeta(x,\alpha)=\alpha^{2/3}\tilde\zeta(\frac{1+x}{\alpha})$. 

According to Proposition \ref{MT}, at the glancing point in $T^*\mathbb{S}^{1}$ the trace on the boundary of $\zeta(x,\alpha)$ cancels only at a glancing direction and the symbols $a$, $b$ are supported near the glancing point; as $\zeta_0(\alpha)=0$ only at $\alpha=1$, we can introduce in the integral defining $T_{\tau}$ a cut-off $\kappa_{\e_1}(1-\alpha)$ with $\kappa_{\e_1}\in C^{\infty}_0((-\e_1,\e_1))$ as in Definition \ref{dfnkappa} supported for $|1-\alpha|<\e_1$ for some sufficiently small $\e_1>0$ depending on $\e_0$, such that $\kappa_{\e_1}(1-\alpha)=1$ on the support of $a$,$b$, so without changing the contribution of $T_{\tau}$ modulo $O(\tau^{-\infty})$. 

A point of $\mathcal{C}^{d,0}_{Q_0}$ is of the from $(0,y_*(s),\omega)$ with $y_*(s)=\arcsin(1/s)$: as $w_{in}(Q,Q_0,\tau)$ (from \eqref{vfreeball}) is independent of $\omega$, we need to "invert" $T_{\tau}$ near $(0,y_*(s))$ to find $F_{\tau}$.
Let $\varkappa\in C^{\infty}_0(\mathbb{R}^2)$ be a smooth cut-off function supported for $(x,y)$ in a sufficiently small neighborhood $\mathcal{V}_{(0,y_*(s))}$ of $(0,y_*(s))$ and equal to $1$ in near $(0,y_*(s))$, with $0\leq \varkappa\leq 1$ and such that $x<c$ on the support of $\varkappa$:  we aim at obtaining $F_{\tau}$ such that
\[
\varkappa(x,y)T_{\tau}(F_{\tau})(x,y)=\varkappa(x,y) w_{in}(x,y,Q_0,\tau)+O(\tau^{-\infty}).
\]
We define $\tilde T_{\tau}: \mathcal{D}'(\mathbb{R}_{\tilde y})\rightarrow \mathcal{D}'(\mathcal{V}_{(0,y_*)})$  by $\tilde T_{\tau}:=\varkappa T_{\tau}$, (and with $\kappa_{\e_1}(1-\alpha)$ in the integral defining $T_{\tau}$),
\begin{equation}\label{inttildeT}
\tilde T_{\tau}(F)(x,y):=\tau^{1/3}\frac{\tau}{2\pi}\int  e^{i\tau((y-\tilde y)\alpha+\sigma^3/3+\sigma\zeta(x,\alpha))} (a+b\sigma/i)\varkappa(x,y)\kappa_{\e_1}(1-\alpha)F(\tilde y)d\tilde y d\sigma d\alpha.
\end{equation}
\begin{lemma}
The operator $\tilde T_{\tau}$ is well defined microlocally near $(y,\alpha)=(y_*(s),1)$ and $(\tau^2+\Delta)\tilde T_{\tau}(F)\in O(\tau^{-\infty})$ for $(x,y)$ near $(0,y_*(s))$ (where $\varkappa=1$). Moreover, for $F$ supported outside a small neighborhood of $y_*(s)$ (depending on $\e_0$ and the support of $\varkappa$) we have $\tilde T_{\tau}(F)(x,y)=O(\tau^{-\infty})$.
\end{lemma}
\begin{proof}
The operator $\tilde T_{\tau}$ is equal to $T_{\tau}$ modulo $O(\tau^{-\infty})$ (due to the cut-off $\kappa_{\e_1}$). As $\varkappa=1$ near $(0,y_*(s))$, near this point $\tilde T_{\tau}(F)$ satisfies the same equation as $T_{\tau}$. The critical points of \eqref{inttildeT} satisfy $\sigma^2=-\zeta(x,\alpha)$ and $y=\tilde y-\sigma^3/3-\sigma\zeta(x,\alpha)$, hence the phase may be stationary only for values $\tilde y$ satisfying
\[
|\tilde y -y_*(s)|\leq |y-y_*(s)|+|\sigma^3+\sigma\zeta(x,\alpha)| =  |y-y_*(s)|+\frac 23 |\zeta(x,\alpha)|^{3/2}.
\]
As $\zeta(x,\alpha)$ is close to $0$ only for $x$ and $\alpha$ on the support of $\varkappa(x,y)\kappa_{\e_1}(1-\alpha)$, $\tilde y$ must stay close to $y_*(s)$.
\end{proof}
Next, we construct the adjoint $\tilde T^*_{\tau}: \mathcal{D}'(\mathcal{V}_{(0,y_*)})\rightarrow \mathcal{D}'(\mathbb{R}_{\overline y})$ of $\tilde T_{\tau}$, show that the operator $\tilde T^{*}_{\tau}\circ \tilde T_{\tau}$ is an elliptic FIO, hence invertible, and then define $F_{\tau}:=(\tilde T^{*}_{\tau}\circ \tilde T_{\tau})^{-1}(\varkappa w_{in})$ modulo $O(\tau^{-\infty})$.
The adjoint reads as
\begin{equation}\label{tildeT*}
\tilde T^*_{\tau}(w)(\overline y):=\tau^{1/3}\frac{\tau}{2\pi}\int  e^{i\tau((\overline y-y)\tilde \alpha-\tilde\sigma^3/3-\tilde\sigma\zeta(x,\tilde\alpha))} (a-b\sigma/i)\varkappa(x,y)\kappa_{\e_1}(1-\tilde \alpha)w(x,y)dx dy d\tilde\sigma d\tilde\alpha\,,
\end{equation}
and for $(x,y,\tilde\alpha)$ on the support of the symbol, its phase may be stationary only for $\sigma$ close to $0$. We can introduce a smooth cut-off $\kappa_{\e_2}(\tilde\sigma)$ supported for $|\tilde \sigma|<\e_2$ for some $\e_2>0$ small enough (depending only on $\e_1$ and the support of $\varkappa$), without changing $\tilde T_{\tau}^*(w)$ modulo $O(\tau^{-\infty})$. We now obtain 
\begin{lemma}\label{lemEtau}
Let $E_{\tau}:=(\tau^{1/3}\tilde T^*_{\tau})\circ \tilde T_{\tau}: \mathcal{D}(\mathbb{R}_{\tilde y})\rightarrow \mathcal{D}(\mathbb{R}_{\tilde y})$, then, microlocally near $(y,\alpha)=(y_*(s),1)$, $E_{\tau}$ is an elliptic pseudo-differential operator of degree $0$, hence invertible. Setting $F_{\tau}=E^{-1}_{\tau}\circ (\tau^{1/3}\tilde T^*_{\tau})(w_{in}(.,\tau))$ yields $\tilde T_{\tau}(F_{\tau})=w_{in}(.,\tau)+O(\tau^{-\infty})$ microlocally near $(y_*(s),1)$.
\end{lemma}
\begin{proof}
As $E_{\tau}:=\tau^{1/3}\tilde T^*_{\tau}\circ \tilde T_{\tau}$, using \eqref{inttildeT}, \eqref{tildeT*}, we explicitly compute
\begin{multline*}
E_{\tau}(F)(\overline y)=\frac{\tau^{3}}{4\pi^2}\int e^{i\tau\Big((\overline y-y)\tilde \alpha-\tilde \sigma^3/3-\tilde\sigma\zeta(x,\tilde\alpha)+(y-\tilde y)\alpha +\sigma^3/3+\sigma\zeta(x,\alpha)\Big)}\kappa_{\e_2}(\tilde\sigma)\kappa_{\e_1}(1-\alpha)\kappa_{\e_1}(1-\tilde\alpha)\\
\times (\varkappa(x,y))^2(a-b\tilde\sigma/i)(x,y,\tilde\alpha)(a+b\sigma/i)(x,y,\alpha)F(\tilde y)d\tilde y d\tilde\sigma d\tilde\alpha dx dy d\sigma d\alpha\,. 
\end{multline*}
Stationary phase applies in $y,\alpha$, with critical points $y(\tilde y, \sigma,\alpha):=\tilde y-\sigma\partial_{\alpha}\zeta(x,\alpha)$ and $\tilde\alpha=\alpha$; the critical value equals $(\overline y-\tilde y)\alpha-\tilde\sigma^3/3-\tilde\sigma\zeta(x,\alpha)+\sigma^3/3+\sigma\zeta(x,\alpha)$ and the symbol becomes 
\[
\tau^2\kappa_{\e_2}(\tilde\sigma)\kappa_{\e_1}(1-\alpha)\tilde\varkappa(x,y)(\tilde a-\tilde b\tilde\sigma/i)(x,y,\alpha)(\tilde a+\tilde b\sigma/i)(x,y,\alpha)|_{y=y(\tilde y,\sigma,\alpha)}
\]
where $\tilde a$, $\tilde b$ are asymptotic expansions with small parameter $\tau^{-1}$ and main contribution $a(x,y,\alpha)$, and $b(x,y,\alpha)$ which now depend upon $\sigma$ through $y=y(\tilde y, \sigma,\alpha)$; here $\tilde \varkappa$ is a smooth cut-off supported for $(x,y(\tilde y, \sigma,\alpha))$ near $(0,y_*(s))$. We further apply the stationary phase with respect to $x$ and $\tilde \sigma$ with critical points $\tilde\sigma^2+\zeta(x,\alpha)=0$ and $\partial_x\zeta(x,\alpha)(\tilde \sigma-\sigma)=0$. As $\partial_x\zeta(0,1)=\partial_x\tilde\zeta(1)=-2^{1/3}$, then $\partial_x\zeta(x,\alpha)$ does not vanish near $x=0$, $\alpha=1$ and the critical points satisfy $\tilde\sigma=\sigma$, $\zeta(x,\alpha)=-\sigma^2$. This yields $x_c$ as a smooth function of $\sigma^2$ and $\alpha$ for $\sigma$ close to $0$ and for $\alpha$ close to $1$. 
The second derivative w.r.t $x$, $\partial^2_{x}\zeta (x,\alpha) (\tilde\sigma-\sigma)$, vanishes at $\tilde\sigma=\sigma$, therefore the absolute value of the determinant of the Hessian matrix equals $|\partial_x\zeta(x,\alpha)|$, which is close to $2^{1/3}$ near  $x=0,\alpha=1$. The critical value of the phase becomes $(\overline y-\tilde y)\alpha$, and the symbol has main contribution $a^2+b^2\sigma^2$ with $\sigma$ close to $0$ on the support of $\kappa_{\e_2}$ and with $a$ elliptic near $(0,y_*(s),1)$. Integration w.r.t. $\sigma$ (which is no more in the phase) yields
\begin{equation}\label{formEtau}
E_{\tau}(F)(\overline y) 
= \frac{\tau}{2\pi}\int e^{i\tau(\overline y-\tilde y)\alpha}\kappa_{\e_1}(1-\alpha)\tilde\varkappa(x_c,\tilde y)\Sigma_E (\tilde y, \alpha,\tau)
 F(\tilde y)d\tilde y d\alpha\,,
\end{equation}
where 
$\Sigma_{E}(\tilde y,\alpha,\tau)$ is a symbol of order $0$, elliptic at $(y_*(s),1)$ that reads as an asymptotic expansion with main contribution $a^2$ and small parameter $\tau^{-1}$. Microlocally near $(y_*(s),1)$, the operator $E_{\tau}$ has an inverse $E^{-1}_{\tau}$ with an elliptic symbol $\kappa_{\e_1}(1-\alpha)\Sigma_{E^{-1}}(\tilde y,\alpha,\tau)$ of order $0$ and supported for $(\tilde y,\alpha)$ near $(y_*(s),1)$. 
Let $F_{\tau}: =E^{-1}_{\tau}(\tau^{1/3}\tilde T^*_{\tau})(\varkappa w_{in})$, then 
$F_{\tau}$ satisfies \eqref{windfn} near $\mathcal{C}^{d,0}_{Q_0}$ and, modulo $O(\tau^{-\infty})$ terms, it is given by
\begin{align}\label{Ftaufff}
F_{\tau}(\tilde y)= &\frac{\tau}{2\pi}\int e^{i\tau(\tilde y-\overline{y})\alpha}\kappa_{\e_1}(1-\alpha)\Sigma_{E^{-1}}(\tilde y,\alpha,\tau)(\tau^{1/3}\tilde T^*_{\tau})(\varkappa w_{in})(\overline y)d\overline y d\alpha\\
= &
    \begin{multlined}[t]
      \frac{\tau}{2\pi}\int e^{i\tau(\tilde y-\overline y)\alpha }\kappa_{\e_1}(1-\alpha)\Sigma_{E^{-1}}(\tilde y,\alpha,\tau)\tau^{1/3}\frac{\tau}{2\pi}\int e^{i\tau(\overline y-y)\tilde\alpha}
\Big( aA(\tau^{2/3}\zeta(x,\tilde \alpha))- b\tau^{-1/3}A'(\tau^{2/3}\zeta(x,\tilde\alpha))\Big)\\
\kappa_{\e_1}(1-\tilde\alpha) \varkappa (x,y)w_{in}(x,y,Q_0,\tau)dx dy d\tilde \alpha d\overline y d\alpha\,
\end{multlined}\\
=&
 \begin{multlined}[t]
\frac{\tau^{4/3}}{4\pi^2}\int e^{i\tau(\tilde y- y)\alpha }\kappa_{\e_1}(1-\alpha) \Sigma_{F}(\tilde y,\alpha,\tau)
\Big(aA(\tau^{2/3}\zeta(x,\alpha))-b\tau^{-1/3}A'(\tau^{2/3}\zeta(x,\alpha))\Big)\\ \varkappa(x,y) w_{in}(x,y,Q_0,\tau)dx dy d\alpha,
\end{multlined}
\end{align}
where one factor $\tau^{1/3}$ was used integrating over $\sigma$ in $\tilde T^*_{\tau}$, to obtain a linear combination of Airy functions and where, to obtain the last line we have applied the stationary phase with respect to $\overline y$ and $ \alpha$ with $\alpha=\tilde \alpha$, $\overline y=y$, which provided a factor $\tau^{-1}$. The symbol $\Sigma_F$ is an asymptotic expansion with main contribution $\kappa_{\e_1}(1-\alpha)\Sigma_{E^{-1}}(\tilde y,\alpha,\tau)$ and small parameter $\tau^{-1}$. It remains to deal with $\varkappa w_{in}$. 
\end{proof}
The next lemma allows to express $\varkappa w_{in}$ as an oscillatory integral with phase $\zeta$ given in Lemma \ref{lemzeta}.
\begin{lemma}\label{lemwingl}
For $Q_0=Q_N(s)$ and observation points $(x,y,\omega)$ near $\mathcal{C}^{d,0}_{Q_0}$, $\varkappa w_{in}(\cdot,Q_0,\tau)$ reads as
\begin{multline}\label{wjglformlem}
\varkappa(x,y)w_{in}(x,y,Q_0,\tau)=\frac{\tau}{2\pi} \Big(\frac{\tau}{\sqrt{s^2-1}}\Big)^{\frac{d-1}{2}}\varkappa(x,y)\int \kappa_{\e_3}(1-\tilde\alpha) \Sigma_{w}(x,\tilde\sigma,\tau)\\
\times e^{i\tau\big(y\tilde \alpha- \Gamma_0(y_*(s),\tilde\alpha,s)-\tilde\sigma^3/3-\tilde\sigma \zeta(x,\tilde\alpha)\big)} d\tilde \sigma d\tilde\alpha,
\end{multline}
where 
$\Sigma_{w}(x,\tilde\sigma,\tau)$ is an elliptic symbol of order $0$ and for a small $\e_3>0$ depending on the support of $\varkappa$, $\e_0$.
\end{lemma}
The proof of Lemma \ref{lemwingl} is postponed to the end of this section. 
We first show how to recover $F_{\tau}$ using \eqref{Ftaufff} and Lemma \ref{lemwingl}. 
We introduce \eqref{wjglformlem} in \eqref{Ftaufff} and replace the Airy bracket in \eqref{Ftaufff} by
\[
aA(\tau^{2/3}\zeta(x,\alpha))-b\tau^{-1/3}A'(\tau^{2/3}\zeta(x,\alpha))=\tau^{1/3}\int e^{i\tau(\frac{\sigma^3}{3}+\sigma\zeta(x,\alpha))}(a-b\sigma/i)d\sigma.
\]
As before, we can introduce a cut-off $\kappa_{\e_2}(\sigma)$ in the last integral without changing the contribution of the integral \eqref{Ftaufff}, as, for $\sigma$ outside a small neighborhood of $0$, the phase of $F_{\tau}(\tilde y)$ is non-stationary.

After stationary phase in $y, \alpha$, the phase of $F_{\tau}(\tilde y)$ becomes 
$\tilde y \alpha+\frac{\sigma^3}{3}+\sigma \zeta(\tilde x,\alpha)-\frac{\tilde \sigma^3}{3}-\tilde\sigma \zeta(\tilde x,\alpha)-\Gamma_0(y_*(s),\alpha,s)$. The stationary phase with respect to $x$, $\tilde \sigma$ also applies (as in Lemma \ref{lemEtau}) : at the critical points $\tilde\sigma=\sigma$ and $\tilde\sigma^2=-\zeta(x,\alpha)$ the hessian matrix has determinant $|\partial_{x}\zeta(x,\alpha)|\sim 2^{1/3}$. The stationary phase yields a factor $\tau^{-1}$ and the critical phase equals $\tilde y \alpha-\Gamma_0(y_*(s),\alpha,s)$. Hence, we obtain, modulo $O(\tau^{-\infty})$ terms
\begin{equation}\label{integralformFtauzII}
F_{\tau}(\tilde y)=\tau^{4/3+1/3-2}\frac{\tau}{2\pi}\Big(\frac{\tau}{\sqrt{s^2-1}}\Big)^{\frac{d-1}{2}} \int\int e^{i\tau (\tilde y \alpha-\Gamma_0(y_*(s),\alpha,s))}\kappa_{\e_1}(1-\alpha) \kappa_{\e_2}(\sigma)\tilde\Sigma_{F}(\tilde y,\alpha,\sigma,\tau) 
 d\sigma d\alpha,
\end{equation}
where $\tilde\Sigma_F $, has main contribution $\Sigma_{w}\Sigma_F(a-b\sigma/i)$, hence it is elliptic near $\sigma=0$, $\alpha=1$. Integrating in $\sigma$ (which appears only in the symbol and belongs to  a small neighborhood of $0$) and taking the Fourier transform achieves the proof of Lemma \ref{lemmaF}. In the following we prove Lemma \ref{lemwingl}.

\begin{proof}(of Lemma \ref{lemwingl})
Recall that $w_{in}$ is given in \eqref{winform}, with phase $\phi(x,\tilde y,0,s)$ defined in \eqref{Phi} (as $d_2=0$). If $|y-\tilde y|\geq \e>0$ for some $\e>0$, repeated integrations by parts w.r.t. $\alpha$ yield $O(\tau^{-\infty})$; as $y$ stays close to $y_*(s)$ on the support of $\kappa_{\e_0}$, we can introduce $\kappa_{2\e_0}(\tilde y-y_*(s))$ in the integral without changing its contribution modulo $O(\tau^{-\infty})$. Consider the integral in \eqref{winform} : as $\phi$ has a degenerate critical point of order two near $y_*(s)$, we will prove in Lemma \ref{lemgam} below that there exists a change of variables $\tilde y\rightarrow \sigma$ which transforms $\tilde y \alpha+\phi(x,\tilde y,0,s) $ into an Airy phase function of the form $\frac{\sigma^3}{3}+\sigma \alpha^{2/3}\tilde\zeta(\frac{1+x}{\alpha})+\Gamma_0(y_*(s),\alpha,s)$ (which is the normal form of functions with critical points of order two). As $\phi$ and $\iota\pm \frac 23(-\zeta)^{3/2}$ from Proposition \ref{MT} solve the same eikonal equation \eqref{MT-para}, it turns out that $\tilde\zeta$ is precisely the function defined in Lemma \ref{lemzeta}.  
Let 
\begin{equation}\label{phi}
\Phi(x,\tilde y,\alpha, s):=\tilde y \alpha+\phi(x,\tilde y,0,s).
\end{equation}
As $|\tilde y-y_*(s)|\leq 2\e_0$ on the support of $\kappa_{\e_0}$, $x$ is close to $0$ on the support of $\varkappa(x,y)$ and as moreover
\[
\partial_{\tilde y}\phi(x,\tilde y,0,s)=-\frac{s(1+x)\cos \tilde y}{\phi(x,\tilde y, 0,s)}, \quad \partial_{\tilde y}\phi(x,\tilde y,0,s)|_{\tilde y=\arcsin(\frac{1+x}{s})}=-(1+x),
\]
$\alpha$ has to stay close to $1+x$ as otherwise $\Phi$ is non-stationary w.r.t. $\tilde y$. Hence we can introduce $\kappa_{\e_3}(1-\alpha)$ in the symbol without changing its contribution modulo $O(\tau^{-\infty})$, with $\e_3>0$ small depending on $\e_0$ and the support of $\varkappa$.
As
$\partial_{\tilde y}^2\phi(x,\tilde y,0,s)=\frac{s(1+x)\sin \tilde y-(\partial_{\tilde y}\phi)^2}{\phi}$, it follows that $\partial_{\tilde y}^2\phi(x,\tilde y,0,s)=0$ when $\tilde y =\arcsin\left(\frac{1+x}{s}\right)$. 

\begin{lemma}\label{lemgam} Let $y_*(s,x):=\arcsin\left(\frac{1+x}{s}\right)$ and $\tilde y=y_*(s,x)+\mathtt{Y}$. 
There exists a unique change of variables $\mathtt{Y}\mapsto \sigma$, smooth and satisfying $\frac{d\mathtt{Y}}{d{\sigma}}\notin\{0,\infty\}$ so that, for $\tilde \zeta$ as in Lemma \ref{lemzeta} and $\Gamma_0$ as in \eqref{defGamm}, we have
\begin{equation}\label{fctssfin}
\Phi(x,y_*(s,x)+\mathtt{Y},\alpha,s)=\frac{{\sigma}^3}{3}+{\sigma}\alpha^{2/3}\tilde\zeta\big(\frac{1+x}{\alpha}\big)+\Gamma_0(y_*(s),\alpha,s).
\end{equation} 
\end{lemma}
We postpone the proof of Lemma \ref{lemgam} and continue the proof of Lemma \ref{lemwingl}.
After the changes of coordinates $\tilde y=y_*(s,x)+Y$ and $Y\rightarrow \sigma$ we obtain $\varkappa w_{in}$, modulo $O(\tau^{-\infty})$ terms, under the form
\begin{multline}\label{wjglformnew}
\varkappa(x,y) w_{in}(x,y,Q_0,\tau)=\frac{\tau}{2\pi}\Big(\frac{\tau}{\sqrt{s^2-1}}\Big)^{\frac{d-1}{2}}\varkappa(x,y)\int \Big(\frac{\sqrt{s^2-1}}{\phi(x,y_*(s,x)+Y(\sigma),0,s)}\Big)^{\frac{d-1}{2}}  \frac{dY}{d\sigma}\kappa_{\e_3}(1-\alpha)\\
\times \Sigma_d(\tau \phi(x,y_*(s,x)+Y(\sigma),0,s)) e^{i\tau(y\alpha-\Gamma_0(y_*(s),\alpha,s)-\sigma^3/3-\sigma \alpha^{2/3}\tilde\zeta(\frac{1+x}{\alpha}))} d\sigma d\alpha,
\end{multline}
where $|\frac{dY}{d\sigma}|\sim 1$ and $\sqrt{s^2-1}/\phi \sim 1$ for $(x, y_*(s,x)+Y(\sigma))$ on the support of $\varkappa$. 
Lemma \ref{lemwingl} follows taking 
\begin{equation}\label{symbnewwin}
\Sigma_w(x,\sigma,\tau):=\Big(\frac{\sqrt{s^2-1}}{\phi(x,y_*(s,x)+Y(\sigma),0,s)}\Big)^{\frac{d-1}{2}}  \frac{dY}{d\sigma}\Sigma_d(\tau \phi(x,y_*(s,x)+Y(\sigma),0,s)).
\end{equation}
\end{proof}

\begin{proof}(of Lemma \ref{lemgam}; this follows closely \cite[Lemma 2.10]{IaIv23})
As the phase $\Phi$ has degenerate critical points of order exactly two, it follows from \cite[7.7.18]{horFIO} that there exists a unique change of variables $\mathtt{Y}\mapsto \sigma$ which is smooth and satisfying $\frac{d\mathtt{Y}}{d{\sigma}}\notin\{0,\infty\}$ and that there exist smooth functions $\zeta^{\#}(x,\alpha,s)$ and $\Gamma(x,\alpha,s)$ such that
\begin{equation}\label{fctss}
\Phi(x,y_*(s,x)+\mathtt{Y},\alpha,s)=\frac{{\sigma}^3}{3}+{\sigma}\zeta^{\#}(x,\alpha,s)+\Gamma(x,\alpha,s).
\end{equation}

For $\tilde y$ near $y_*(s,x)$ there are two (non-degenerate) critical points $y_{\pm}=y_{\pm}(s,x,\alpha)$ of $\Phi$ satisfying
\begin{equation}\label{ycritpmsF}
s(1+x)\sin(y_\pm)=\alpha^2\pm\sqrt{s^2-\alpha^2}\sqrt{(1+x)^2-\alpha^2},  \phi(x,y_{\pm},0,s)=\sqrt{s^2-\alpha^2}\mp\sqrt{(1+x)^2-\alpha^2}.
\end{equation}
As the change of coordinates is regular, the critical points $\mathtt{Y}_{\pm}:=y_{\pm}(s,x,\alpha)-y_*(s,x)$ of $\Phi$ must correspond to $\sigma_{\pm}=\pm\sqrt{-{\zeta^{\#}}(x,\alpha,s)}$.
Write $\zeta^{\#}(x,\alpha,s):=\alpha^{\frac23}\tilde{\zeta}^{\#}\big(\frac{1+x}{\alpha},\alpha,s\big)$. 
We show that $\tilde\zeta^{\#}$ satisfies the same equation as $\tilde\zeta$ in \eqref{lemzeta}.
As the critical values of the two functions in \eqref{fctss} coincide, we have
\begin{equation}\label{sum-diff}
\Phi(x,y_*(s,x)+\mathtt{Y}_{\pm},\alpha,s)=\mp\frac23{(-\zeta^{\#})}^{\frac32}(x,\alpha,s)+\Gamma(x,\alpha,s),
\end{equation}
which implies
$\frac 43 \alpha{(-\tilde\zeta^{\#})}^{\frac32}(\frac{1+x}{\alpha},\alpha,s)=\Phi(x,y_{-},\alpha,s)-\Phi(x,y_{+},\alpha,s)$.
Taking the derivative with respect to $x$ in the last equation yields (with $y_{\pm}=y_*(s,x)+\mathtt{Y}_{\pm}$)
\begin{equation}\label{eltilde}
\begin{aligned}
2(-\partial_x\tilde\zeta^{\#})(-\tilde\zeta^{\#})^{\frac12}=&\partial_x\phi(x,y_*(s,x)+\mathtt{Y}_{-},0,s)-\partial_x\phi(x,y_*(s,x)+\mathtt{Y}_{+},0,s)\\
									  &-\partial_x y_+\partial_y\Phi(x,y_+,\tilde\alpha,s)+\partial_x y_{-} \partial_y\Phi(x,y_-,\tilde\alpha,s).
\end{aligned}\end{equation}
The last two terms in the second line of \eqref{eltilde} vanish as $y_{\pm}$ are the critical points of the function $\Phi$ with respect to $y$ ; for the same reason we have that $\partial_y\phi(x,y_{\pm}(s,x,\alpha),0,s)=-\alpha$. As $\phi$ satisfies the eikonal equation $(\partial_x\phi)^2(x,y,0,s)+\frac{1}{(1+x)^2}(\partial_y\phi)^2(x,y,0,s)=1$,
then $(\partial_x\phi(x,y_{\pm},0,s))^2=1-\frac{\alpha^2}{(1+x)^2}$.
Moreover, $\partial_x\phi_{|y_{\pm}}=\frac{s}{\phi(x,y_{\pm},0,s)}(\frac{1+x}{s}-\sin(y_{\pm}))$ which is non positive in the ``$y_+$ case" and positive in the ``$y_{-}$ case".
Eventually we obtain, using \eqref{eltilde} and the corresponding signs of $\partial_x\phi$, $-\tilde\zeta^{\#}[-\partial_x\tilde\zeta^{\#}]^2=1-\frac{\alpha^2}{(1+x)^2}$,
which is the same equation as in Lemma \ref{lemzeta} with $\rho=\frac{1+x}{\alpha}$. As the degenerate critical point occurs at $\sigma=0$, hence at $\zeta^{\#}=0$, we deduce by uniqueness of the solution that $\tilde\zeta^{\#}=\tilde\zeta=\tilde\zeta(\frac{1+x}{\alpha})$.\\

Next, we compute the explicit form of the function $\Gamma(x,\alpha,s)$. Taking the sum in \eqref{sum-diff} gives $\Gamma(x,\alpha,s)=\tfrac12(\Phi(x,y_+(s,x,\alpha),\alpha,s)+\Phi(x,y_{-}(s,x,\alpha),\alpha,s))$ ; taking the derivative in $x$ yields $\partial_x\Gamma(x,\alpha,s)=0$. As such, $\Gamma$ is independent of $x$ ; setting $\Gamma_0(y_*(s),\alpha,s):=\Gamma(0,\alpha,s)$, then
\[
\Gamma_0(y_*(s),\alpha,s)=\frac{1}{2}\big((y_++y_-)\alpha+\phi(0,y_+,0,s)+\phi(0,y_-,0,s)\big),
\]
where $y_{\pm}=y_{\pm}(s,0,\alpha)$. We develop the right hand side term in the equality above near $\alpha=1$. For $(x,y)$ near $(0,y_*(s))$, $y$ remains sufficiently close to $y_*(s,x)$ : shrinking the support of $\varkappa$ if necessary, we may assume $|y-y_*(s,x)|<1/2$ for all $(x,y)$ on the support of $\varkappa$. For $|y_{\pm}-y_*(s)|<1/2$ we may compute, using \eqref{ycritpmsF} with $x=0$, the first approximation of $y_{\pm}$ :  we have
\begin{equation}\label{ypm}
y_{\pm}(s,0,\alpha)=\arcsin \Big(\frac{\alpha^2}{s}\pm\sqrt{1-\alpha^2}\sqrt{1-\frac{\alpha^2}{s^2}}\Big), \quad y_\pm(s,0,1)=\arcsin(\frac 1s)=y_*(s).
\end{equation}
As $\Gamma_0(y_*(s),\alpha,s)=\frac 12(\Phi(0,y_+,\alpha,s)+\Phi(0,y_-,\alpha,s))$ and $\partial_y\Phi|_{y_{\pm}}=0$ then
\[
\partial_{\alpha}\Gamma_0(y_*(s),\alpha,s)=\frac 12(y_++y_-)+\frac 12 \sum_{\pm}\partial_{\alpha}y_{\pm}\partial_y\Phi|_{y_{\pm}}=\frac 12(y_++y_-).
\]
This yields $\Gamma_0(y_*(s),1,s)=\sqrt{s^2-1}+\arcsin \frac{1}{s}$ and $\partial_{\alpha}\Gamma_0(y_*(s),1,s)=\arcsin(1/s)$. We need the higher order derivatives : using \eqref{ypm}, it follows that $(y_++y_-)$ reads as an asymptotic expansion of even powers of $\sqrt{1-\alpha^2}$ and with main term $\arcsin (\frac{\alpha^2}{s})$. We find, with $Z_{\pm}=\frac{\alpha^2}{s}\pm\sqrt{1-\alpha^2}\sqrt{1-\frac{\alpha^2}{s^2}}$, $Z_{\pm}|_{\alpha=1}=\frac{1}{s}$,
\[
\frac 12 \partial_{\alpha}(y_++y_-)=\frac{\alpha}{s}\Big(\frac{1}{\sqrt{1-Z_+^2}}+\frac{1}{\sqrt{1-Z_-^2}}\Big)-\frac{\alpha(s^2+1-2\alpha^2)}{2s^2\sqrt{1-\alpha^2}\sqrt{1-\frac{\alpha^2}{s^2}}}\Big(\frac{1}{\sqrt{1-Z_+^2}}-\frac{1}{\sqrt{1-Z_-^2}}\Big).
\] 
As $\Big(\frac{1}{\sqrt{1-Z_+^2}}-\frac{1}{\sqrt{1-Z_-^2}}\Big)=\frac{Z_+^2-Z_-^2}{\sqrt{1-Z_+^2}\sqrt{1-Z_-^2}(\sqrt{1-Z_+^2}+\sqrt{1-Z_-^2})}$ and $Z_+^2-Z_-^2=4\frac{\alpha^2}{s}\sqrt{1-\alpha^2}\sqrt{1-\frac{\alpha^2}{s^2}}$,
\[
\frac 12 \partial_{\alpha}(y_++y_-)=\frac{\alpha}{s}\Big(\frac{1}{\sqrt{1-Z_+^2}}+\frac{1}{\sqrt{1-Z_-^2}}\Big)-\frac{2\alpha^3(s^2+1-2\alpha^2)}{s^3\sqrt{1-Z_+^2}\sqrt{1-Z_-^2}(\sqrt{1-Z_+^2}+\sqrt{1-Z_-^2})}.
\]
At $\alpha=1$ we obtain $\partial^2_{\alpha}\Gamma_0(y_*(s),1,s)=\frac 12 \partial_{\alpha}(y_++y_-)|_{\alpha=1}=\frac{1}{\sqrt{s^2-1}}$. In the same way we notice that all the higher order derivatives of $\Gamma_0$ come with a factor $\frac{1}{\sqrt{s^2-1}}$. The proof of Lemma \ref{lemgam} is achieved.
\end{proof}

\subsubsection{Proof of Lemma \ref{lemmaF} for $d_2\geq 1$} \label{seclemFcyl}
We now we explain how to obtain the proof of Lemma \ref{lemmaF} from the case $d_2=0$, when the obstacle is a cylinder in $\R^{d_1+d_2}$ when $d_2\geq 1$. At a point on $\mathcal{C}^{d_1,d_2}_{Q_0}$, of the form $(0,y_*(s),\omega,z)$, the phase $\zeta(x,\alpha,\gamma)|_{x=0}$ cancels only at a glancing direction, that is for $\frac{\alpha}{\sqrt{1-|\gamma|^2}}=1$. Let $T_{\tau}$ as in \eqref{T-tau} and introduce $\tilde T_{\tau}$ (as in \eqref{tildeT*} with with additional variables $(z,\gamma)$, a factor $(\frac{\tau}{2\pi})^{d_2}$ and) with cut-off $\varkappa(x,y) \kappa_{\e_1}(1-\frac{\alpha}{\sqrt{1-|\gamma|^2}})$. We then proceed in exactly the same way as before (defining $\tilde T^*_{\tau}$ and $E_{\tau}$) to obtain $F_{\tau}(y,z)$ such that $\varkappa w_{in}=\tilde T_{\tau}(F_{\tau})+O(\tau^{-\infty})$, by systematically eliminating the $z,\gamma$ variables by usual stationary phase, which lead to a very similar formula for $E_{\tau}$. 
As the critical points with respect to $\tilde z,\gamma$ in $w_{in}$ satisfy $z=\tilde z=-\phi(0,\tilde y,0,s)\frac{\gamma}{\sqrt{1-|\gamma|^2}}$ with $\phi(0,\tilde y,0,s)\sim s$, the condition $|z|\ll s$ is useful as it implies $\Big|\frac{\gamma}{\sqrt{1-|\gamma|^2}}\Big|\ll 1$, hence $|\gamma|\ll1$ and $\sqrt{1-|\gamma|^2}\sim 1$, avoiding the (degenerate) situation when $(\alpha,|\gamma|)$ is near $(0,1)$ which corresponds to initial directions almost parallel to the $Oz$ axis for which stationary methods do not hold as the Hessian may be too small. Lemma \ref{lemwingl} in the case $d_2\geq 1$ yields $\varkappa w_{in}$ as follows 

\begin{multline}\label{wjglformlemcyl}
\varkappa(x,y)w_{in}(x,y,Q_0,\tau)=\Big(\frac{\tau}{2\pi}\Big)^{1+d_2}\Big(\frac{\tau}{\sqrt{s^2-1}}\Big)^{(d_1-1)/2}\varkappa(x,y)\int \kappa_{\e_3}(1-\frac{\alpha}{\sqrt{1-|\gamma|^2}}) \\
\times\Sigma_{w}(x,\tilde\sigma,\gamma,\tau)
 e^{i\tau\big(<z,\gamma>+y\alpha -\sqrt{1-|\gamma|^2}\Gamma_0\big(y_*(s),\frac{\alpha}{\sqrt{1-|\gamma|^2}},s\big)+\tilde\sigma^3/3+\tilde\sigma \zeta(x,\alpha,\gamma)\big)} d\tilde \sigma d\alpha d\gamma,
\end{multline}
where $\Sigma_{w}$ as in \eqref{symbnewwin} with $\sqrt{1-|\gamma|^2}^{\frac{d_1}{2}}\Sigma_d\Big(\tau \frac{\phi(x,y_*(s,x)+Y(\sigma)}{\sqrt{1-|\gamma|^2}},0,s\Big)$ instead of $\Sigma_d(\tau \phi(x,y_*(s,x)+Y(\sigma),0,s))$. The factor $\sqrt{1-|\gamma|^2}$ appears naturally as critical value of the phase of \eqref{winform} (after the stationary phase in $\overline z$ after setting $\tilde z=\phi(x,\tilde y,\phi(x,\tilde y,0,s)\overline z,s)$). For $d_1=2$, $d_2=1$, \eqref{wjglformlemcyl} has been proved in \cite{IaIv23}.

\section{Proof of Theorem \ref{thmCE} for the wave equation }\label{secCE}

In this section we prove Theorem \ref{thmCE}. Recall that in Proposition \ref{propudiezh} we obtained the form of $u^{\#}_{h}(Q,Q_0,t)$ in terms of $\mathcal{F}(\partial_{x}u^+|_{\partial\Omega_{d_1,d_2}})$ and in Corollary \ref{corw0gl} we have explicitly computed $\mathcal{F}(\partial_x u^+|_{\partial\Omega_{d_1,d_2}})$ near the $\mathcal{C}^{d_1,d_2}_{Q_0}$. In the following we estimate $u^{\#}_{h,\kappa_{\e_0}}(Q,Q_0,t)$ given in \eqref{vhformnewchi} when the observation point is $Q=Q_S(s)$ is placed behind the obstacle, at the same distance $s$ as $Q_0=Q_N(s)$ from the center of the ball $B_{d_1}(0,1)\times \{0\}^{d_2}$. We aim at proving that, for $s\sim h^{-1/3}$, the maximum value of $|u^{\#}_{h,\kappa_{\e_0}}(Q_S(s),Q_0,t)|$ is reached for $t=2(\sqrt{s^2-1}+y_*(s))$, $y_*(s)=\arcsin \frac 1s$, and yields the same loss as in the exterior of a ball in $\mathbb{R}^{d_1}$, for all $d_1\geq 4$.

\begin{lemma}
For $\tau$ on the support of $\chi(h\tau)$ the following holds, modulo $O(\tau^{-\infty})$ contributions
\begin{equation}\label{Fourtransglancing}
\mathcal{F}(u^{\#}_{h,\kappa_{\e_0}})(Q_S(s),Q_N(s),\tau)=\chi(h\tau) \frac{\tau^{d-2+\frac 13-\frac{d_2}{2}}}{\sqrt{s^2-1}^{d_1-1+\frac{d_2}{2}}} e^{-i\tau(2\sqrt{s^2-1}+2y_*(s))} f_0(y_*(s)\tau^{1/3},s,\tau),
\end{equation}
where $\mathcal{F}(u^{\#}_{h,\kappa_{\e_0}})(\cdot,\tau)$ denotes the Fourier transform in time of $u^{\#}_{h,\kappa_{\e_0}}(\cdot,t)$ and where $f_0$ is an asymptotic expansion with small parameter $\tau^{-1/3}$ so that $|f_0(\delta,s,\tau)|\sim 1$ for $\delta\sim 1$. 

As a consequence, at $t=2(\sqrt{s^2-1}+\arcsin \frac 1s)$ and for $s\sim h^{-1/3}$, this yields
\begin{equation}\label{glancing}
\Big|u^{\#}_{h,\kappa_{\e_0}}(Q_S(s),Q_N(s),t)\Big| \sim \frac{1}{h^d}\Big(\frac{h}{t}\Big)^{(d-1)/2}\times h^{-\frac{d_1-3}{3}}.
\end{equation}
\end{lemma}

\begin{proof}
Using \eqref{vhformnewchi}, it follows that the Fourier transform in time of $u^{\#}_{h,\kappa_{\e_0}}$ reads, modulo $O(h^{\infty})$ terms,  
\begin{multline}\label{Fudiezzer}
\mathcal{F}(u^{\#}_{h,\kappa_{\e_0}})(Q_S(s),Q_N(s),\tau)=\chi(h\tau) \int_{P=(0,y,\omega,z)\in\mathbb{S}^{d_1-1}\times \mathbb{R}^{d_2}}\mathcal{F}(\partial_x u^+)(P,Q_0,\tau)\kappa_{\e_0}(y-y_*(s))\\
 \times \frac{\tau^{\frac{d-3}{2}}}{|Q_S(s)-P|^{\frac{d-1}{2}}}\Sigma_d(\tau|Q_S(s)-P|)e^{-i\tau|Q_S(s)-P|}d\sigma(P).
 \end{multline}
The form of $\mathcal{F}(\partial_x u^+)(P,Q_0,\tau)$ is given in Corollary \ref{corw0gl}.
Introducing \eqref{u_xomega2} in \eqref{Fudiezzer} yields
\begin{multline}\label{Fudiez}
\mathcal{F}(u^{\#}_{h,\kappa_{\e_0}})(Q_S(s),Q_N(s),\tau)
=\chi(h\tau)\left(\frac{\tau}{2\pi}\right)^{1+d_2}
\frac{\tau^{\frac{1}{3}+\frac{d_1-1}{2}}}{\sqrt{s^2-1}^{\frac{d_1-1}{2}}}\int e^{i\tau (y\alpha+<z,\gamma>-\sqrt{1-|\gamma|^2}\Gamma_0(y_*(s),\frac{\alpha}{\sqrt{1-|\gamma|^2}},s))}\\ \times \kappa_{\e_0}(y-y_*(s)) 
\kappa_{\e_1}\Big(1-\frac{\alpha}{\sqrt{1-|\gamma|^2}}\Big)\frac{i e^{-i\pi/3}b_{\partial}(y,z,\alpha,\gamma,\tau)}{A_+(\tau^{2/3}\zeta_0(\alpha,\gamma))}  f\Big(\frac{\alpha}{\sqrt{1-|\gamma|^2}},\gamma,s,\tau\Big)\\
\times \tau^{\frac{d-3}{2}}\frac{\Sigma_d(\tau\phi(0,-y,z,s))}{\phi(0,-y,z,s)^{\frac{d-1}{2}}}e^{-i\tau\phi(0,-y,z,s)} d\alpha d\gamma dy dz.
\end{multline}
with $b_{\delta}$ and $f$ elliptic and where, for $P=(0,y,\omega,z)\in \mathbb{S}^{d_1-1}\times \mathbb{R}^{d_2}$ we have replaced the distance $|Q_S(s)-P|$ by the $\phi(0,-y,z,s)$ with $\phi$ defined in \eqref{Phi}. In fact, for $P$ with $(r,\varphi,\omega,z)$ - coordinates $(1,\pi/2-y,\omega,z)$ and $Q_S(s)$ with $(r,\varphi,\omega,z)$ - coordinates $(s,\pi,\cdot,0)$, we have by \eqref{tildePhi}
\begin{multline}\label{psiphi-}
|P-Q_S(s)|=\big(|z-z_{Q_S(s)}|^2+(\cos(\pi/2-y)-s\cos(\pi))^2+\sin^2(\pi/2-y)\big)^{1/2}\\
=\sqrt{|z|^2+1+2s\sin y +s^2}=\phi(0,-y,z,s).
\end{multline}
The next lemma follows by performing a suitable change of coordinate in $y$ (and integrating in $(z,\gamma)$)
\begin{lemma}\label{lemphastatG}
There exist symbols $f_{1,2}$ such that one has, modulo $O(h^{\infty})$ terms,
\begin{multline}\label{Fudiez2chi00}
\mathcal{F}(u^{\#}_{h,\kappa_{\e_0}})(Q_S(s),Q_N(s),\tau)
=\chi(h\tau) \frac{\tau^{d-1-\frac{d_2}{2}-\frac 23}}{\sqrt{s^2-1}^{d_1-1+\frac{d_2}{2}}} \\
\times \int e^{-2i\tau\Gamma_0(y_*,\alpha(\beta,\tau),s)} \kappa_{2\e_1}(\tau^{-2/3}\beta)
\frac{( f_1A(-\beta)+\tau^{-1/3} f_2A'(-\beta))}{A_+(-\beta)} d\beta,
\end{multline}
where $f_1,f_2(\alpha,s,\tau)$ are asymptotic expansions 
with $f_1$ elliptic and where $\alpha=\alpha(\beta,\tau)$ satisfies $\tau^{2/3}\zeta_0(\alpha)=-\beta$.
 \end{lemma}
\begin{proof}
Consider first the case $d_2=0$, $d_1=d$, when \eqref{Fudiez} has a simpler form as follows
\begin{multline}\label{Fudiezball}
\mathcal{F}(u^{\#}_{h,\kappa_{\e_0}})(Q_S(s),Q_N(s),\tau)
=\chi(h\tau) \frac{\tau}{2\pi}
\frac{\tau^{\frac{1}{3}+\frac{d_1-1}{2}}}{\sqrt{s^2-1}^{\frac{d_1-1}{2}}}\int e^{i\tau (y\alpha-\Gamma_0(y_*(s),\alpha,s)-\phi(0,-y,0,s))}\\ \times \kappa_{\e_0}(y-y_*(s)) 
\kappa_{\e_1}(1-\alpha)\frac{i e^{-i\pi/3}b_{\partial}(y,\alpha,\tau)}{A_+(\tau^{2/3}\zeta_0(\alpha))}  f(\alpha,s,\tau) \tau^{\frac{d-3}{2}}\frac{\Sigma_d(\tau\phi(0,-y,0,s))}{\phi(0,-y,0,s)^{\frac{d-1}{2}}} d\alpha dy.
 \end{multline}
Notice that the part of the phase depending on $y$ equals $-\big((-y)\alpha+\phi(0,-y,0,s)\big)=-\Phi(0,-y,\alpha,s)$, with $\Phi$ defined in \eqref{phi}. Applying Lemma \ref{lemgam} with $x=0$ and $-y=y_*(s)+Y$, $y_*(s)=\arcsin(1/s)$, there exists a unique change of variables $Y\rightarrow \sigma$, smooth, such that, for $-y=y_*(s)+Y(\sigma)$,
\[
\Phi(0,-y,\alpha,s)=\frac{\sigma^3}{3}+\sigma \zeta_0(\alpha)+\Gamma_0(y_*(s),\alpha,s).
\]
As $\phi(0,y_*(s),0,s)=\sqrt{s^2-1}$ and $-y=y_*(s)+Y$ with $Y$ on the support of $\kappa_{\e_0}(-Y-2y_*(s))$, we write 
\[
\frac{\kappa_{\e_0}(y-y_*(s))}{\phi(0,-y,0,s)^{\frac{d-1}{2}}}=\frac{1}{\sqrt{s^2-1}^{\frac{d-1}{2}}}\times \kappa_{\e_0}(-Y-2y_*(s))\Big(\frac{\sqrt{s^2-1}}{\phi(0,y_*(s)+Y,0,s)}\Big)^{\frac{d-1}{2}},
\]
where $|Y|\leq 2 y_*(s)+\e_0$ on the support of $\kappa_{\e_0}$. As $s$ will be taken large enough, depending on the frequency, the term in brackets will remain close to $1$ on the support of $\kappa_{\e_0}$. 
We may then write, with $d_1=d$,
\begin{multline}\label{Fudiezballong}
\mathcal{F}(u^{\#}_{h,\kappa_{\e_0}})(Q_S(s),Q_N(s),\tau)
=\chi(h\tau) \frac{\tau}{2\pi}
\frac{\tau^{\frac{d_1-1}{2}+\frac{d-3}{2}}}{\sqrt{s^2-1}^{d_1-1}}\tau^{1/3}\int e^{-i\tau (\frac{\sigma^3}{3}+\sigma \zeta_0(\alpha)+2\Gamma_0(y_*(s),\alpha,s))}\\ \times 
\kappa_{\e_1}(1-\alpha)\frac{\Sigma_{\#}(\sigma,\alpha,s,\tau)}{A_+(\tau^{2/3}\zeta_0(\alpha))} d\sigma d\alpha,
 \end{multline}
where, on the support of $\kappa_{\e_1}$, $\Sigma_{\#}$ is elliptic near $\sigma=0$ and $\alpha=1$, given by, with $Y=Y(\sigma)$,
\begin{multline}
\Sigma_{\#}(\sigma,\alpha,s,\tau):=i e^{-i\pi/3} \kappa_{\e_0}(-Y(\sigma)-2y_*(s))b_{\partial}(-y_*(s)-Y,\alpha,\tau) f(\alpha,s,\tau)\\
\times \Sigma_d(\tau\phi(0,y_*(s)+Y,0,s))\Big(\frac{\sqrt{s^2-1}}{\phi(0,y_*(s)+Y,0,s)}\Big)^{\frac{d-1}{2}}\frac{dY}{d\sigma}.
\end{multline}
We now apply \cite[Theorem 7.7.18]{horFIO} to write the integral \eqref{Fudiezballong} as a combination of Airy functions.
\begin{thm}\label{thmhorFIO}(\cite[Theorem 7.7.18]{horFIO})
Let $\psi(\sigma,\alpha):=\frac{\sigma^3}{3}+\sigma \zeta_0(\alpha)$ be a real valued $C^{\infty}$ function near $0$ in $\R$ and $\Sigma_{\#}(\sigma,\cdot)\in C^{\infty}_0$ supported for $\sigma$ near $0$. Then there exist real valued functions $f_{i}\sim_{1/\tau}\sum_{j\geq 0} f_{i,j} \tau^{-j}$, $i\in \{1,2\}$ with $f_{i,j}\in C^{\infty}_0$ and $f_{1,0}=\Sigma_{\#}(0,\cdot)$ such that
\begin{equation}
\tau^{1/3}\int \Sigma_{\#}(\sigma, \alpha,s,\tau) e^{i\tau\psi(\sigma,\alpha)} d\sigma= f_1(\alpha,s,\tau)A(\tau^{2/3}\zeta_0(\alpha))+\tau^{-1/3}f_2(\alpha,s,\tau)A'(\tau^{2/3}\zeta_0(\alpha)) +O(\tau^{-\infty}).
\end{equation}
\end{thm}
Set $\tau^{2/3}\zeta_0(\alpha)=-\beta$: as $\zeta_0(\alpha)=\alpha^{2/3}\tilde\zeta(1/\alpha)$ and $|\tilde\zeta(\rho)|=2^{1/3}\frac{|1-\rho|}{\rho}$ for $\rho$ near $1$, we get $\alpha=\alpha(\beta,\tau)$,
\begin{equation}\label{eqtalphabet}
\alpha(\beta,\tau):=1-2^{-1/3}\beta \tau^{-2/3}+O(\beta^2\tau^{-4/3}),\quad \frac{d\alpha}{d\beta}=- \tau^{-2/3}(1- O(\tau^{-2/3})).
\end{equation}
Also, $|\beta|\leq 2^{1/3}\e_1 \tau^{2/3}$ on the support of $\kappa_{\e_1}(1-\alpha(\beta,\tau))$, so introducing $\kappa_{2\e_1}(\tau^{-2/3}\beta)$ doesn't change the integral modulo $O(\tau^{-\infty})$. This concludes the proof when $d_2=0$. When $d_2\geq 1$ we reduce the integral in \eqref{Fudiez} to \eqref{Fudiezballong} as follows : let $\alpha=\sqrt{1-|\gamma|^2}\tilde\alpha$, $z=\phi(0,-y,0,s) \overline{z}$, then the phase becomes
\[
\sqrt{1-|\gamma|^2}(y\tilde \alpha-\Gamma_0(y_*(s),\tilde\alpha,s))+\phi(0,-y,0,s) (<\overline z,\gamma>-\sqrt{1+|\overline z|^2}),
\]
with critical point $\overline z_{c}=\frac{\gamma}{\sqrt{1-|\gamma|^2}}$. The stationary phase yields a factor $\tau^{d_2-\frac{d_2}{2}}$ and the critical value is 
\begin{equation}
\sqrt{1-|\gamma|^2}\Big(y\tilde \alpha-\Gamma_0(y_*(s),\tilde\alpha,s)-\phi(0,-y,0,s))\Big)=\sqrt{1-|\gamma|^2}\Big(\frac{\tilde\sigma^3}{3}+\tilde\sigma \tilde\alpha^{2/3}\tilde\zeta(\frac{1}{\tilde\alpha}) -2\Gamma_0(y_*(s),\tilde\alpha,s)\Big).
\end{equation}
We obtain $\mathcal{F}(u^{\#}_{h,\kappa_{\e}})$ as in \eqref{Fudiezballong}, where the power of $\tau$ is $1+\frac{d_1-1}{2}+\frac{d-3}{2}+\frac{d_2}{2}+\frac 13=d-1+\frac 13$, where $\tau^{2/3}\zeta_0$ is now replaced by $(\tau\sqrt{1-|\gamma|^2})^{2/3}\zeta_0$, where the coefficient of $2i\Gamma_0$ in the phase is $\tau\sqrt{1-|\gamma|^2}$ (which is the new large parameter) and where $f_{1,2}$ depend also on $\gamma$. It remains to eliminate $\gamma$ by stationary phase. We won't do this directly as this may affect the form of the Airy combination that we want to keep. Instead, for $\tilde \tau:=\tau\sqrt{1-|\gamma|^2}$, we consider the same change of coordinates as before
\[
\tau^{2/3}\zeta_0(\alpha)=\tilde\tau^{2/3}\tilde\alpha^{2/3}\tilde\zeta(1/\tilde\alpha)=-\beta, \text{ where }  \tilde\tau:=\tau\sqrt{1-|\gamma|^2}, \alpha=\sqrt{1-|\gamma|^2}\tilde \alpha.
\]
Then $\tilde\alpha=\tilde\alpha(\beta,\tilde\tau)$ is given by \eqref{eqtalphabet}. As such, the Airy factors do not involve $\gamma$ anymore, and the phase $2i\tau\sqrt{1-|\gamma|^2}\Gamma_0(y_*(s),\tilde \alpha(\beta,\tau\sqrt{1-|\gamma|^2}),s)$ is stationary w.r.t. $\gamma$ when
\begin{equation}\label{contribgam}
\frac{\gamma}{\sqrt{1-|\gamma|^2}}\Gamma_0(y_*(s),\tilde\alpha,s)+\sqrt{1-|\gamma|^2}(\nabla_{\gamma}\tilde\alpha) \partial_{\tilde\alpha}\Gamma_0(y_*(s),\tilde\alpha,s)=0.
\end{equation}
As $\tilde\alpha=\tilde\alpha(\beta,\tau\sqrt{1-|\gamma|^2})$ depends on $\gamma$ through $\sqrt{1-|\gamma|^2}$, we obtain, using \eqref{eqtalphabet}, that
\begin{equation}
\sqrt{1-|\gamma|^2}\nabla_{\gamma}\tilde\alpha=-\frac{\gamma}{\sqrt{1-|\gamma|^2}}\big(\tilde \tau \partial_{\tilde\tau}\alpha(\beta,\tilde\tau)\big)|_{\tilde\tau=\tau\sqrt{1-|\gamma|^2}}
=-\frac{\gamma}{\sqrt{1-|\gamma|^2}}\Big(\frac{2}{3}2^{-1/3}\beta\tilde\tau^{-2/3}
+O(\beta^2\tilde\tau^{-4/3}) \Big).
\end{equation}
Using \eqref{defGamm} we also have, with $\tilde \tau=\tau\sqrt{1-|\gamma|^2}$,
\[
\partial_{\tilde\alpha}\Gamma_0(y_*(s),\tilde\alpha(\beta,\tilde \tau),s)=y_*(s)-\frac{(1-\tilde\alpha(\beta,\tilde\tau))}{2\sqrt{s^2-1}}(1+O(1-\tilde\alpha(\beta,\tilde\tau))).
\] 
Therefore, the sum of the left hand side terms in \eqref{contribgam} equals
\begin{multline}
\frac{\gamma}{\sqrt{1-|\gamma|^2}}\Big[2\sqrt{s^2-1}+2y_*(s)(1-\beta\tilde\tau ^{-2/3}-\frac 23 \beta\tilde\tau^{-2/3})+O(\beta^2\tilde\tau^{-4/3}/s)\Big]\Big|_{\tilde\tau=\tau\sqrt{1-|\gamma|^2}}\\
=\frac{\gamma}{\sqrt{1-|\gamma|^2}}\Big[2\sqrt{s^2-1}+2y_*(s)\big(1-\frac 53 \beta(\tau\sqrt{1-|\gamma|^2})^{-2/3}\big)+O(\beta^2(\tau\sqrt{1-|\gamma|^2})^{-4/3}/s)\Big],
\end{multline}
where we used $O(y_*(s);\frac{1}{\sqrt{s^2-1}})=O(\frac 1s)$ as $s$ is large and $y_*(s)=\arcsin(\frac 1s)$. 
It follows that the phase is stationary w.r.t. $\gamma$ only at $\gamma=0$. As the matrix of second order derivatives equals the identity at the critical point, its Jacobian is equal to $1$. The stationary phase w.r.t. $\gamma$ yields a factor $(\tau\sqrt{s^2-1})^{-d_2/2}$ (where the powers of $\sqrt{s^2-1}$ come from the main term of \eqref{defGamm}). The new symbols, still denoted $f_{1,2}$, are asymptotic expansions and the main contribution of $f_1$ is still $\Sigma_{\#}(0,\cdot)$, hence elliptic.
\end{proof}
  
We are left with the integration in $\beta$ in \eqref{Fudiez2chi00}. Our goal is to prove that this integral doesn't vanish and is uniformly bounded from above and from below by positive constants independent of $\tau$. We let 
$s= h^{-1/3}/\delta$ for some $\delta$ to be chosen later, then $y_*(s)=\arcsin(1/s)=1/s(1+O(1/s^2)) \sim h^{1/3}\delta$ and therefore $\tau^{1/3}y_*(s)=(h\tau)^{1/3} \delta (1+O(\delta^2h^{2/3}))$ on the support of $\chi(h\tau)$. We prove that the main contribution of \eqref{Fudiez2chi00} comes from values $\beta\sim 1$ and that this contribution doesn't vanish. Using \eqref{defGamm} and \eqref{eqtalphabet}, we have
\begin{equation}\label{phaGam2}
-2\tau \Gamma_0(y_*(s),\alpha(\beta,\tau),s)=-2\tau(y_*(s)+\sqrt{s^2-1})+2^{2/3}y_*(s)\tau^{1/3}\beta-2^{-2/3}\frac{\tau^{-1/3}}{\sqrt{s^2-1}}\beta^2(1+O(\tau^{-2/3}\beta)).
\end{equation}
As $y_*(s)=\arcsin(\frac 1s)$, the part of the phase $\Gamma_0(y_*(s),\alpha(\beta,\tau),s)$ depending on $\beta$ reads as 
\[
\delta\left(2^{2/3}(h\tau)^{1/3}\beta-2^{-2/3}h^{2/3}(h\tau)^{-1/3}\beta^2\right)+\delta^2O(h^{2/3}\beta), \quad \text{ with } h\tau\in [1/2,2] \text{ on supp }\chi(h\tau).
\]

\begin{lemma}\label{lemIdelta}
Let $\frac{1}{2h}\leq \tau\leq \frac{2}{h}$ on the support of $\chi(h\tau)$, $h\in (0,h_0)$, let $\delta>0$ and set
\[
I(\delta,\tau):=\int e^{i\delta\left(2^{2/3}(h\tau)^{1/3}\beta-2^{-2/3}h^{2/3}(h\tau)^{-1/3}\beta^2\right)}\kappa_{2\e_1}(\beta\tau^{-2/3})\frac{( f_1A(-\beta)+\tau^{-1/3} f_2A'(-\beta))}{A_+(-\beta)}d\beta.
\]
Then $I(\delta,\tau)$ is rapidly decreasing as $\delta\rightarrow \infty$. For $\delta \sim 1$, the main contribution of $\chi(h\tau)I(\delta,\tau)$ is
\begin{equation}\label{macon}
\int e^{i\delta\left(2^{2/3}(h\tau)^{1/3}\beta-2^{-2/3}h^{2/3}(h\tau)^{-1/3}\beta^2\right)} \kappa_{1}(\beta/8)\frac{( f_1A(-\beta)+\tau^{-1/3} f_2 A'(-\beta))}{A_+(-\beta)}d\beta,
\end{equation}
which is the Fourier transform of a smooth and compactly supported function at $\delta 2^{2/3}(h\tau)^{1/3}\sim 1$.

Moreover, for $\delta\sim 1$, and any $\epsilon>0$, we have
\[
\int e^{i\delta\left(2^{2/3}(h\tau)^{1/3}\beta-2^{-2/3}h^{2/3}(h\tau)^{-1/3}\beta^2\right)}(1-\kappa_1(\beta\tau^{-\epsilon}))\kappa_{2\e_1}(\beta\tau^{-2/3})\frac{(f_1A(-\beta)+\tau^{-1/3} f_2A'(-\beta))}{A_+(-\beta)}d\beta=O(h^{\infty}),
\]
while for $\chi_1\in C^{\infty}_0((1/4,4))$, equal to $1$ near $1$ we have, for all $4\leq 2^{2j}\lesssim h^{-\epsilon}$ and for all $N\geq 1$,
\[
\int e^{i\delta\left(2^{2/3}(h\tau)^{1/3}\beta-2^{-2/3}h^{2/3}(h\tau)^{-1/3}\beta^2\right)}\chi_1(\beta 2^{-2j})\frac{( f_1A(-\beta)+\tau^{-1/3} f_2A'(-\beta))}{A_+(-\beta)}d\beta=O(2^{-jN}).
\]
On the other hand, for $\delta\ll 1$, the main contribution of $I(\delta,\tau)$ comes from values $|\beta|\lesssim 1/\delta$ and $|I(\delta,\tau)|\lesssim 1/\delta$.
\end{lemma}

\begin{proof}
If $|\beta|$ is large and $\beta<-2$, then quotient $\frac{A(-\beta)}{A_+(-\beta)}\sim_{1/|\beta|} e^{-\frac 43(-\beta)^{3/2}}(1+O(1/|\beta|^{3/2}))$ is exponentially decreasing with $(-\beta)$. If $\beta>2$, this quotient has oscillatory behavior (see Lemma \ref{lem:Phi+}) and as $A(-\beta)=e^{i\pi/3}A_+(-\beta)+e^{-i\pi/3}A_-(-\beta)$, then $\frac{A(-\beta)}{A_+(-\beta)}=e^{i\pi/3}+e^{-i\pi/3}e^{+\frac 43i\mu(\beta)}$ with $\mu$ given in \eqref{eq:Phi+}. 
Let $\delta \sim 1$ and $\beta=2^{2j}\tilde\beta<2\e_1\tau^{2/3}$ with $\tilde \beta\in [1/4,4]$. The critical points of the phase  of $I(\delta,\tau)$ satisfy 
\begin{equation}\label{phastatdelt}
\delta 2^{2j}\left(2^{2/3}(h\tau)^{1/3}-2^{1/3}h^{2/3}(h\tau)^{-1/3}2^{2j}\beta\right)\mp2^{3j}\sqrt{\beta}+2^{3j}\sqrt{\beta}=0.
\end{equation}
As $h^{2/3}2^{2j}<4\e_1$, taking $\e_1$ smaller if necessary, it follows that, if $\delta 2^{2j}>\tau^{\epsilon}$ for some $\epsilon>0$, the contribution of the integral for $\tilde\beta\in [1/2,2]$ is $O(\tau^{-\infty})$, while for $2^{2j}\leq \tau^{\epsilon}$ such that $\delta 2^{2j-1}\geq 4$, this contribution is $O(2^{-2jN})$ after $N\geq 1$ integrations by parts.
Let $\delta\ll 1$, then the phase is non-stationary for values $\beta\gtrsim 2^{2j}$ satisfying $\delta 2^{2j}\gg 1$, which proves the Lemma.
\end{proof}
We now proceed with the proof of \eqref{Fourtransglancing}.
From Lemma \ref{lemIdelta} it follows that for bounded values of $\delta$ the main contribution of $\chi(h\tau)I(\delta,\tau)$ is \eqref{macon}, which takes non vanishing values for $\delta$ near $1/2$, hence the integral in the second line of \eqref{Fudiez2chi00} reads as
$
e^{-i\tau(2\sqrt{s^2-1}+2y_*(s))} f_0(y_*(s)\tau^{1/3},s,\tau)$, where, for $y_*(s)\tau^{1/3}\sim \delta\sim 1$, the function $f_0$ doesn't vanish. At $t=2\sqrt{s^2-1}+2y_*(s)$ 
the integral in $\tau$ defining $u^{\#}_{h,\e_0}(Q_S(s),Q_N(s),t)$ satisfies
\begin{multline}\label{calculfin}
\Big|u^{\#}_{h,\kappa_{\e_0}}(Q_S(s),Q_N(s),t)\Big|=\Big| \int e^{it\tau} \chi(h\tau)\mathcal{F}(u^{\#}_{h,\kappa_{\e_0}})(Q_S(s),Q_N(s),\tau)d\tau\Big|\\
=\frac{1}{\sqrt{s^2-1}^{d_1-1+\frac{d_2}{2}}}\Big|\int e^{i\tau(t-2\sqrt{s^2-1}-2y_*(s))}\chi(h\tau)\tau^{d-2+\frac 13-\frac{d_2}{2}}f_0(y_*(s)\tau^{1/3},s,\tau)d\tau\Big|\\
\sim \frac{1}{\sqrt{s^2-1}^{\frac{d_1+d_2-1}{2}+\frac{d_1-1}{2}}}\frac{1}{h^{1+d-2+\frac 13-\frac{d_2}{2}}}\sim \frac{1}{h^d}\times\Big(\frac ht\Big)^{\frac{d-1}{2}}\times\frac{h^{-\frac{d-1}{2}+\frac 23+\frac{d_2}{2}}}{s^\frac{d_1-1}{2}}\\
\sim  \frac{1}{h^d}\times\Big(\frac ht\Big)^{\frac{d-1}{2}}\times h^{-\frac{d-1}{2}+\frac 23+\frac{d_2}{2}+\frac{d_1-1}{6}},
 \end{multline}
which allows to conclude as $-\frac{d-1}{2}+\frac 23+\frac{d_2}{2}+\frac{d_1-1}{6}=-\frac{d_1-3}{3}$. 
For large $\delta\gg 1$, $I(\delta,\tau)$ decreases rapidly by Lemma \ref{lemIdelta} while for $\delta\ll 1$, i.e. for larger values of $s =h^{-1/3}/\delta \gg h^{-1/3}$, $|f_0(y_*(s)\tau^{1/3},s,\tau)|\lesssim 1/\delta$ and the last factor in the third line of \eqref{calculfin} becomes smaller with $\delta$ for  $d_1\geq 3$ as it is bounded by $(1/\delta)\times h^{-\frac{d_1-1}{3}}\delta^{\frac{d_1-1}{2}}= h^{-\frac{d_1-1}{3}}\delta^{\frac{d_1-3}{2}}\ll h^{-\frac{d_1-3}{3}}$. Therefore, for $s\sim h^{-1/3}\sim \tau^{1/3}$, the loss is sharp.
\end{proof}

\section{Proof of Theorem \ref{thmdisp3D} for the wave equation}\label{secdispext3D}
 In this section we provide a proof of Theorem \ref{thmdisp3D}, in the exterior of a ball in $\R^3$. We deal separately with the high frequency regime, and the low frequency one. In fact, to investigate exterior problems one needs to take into account the magnitude of $\tau |\Theta|$, where $|\Theta|$ denotes the size of the obstacle. More specifically, the set of $\tau$ such that $\tau |\Theta|\lesssim 1$ is the low frequency regime (which, in scattering problems divides into the Rayleigh region $\tau|\Theta|\ll 1$ and the resonance region $\tau |\Theta|\sim1$) 
 and the set of values $\tau$ such that $\tau|\Theta|\gg 1$ is the high frequency regime. Mathematical methods used to study scattering phenomena in the Rayleigh or resonance region differ sharply from those used in the high frequency regime.\\

In sections \ref{secffb}, \ref{secofoc}, \ref{secctb} we deal with the high-frequency regime when $\tau$ is large, which is by far the most intricate situation. The low frequency regime is dealt with in section \ref{secHelm}, using classical results for the exterior Dirichlet problem for Helmholtz equation. In section \ref{secfsop} below, we first show that we may reduce the analysis to $t\sim d(Q_0,\partial\Omega)+d(Q,\partial\Omega)$, where $Q_0$ and $Q$ are the source and the observation points, respectively. 

\subsection{Finite speed of propagation. Let $d=3$, $\Omega=\mathbb{R}^3\setminus \Theta$, $\Theta$ non-trappping (here $\Theta=B_3(0,1)$)}\label{secfsop}

\begin{dfn}\label{dfnnontrap}
A domain $\Omega$ is said to be non-trapping if for some $R>0$ such that $|P|<R$ for every $P\in\partial\Omega$ there exists $T_R$ such that no generalized geodesic of length $T_R$ lies completely within the ball $B_R=\{Q\in\Omega | |Q|\leq R\}$. $T_R$ is called escape time.
\end{dfn}
\begin{thm}\label{thmmel79}(\cite{mel79})
Let $\Omega$ be a non-trapping domain in $\mathbb{R}^d$ with $d\geq 3$ odd, let $\Delta_\Omega$ denote the Dirichlet Laplace operator on $\Omega$ and let $R$, $T_R$ as in Definition \ref{dfnnontrap}. Then there exists a sequence $\lambda_j\in\mathbb{C}$, $\Im (\lambda_j)<0$, $(\Im(\lambda_j))_{j\rightarrow\infty}\searrow -\infty$ and associated generalized eigenspaces $V_j\subset C^{\infty}(\Omega)$ of dimensions $m_j<\infty$ so that
\begin{enumerate}
\item $v\in V_j\Rightarrow v|_{\partial\Omega}=0$;\quad  $(\Delta_\Omega+\lambda^2_j)V_j\subset V_j$, $(\Delta_\Omega+\lambda^2_j)^{m_j-1}V_j=\{0\}$;
\item if $(U_0,U_1)$ are supported in $\{\vert Q\vert \leq R\}$, there exists $v_{j,k}\in V_j$,  such that for all $\epsilon>0$, $N\in\mathbb{N}$ and multi-index $\alpha$ and for some constant $C=C(R,N,\epsilon,\alpha)$, the solution $U$ to \eqref{WE} with Dirichlet condition on $\partial\Omega$ and initial data $(U_0,U_1)$ satisfies for $t>T_R$:
\[
\sup _{\{|Q|\leq R\}}\Big|D^{\alpha}_{(t,Q)}\Big[U(Q,t)-\sum_{j=1}^{N-1} e^{-i\lambda_j t}\sum_{k=0}^{m_j-1}t^kv_{j,k}(Q)\Big]\Big|\leq C e^{-(t-T_R)(|\Im \lambda_N|-\epsilon)}.
\]
\end{enumerate}
\end{thm}
Let $u$ solve \eqref{WEOC} with data $(u_0,u_1)=(\delta_{Q_0},0)$ for some $Q_0\in \Omega$. Then, according to \eqref{Gbarform}, the extension $\underline{u}$ defined in \eqref{defUbar} has the form $\underline{u}|_{t>0}=u^+_{free}|_{t>0}-u^{\#}(Q,Q_0,t)$, where $u^{\#}(Q,Q_0,t):=\square^{-1}_{+}\Big((\partial_{x} u^+)_{\partial\Omega}\Big) (Q,Q_0,t)$ is defined in \eqref{Gbarform} (with $d=d_1=3$). The integral in \eqref{vhform} is supported in $t\geq \text{dist}(Q,\partial\Omega)+\text{dist}(Q_0,\partial\Omega)$ and by Huygens principle in odd dimension, one has $u^{\#}(Q,Q_0,t)\vert_{\partial\Omega}=0$ for 
$t>\text{dist}(Q_0,\partial\Omega)+\text{diam}(\Theta)$. We define $T_0=\text{dist}(Q_0,\partial\Omega)+\text{diam} (\Theta)+1$. 
One has 
\begin{equation}\label{supportUU0}
\text{support}(u^{\#}(\cdot,Q_0,T_0))\cup \text{support}(\partial_t u^{\#}(\cdot,Q_0,T_0))\subset \{Q, \text{dist}(Q,\partial\Omega)\leq \text{diam}(\Theta)+1\}.
\end{equation}
Let $R>0$ such that $\{Q, \text{dist}(Q,\partial\Omega)\leq \text{diam}(\Theta)+1\}\subset B_3(0,R)$. For $t\geq T_0$, one has 
$\partial_{x} u^+(.,t)\vert_{\partial\Omega}= \partial_{x} u^{\#}(.,t)\vert_{\partial\Omega}$ and $u^{\#}(Q,Q_0,t)|_{\partial\Omega}=0$.
We apply Theorem \ref{thmmel79} to $u^{\#}(Q,Q_0,t)$ whose initial data at time $T_0$ satisfies \eqref{supportUU0} : as such, for every multi-index $\alpha$, there exists $C_{\alpha}$ such that 
\[
\sup_{P\in \partial\Omega}|\partial ^{\alpha}_{(t,Q)}u^+(Q,Q_0,t)|\Big|_{Q=P}\lesssim e^{-C_{\alpha}(t-(T_R+T_0))},\quad \forall t>T_R+T_0.
\]
It follows that $\partial_{x} u^+(P,Q_0,t)$ is smooth for 
$t>T_R+T_0$ and there exists a constant $C>0$ such that
\[
\sup_{P\in \partial\Omega}|\partial_{x} u^+(P,Q_0,t)|\lesssim e^{-C(t-(T_R+T_0))},\quad \forall t>T_R+T_0.
\]
This shows that, in order to prove Theorem \ref{thmdisp3D} for the wave equation, it is sufficient to prove fixed time bounds for $t$ such that %
$t-|Q-P|\leq T_0+T_R=\text{dist}(Q_0,\partial\Omega)+\text{diam}(\Theta)+T_R+1$. Here, for $R$ as above,
the escape time $T_R$ satisfies $T_R\leq 2R+\text{diam}(\Theta)$.
By finite speed of propagation of the wave flow, %
we must have $t-|Q-P|\geq \text{dist}(Q_0,\partial\Omega)$, as otherwise the contribution of $\partial_{x} u^+(P,Q_0,t-|Q-P|)\Big|_{P\in \partial\Omega}$ is trivial. We have obtained the following :
\begin{lemma}\label{lemtimelocaliz}
There exists $C_0=C(\Omega)>0$ independent of $Q,Q_0$, such that in order to prove dispersive bounds for 
the solution to \eqref{WEOC} (with $d=3$), it is enough to consider only $t$ such that 
\begin{equation}\label{restrictionQ}
d(Q_0,\partial\Omega)+d(Q,\partial\Omega)\leq t\leq d(Q_0,\partial\Omega)+d(Q,\partial\Omega)+C_0.
\end{equation}
\end{lemma}
\begin{rmq}
As a consequence of Lemma \ref{lemtimelocaliz}, if $|t|$ is large, the source and observation points $Q_0$ and $Q$ cannot be very close to $\partial\Omega$ at the same time. Hence, if $Q_0$ is very close to $\partial\Omega$, i.e. if $\text{dist}(Q_0,\partial\Omega)\ll 1$, by symmetry of the Green function we may replace $Q_0$ by $Q$ : we are therefore left to consider only data $Q_0$ such that $\text{dist}(Q_0,\partial\Omega)>c>0$ for some constant $c$, in which case the Melrose and Taylor parametrix holds. 
\end{rmq}
Due to the rotational symmetry of the ball $B_3(0,1)$, we can assume that the source point is of the form $Q_0=Q_N(s)$, for some $s>1$. 
We consider separately the cases :
\begin{itemize}
\item when the source and observation points are outside a fixed neighborhood of the boundary (section \ref{secffb}), we use the parametrix obtained in Section \ref{secgeneral}, for which we obtain dispersive bounds;

\item when the source is far and the observation point is close to the boundary (section \ref{secofoc}), we directly estimate the Melrose -Taylor parametrix which gives the form of the solution near a glancing point;
\item when both the source and the observation points are close to the boundary (section \ref{secctb}) we obtain a parametrix in terms of spherical harmonics and then proceed with the dispersive bounds.
\end{itemize}

\subsection{The case $dist(Q_0,B_3(0,1))\geq  dist(Q,B_3(0,1))\geq c>0$ for some fixed $c>0$}\label{secffb}
In this section we use all the previous notations and results, but in the case $d_1=d=3$ and $d_2=0$, so all the terms depending on $z,\gamma$ will be removed. The incoming wave is $w_{in}$ as in \eqref{winform} with $\phi(x,\tilde y, 0,s)$ instead of $\phi(x,\tilde y,\tilde z,s)$ (without the $d\tilde z$, $d\gamma$ integration). Proposition \ref{propudiezh} yields the form of $u^{\#}_h(Q,Q_0,t)$ where $\partial\Omega=\mathbb{S}^{2}$ in \eqref{vhform} and $\Sigma_3=\frac{i}{4\pi}$. We split the integral \eqref{vhform} by introducing smooth cut-offs $\kappa_{\e_0}(y-y_*(s))$ and $1-\kappa_{\e_0}(y-y_*(s))$ supported for $|y-y_*(s)|\leq \e_0$ and $|y-y_*(s)|\geq \e_0/2$, where $\e_0>0$ is a small parameter as in Definition \ref{dfnkappa} and obtain two contributions as in \eqref{vhformnewchi}, still denoted $u^{\#}_{\kappa_{\e_0}}$ and $u^{\#}_{1-\kappa_{\e_0}}$. Lemma \ref{lemtrans} holds for $u^{\#}_{1-\kappa_{\e_0}}(Q,Q_0,t)$ and $d=3$, for all $Q$ and yield usual dispersive bounds (corresponding to transverse waves). 
We are left with $u^{\#}_{\kappa_{\e_0}}(Q,Q_0,t)$ whose Fourier transform in time reads as in \eqref{Fudiezzer}, with $Q$ instead of $Q_S(s)$ and for $d=3$, $d_2=0$, as follows
\begin{equation}\label{Fudiezzerd3}
\mathcal{F}(u^{\#}_{h,\kappa_{\e_0}})(Q,Q_N(s),\tau) =\frac{i}{4\pi} \chi(h\tau)\int_{P=(1,y,\omega)\in\mathbb{S}^{2}}\mathcal{F}(\partial_x u^+)(y,Q_0,\tau)\frac{\kappa_{\e_0}(y-y_*(s))}{|Q-P|}e^{-i\tau|Q-P|}dy d\omega.
\end{equation}
For $Q\in \mathbb{R}^3\setminus B_3(0,1)$ with coordinates $(r,y_Q,\omega_Q)$ and for a boundary point $P=(1,y,\omega)$, we have  
\begin{equation}\label{defpsi}
|P-Q|=\Big(1+r^2-2r\cos y\cos y_Q\cos(\omega-\omega_Q)-2r\sin y\sin y_Q\Big)^{1/2}:=\psi(y,\omega,Q).
\end{equation}
\begin{prop}\label{propI0}
Let $I_{\kappa_{\e_0}}(Q,Q_0,\tau):=\mathcal{F}(u^{\#}_{\kappa_{\e_0}})(Q,Q_0,\tau)/\tau$, then there exists $C=C(\e_0,c)$ such that for all $Q,Q_0$ with $\text{dist}(Q_0,\partial\Omega)\geq \text{dist}(Q,\partial\Omega)\geq c$ and for $t$ as in Lemma \ref{lemtimelocaliz}, the following holds
\[
\Big|\int e^{it\tau}\chi(h\tau)\tau I_{\kappa_{\e_0}}(Q,Q_0,\tau)d\tau\Big|\leq \frac{C(\e_0,c)}{h^2 t}.
\]
\end{prop}
Theorem \ref{thmdisp3D} follows immediately from Proposition \ref{propI0} when $Q_0,Q$ stay outside a small neighborhood of the boundary. In the remaining of this section we prove Proposition \ref{propI0}. We are reduced to showing that there exists a uniform constant $C(\e_0,c)>0$ (independent of $Q,Q_0,t$), such that for $\tau>1$ large enough and for all $t\in [T_*(Q,Q_0),T_*(Q,Q_0)+C_0]$ with $T_*(Q,Q_0):=\text{dist}(Q,B_3(0,1))+\text{dist}(Q_0,B_3(0,1))$ and $C_0$ depending only on the diameter of the ball $B_3(0,1)$, the following holds:
\begin{equation}\label{boundI0horsOS}
|I_{\kappa_{\e_0}}(Q,Q_0,\tau)|\leq \frac{C(\e_0,c)}{t}\quad \text{for  } t\in [T_*(Q,Q_0),T_*(Q,Q_0)+C_0].
\end{equation}
As $\mathcal{F}(\partial_x u^+)(P,Q_0,\tau)$ is independent of $\omega$ by construction, the part of the phase of $I_0$ that may depend on this parameter is $\psi$ defined in \eqref{defpsi}. When $\cos y_Q\lesssim \tau^{-1+\epsilon}$ for some small $\epsilon>1$, this dependence is very weak and we cannot take advantage of the stationary phase in $\omega$. On the other hand, when $\cos y_Q\geq \tau^{-1+\epsilon}$, i.e. when $Q$ is located outside a cone of aperture $\tau^{-1+\epsilon}$ around the $OS$ (or $ON$) axis, we can immediately eliminate $\omega$ by stationary phase. We distinguish two situations :\\\

\paragraph{\bf $Q$ is outside a conic neighborhood of the $OS$ axis of aperture $\tau^{-1+\epsilon}$ with small $\epsilon>0$}\label{secoutside}
Let $\tau^{\epsilon}\leq \tau\cos y_Q$ for some $\epsilon\in (0,1/6)$, which corresponds to $y_Q+\frac{\pi}{2}\gtrsim \tau^{-1+\epsilon}$. The phase $\psi$ has critical points $\{\omega_Q, \omega_Q+\pi\}$. At $\omega_Q+\pi$, the phase of $I_{\e_0}$ is not stationary in $y$ for $\alpha$ near $1$ as $-\partial_y\psi(y,\omega_Q+\pi,Q)\sim 1$ and we obtain a $O(\tau^{-\infty}/t)$ contribution. As we have
\begin{equation}\label{secderivomomtildepsi}
\partial^2_{\omega}\psi(y,\omega,Q)|_{\omega=\omega_Q}= \frac{r\cos y\cos y_Q}{\psi(y,\omega_Q,Q)},\quad \tau\cos(y_Q)\geq \tau^{\epsilon}
\end{equation}
and as $\cos y=\cos(y_*(s)+(y-y_*(s)))= \cos(\arcsin(\frac 1s))+O(\e_0)= \frac{\sqrt{s^2-1}}{s}+O(\e_0)$ remains bounded from below by a fixed constant on the support of the symbol for $|y-y_*(s)|\leq \e_0$ small enough, it follows that the usual stationary phase applies and yields, modulo $O(\tau^{-\infty})$ contributions 
\begin{equation}\label{I0tosplit}
I_{\kappa_{\e_0}}(Q,Q_0,\tau)=\frac{i}{4\pi \tau} \int 
\frac{\kappa_{\e_0}(y-y_*(s))\Sigma_0}{\sqrt{\cos y (\tau \cos y_Q)}}\frac{e^{-i\tau \psi(y,\omega_Q,Q)}}{\sqrt{r\psi(y,
\omega_Q,Q)}} \mathcal{F}(\partial_x u^+)(y,Q_0,\tau) dy,
\end{equation}
where $\psi(y,\omega_Q,Q)=(1+r^2-2r\cos(y_Q-y))^{1/2}=\phi(0,\frac{\pi}{2}+y_Q-y,0,r)$ and $\Sigma_0(y,(\tau \cos(y_Q))^{-1})$ is an asymptotic expansion with main contribution $\sim 1$. 
As $r>1$, this phase has a degenerate critical point when  $\cos (y-y_Q)=\frac 1r$ and $\partial_y\psi(y,\omega_Q,Q)|_{\cos(y-y_Q)=\frac 1r}=1$, $\partial^3_y\psi(y,\omega_Q,Q)|_{\cos(y-y_Q)=\frac 1r}=-1$.
 \begin{rmq}
The integral \eqref{I0tosplit} is taken over $P=(0,y,\omega_Q)$ for $y$ near $y_*(s)=\arcsin\frac 1s$, hence on a small path of the great circle going through $(0,y_*(s),\omega_Q)$. The ray $P Q$ is tangent to the boundary $\mathbb{S}^2$ at $(y,\omega_Q)$ if and only if $\cos(y_Q-y)=\frac 1r$, hence at $\frac{\pi}{2}+y_Q-y=y_*(r)=\arcsin(1/r)$, while the ray $Q_0P$ is tangent to the boundary $\mathbb{S}^2$ when $\cos(y-y_{Q_0})=\frac 1s$, hence at $\frac{\pi}{2}+y-y_{Q_0}=y_*(s)=\arcsin(1/s)$. As $y_{Q_0}=\frac{\pi}{2}$ and $y<\frac{\pi}{2}$, this means at $y=y_*(s)$. It follows that that the ray $Q_0Q$ is tangent to $\mathbb{S}^2$ when $\cos(y_*(s)-y_Q)=\frac 1r$ i.e. when $y_*(s)=y_Q+\arccos(1/r)=(y_Q+\frac{\pi}{2})-\arcsin\frac 1r=(y_Q+\frac{\pi}{2})-y_*(r)=:y_c$. 
\end{rmq}
The phase of $\mathcal{F}(\partial_x u^+)(P,Q_0,\tau)$ from \eqref{u_xomega23} is linear in $y$, and the phase function of \eqref{I0tosplit} is stationary when $\alpha=\partial_y\psi(y,\omega_Q,Q)$. For $Q$ such that $|y_Q+\arccos \frac 1r-y_*(s)|=|y_c-y_*(s)|\geq 2\e_0$ and $y$ in the support of $\kappa_{\e_0}(y-y_*(s))$ we have $|\partial^2\psi(y,\cdot)|=|y-y_c|(1+O(|y-y_c|))\gtrsim |y_c-y_*(s)|-|y-y_*(s)|\geq 2\e_0-\e_0>\e_0$ and the usual stationary phase applies in $y$ yielding a contribution better than the one at $y_c$. Hence we consider $Q$ with $y_Q+\arccos \frac 1r$ near $y_*(s)$ and use again Lemma \ref{lemgam} near the degenerate critical point $y_c$.
As 
\[
y\alpha-\psi(y,\omega_Q,Q)=(\frac{\pi}{2}+y_Q)\alpha-\Big((\frac{\pi}{2}+y_Q-y)\alpha+\phi(0,\frac{\pi}{2}+y_Q-y,0,r)\Big),
\] 
applying Lemma \ref{lemgam} with $\frac{\pi}{2}+y_Q-y=y_*(r)+Y$, $|Y|<2\e_0$, provides a change of variable $Y\rightarrow\sigma$ such that
\begin{equation}\label{phapsio}
y\alpha-\psi(y,\omega_Q,Q)=(\frac{\pi}{2}+y_Q)\alpha-\frac{\sigma^3}{3}-\sigma\zeta_0(\alpha)-\Gamma_0(y_*(r),\alpha,r), \quad y_*(r)=\arcsin(\frac 1r).
\end{equation}
Introducing \eqref{u_xomega23} in \eqref{I0tosplit}, using \eqref{phapsio} and \cite[Thm. 7.7.18]{horFIO} (as in Theorem \ref{thmhorFIO}) gives, modulo $O(\tau^{-\infty})$,
\begin{multline}\label{I0form}
I_{\kappa_{\e_0}}(Q,Q_0,\tau)=\frac{e^{i\pi/6}}{8\pi^2}
\frac{\tau^{\frac{4}{3}}}{\sqrt{s^2-1}}\int \frac{\kappa_{\e_0}(y-y_*)\Sigma_0}{\sqrt{\cos y (\tau \cos y_Q)}}\frac{e^{-i\tau \psi(y,\omega_Q,Q)}}{\sqrt{r\psi(y,
\omega_Q,Q)}}\\
\times \int e^{i\tau (y\alpha-\Gamma_0(y_*(s),\alpha,s))} 
\kappa_{\e_1}(1-\alpha)\frac{ b_{\partial}(y,\alpha,\tau) f(\alpha,s,\tau)}{A_+(\tau^{2/3}\zeta_0(\alpha))}d\alpha dy\\
=\frac{e^{i\pi/6}}{8\pi^2}
\frac{\tau^{\frac{4}{3}-\frac 13}}{\sqrt{s^2-1}\sqrt{r^2-1}}\frac{\kappa_{2\e_0}(y_c-y_*)}{\sqrt{\tau\cos y_Q}}\int e^{-i\tau(\Gamma_0(y_*(s),\alpha,s)+\Gamma(y_*(r),\alpha,r))}\\
\times \kappa_{\e_1}(1-\alpha)\frac{(f_1A(\tau^{2/3}\zeta_0(\alpha))+\tau^{-1/3}f_2 A'(\tau^{2/3}\zeta_0(\alpha)))}{A_+(\tau^{2/3}\zeta_0(\alpha))}d\alpha,
\end{multline}
where $f_1, f_2$ are smooth asymptotic expansions, $f_1$ has main contribution $\frac{b_{\partial} f}{\sqrt{\cos y}}\times \frac{\sqrt{r^2-1}}{\sqrt{r\psi(y,\omega_Q,Q)}}(\sim 1)$ and where the factor $\tau^{-\frac 13}$ in the second equality in \eqref{I0form} arises from integration in $\sigma$. \\

Let $\beta:=-\tau^{2/3}\zeta_0(\alpha)$, then $\alpha=\alpha(\beta,\tau)$ as in \eqref{eqtalphabet}, with $|d\alpha/d\beta|\sim 2^{-1/3}\tau^{-2/3}$ and we have
\begin{multline}\label{sumGam}
\left(\frac{\pi}{2}+y_Q\right)\alpha-\Big(\Gamma_0(y_*(s),\alpha,s)+\Gamma_0(y_*(r),\alpha,r)\Big)\\=(y_c-y_*(s))\alpha-\sqrt{s^2-1}-\sqrt{r^2-1}-\frac{(1-\alpha)^2}{2}\Big(\frac 1r(1+O(1-\alpha))+\frac 1s(1+O(1-\alpha))\Big)\\
=(y_c-y_*(s))-\sqrt{s^2-1}-\sqrt{r^2-1}\\
-2^{-1/3}(y_c-y_*(s))\tau^{-2/3}\beta-\frac{\tau^{-4/3}\beta^2}{2^{5/3}}\Big(\frac 1r(1+O(\tau^{-2/3}\beta))+\frac 1s(1+O(\tau^{-2/3}\beta))\Big).
\end{multline}
We can now bound the integral \eqref{I0form} and prove \eqref{boundI0horsOS}.
Let $\chi_1\in C^{\infty}((\frac 12, 2))$ equal to $1$ near $1$, so that
\[
\kappa_{2\e_1}(\tau^{2/3}\beta)=\kappa_1(\beta/2)+\sum_{2\leq 2j\leq \log_2(2\e_1\tau^{2/3})}\chi_1(\beta 2^{-2j}).
\] 
\begin{lemma}
Denote by $I_{\kappa_{\e_0},\kappa_1}$ and $I_{\kappa_{\e_0},j}$ the integral \eqref{I0form} with cut-off $\kappa_1(\beta/2)$ and $\chi_1(\beta 2^{-2j})$, respectively. 
There exists a uniform constant $C>0$ such that 
\[
|I_{\kappa_{\e_0},\kappa_1}(Q,Q_0,\tau)|\leq \frac Ct, \quad \sum_{2\leq 2j\leq \log_2(2\e_1\tau^{2/3})} |I_{\kappa_{\e_0},j}(Q,Q_0,\tau)|\leq \frac Ct.
\]
\end{lemma}
\begin{proof}
On the support of $\kappa_1(\beta/2)$ the Airy factors in \eqref{I0form} behave like symbols and the phase $i\tau \times $ \eqref{sumGam} of the integral \eqref{I0form} depending on $\beta$ equals $i(2^{-1/3}\tau^{1/3}(y_c-y_*(s))\beta+O(\tau^{-1/3}))$. If $\tau^{1/3}|y_c-y_*(s)|\geq \tau^{\e/2}$ with $\epsilon$ as in \eqref{secderivomomtildepsi}, integrations by parts yield $O(\tau^{-\infty})$. If $\tau^{1/3}|y_c-y_*(s)|\leq \tau^{\epsilon/2}$, then
\begin{equation}\label{intkappa1}
\frac{t}{rs}\sim \frac{\sqrt{s^2-1}+\sqrt{r^2-1}}{rs}\leq \arcsin\Big(\frac{\sqrt{s^2-1}+\sqrt{r^2-1}}{rs}\Big)=y_Q+\frac{\pi}{2}-(y_c-y_*(s))\leq  \cos(y_Q)+|y_c-y_*(s)|,
\end{equation}
hence, for $\sqrt{s^2-1}\geq \sqrt{r^2-1}\geq c$, $t\sim \sqrt{s^2-1}+\sqrt{r^2-1}$, $\tau^{1/3}|y_c-y_*(s)|\leq \tau^{\epsilon/2}$, $\sqrt{\tau\cos(y_Q)}\geq \tau^{\epsilon/2}$,
\[
|I_{\kappa_{\e_0},\kappa_1}(Q,Q_0,\tau)|\lesssim \frac{\tau^{\frac 43-\frac 13-\frac 23}}{rs}\frac{1}{\sqrt{\tau\cos(y_Q)}} \leq\frac{\tau^{1/3}}{\sqrt{\tau\cos(y_Q)}}\times \frac{(\cos y_Q+|y_c-y_*(s)|)}{t}\leq \frac 1t.
\]
Consider now the integral $I_{\kappa_{\e_0},j}$ where $\beta\sim 2^{2j}\leq 2\e_1\tau^{2/3}$ and let $\beta=2^{2j}\tilde \beta$, with $\tilde\beta\sim 1$ on the support of $\chi_1$. The Airy terms in the last line of \eqref{I0form} start to oscillate and, using $A(-\beta)=\sum_{\pm}e^{\pm i\pi/3}A_{\pm}(-\beta)$, we decompose $I_{\kappa_{\e_0},j}=\sum_{\pm}I^{\pm}_{\kappa_{\e_0},j}$ where $I^{\pm}_{\kappa_{\e_0},j}$ has  phase $\frac 23i \beta^{3/2}\mp\frac 23i \beta^{3/2}$, and symbol $\sim f_1$, so we must distinguish two situations. Notice that  $y_c-y_*(s)$ is not signed as
\[
y_c-y_*(s)=y_Q+\frac{\pi}{2}-\arcsin \frac 1r-\arcsin \frac1s=\cos y_Q(1+O(\cos y_Q))-\arcsin(\frac{\sqrt{s^2-1}+\sqrt{r^2-1}}{rs})
\] 
and may vanish when $QQ_0$ is tangent to $B_3(0,1)$.
The phase $i\tau\times$\eqref{sumGam} with $\beta=2^{2j}\tilde \beta$, $\tilde \beta\sim 1$ becomes
\begin{equation}\label{phaAi1}
-2^{-1/3}\tau^{1/3}(y_c-y_*(s))2^{2j}\tilde\beta - 2^{-5/3}\tau^{-1/3}2^{4j}\tilde\beta^2\Big(\frac 1r(1+O(2^{2j}\tau^{-2/3}\tilde\beta))+\frac 1s(1+O(2^{2j}\tau^{-2/3}\tilde\beta))\Big),
\end{equation}
which has a critical point on the support of $\chi_1(\tilde\beta)$ when $0<y_*(s)-y_c\sim \tau^{-2/3} 2^{2j+1}(\frac 1r+\frac 1s)$.
There are at most three values of $j$, denoted $j_*,j_*\pm 1$ for which we may have such a critical point; for all $j\notin \{j_*,j_*\pm 1\}$, the phase is non stationary. As $t\sim \sqrt{s^2-1}+\sqrt{r^2-1}\leq rs \arcsin(\frac{\sqrt{s^2-1}+\sqrt{r^2-1}}{rs})=rs(\arcsin (1/s)+\arcsin(1/r))$, we then have $\frac{t}{rs}\lesssim \cos y_Q$.
If $j*$ is such that $2^{4j_*}\tau^{-1/3}/r\gg 1$, the stationary phase applies in $\tilde\beta$ and gives
\[
|I^+_{\kappa_{\e_0},j_*(\pm 1)}(Q,Q_0,\tau)|\lesssim\frac{\tau^{1/3}}{rs\sqrt{\tau\cos y_Q}}\times 2^{2j_*}\big(\frac{\tau^{1/3}}{2^{4j_*}}\big)^{1/2}=\frac{1}{rs\sqrt{\cos y_Q}}\lesssim \frac{\sqrt{\cos y_Q}}{t} \leq \frac 1t,
\]
where we have used $\frac{1}{rs}\lesssim \frac{\cos y_Q}{t}$. If $2^{4j_*}\tau^{-1/3}/r\lesssim 1$, then, as $s\geq r$ and $t\in  [s,2s]$, we have
\[
|I^+_{\kappa_{\e_0},j_*(\pm 1)}(Q,Q_0,\tau)|\lesssim\frac{\tau^{1/3}}{rs\sqrt{\tau\cos y_Q}} 2^{2j_*}\lesssim \frac{\tau^{1/3}}{rs\sqrt{\tau\cos y_Q}} \sqrt{\tau^{1/3}r}\sim \frac 1s  \frac{1}{\sqrt{r\cos y_Q}}\leq\frac 1s \sqrt{\frac st}\sim \frac 1t.
\]
When the phase of \eqref{I0form} is $i\tau\times$\eqref{sumGam}$+\frac 43 \beta^{3/2}$, taking $\beta=2^{2j}\tilde \beta$, $\tilde \beta\sim 1$, it becomes \eqref{phaAi1}$+\frac 43 2^{3j}\tilde\beta^{3/2}$ and the last term dominates the second one in \eqref{phaAi1} as $2^{3j}\gg \tau^{-1/3}2^{4j}/r$ as $2^{2j}\leq\e\tau^{2/3}$ and $r>1$, hence this phase is stationary when $0<y_c-y_*(s)\sim \tau^{-1/3}2^{j_*}$ and we find again $\frac{t}{rs}\lesssim \cos y_Q$ and, when such  $j_*$ does exist, the stationary phase with large parameter $2^{3j_*}$ gives
\[
|I^-_{\kappa_{\e_0},j_*(\pm 1)}(Q,Q_0,\tau)|\lesssim\frac{\tau^{1/3}}{rs\sqrt{\tau\cos y_Q}}\times 2^{2j_*-\frac{3j_*}{2}}\lesssim \frac{\cos y_Q}{t} \times \frac{\tau^{1/3-1/2}}{\sqrt{\cos y_Q}}\times 2^{j_*/2}\leq\frac{\sqrt{\cos y_Q}}{t}\times \big(\frac{2^{j_*}}{\tau^{1/3}}\big)^{1/2}\ll \frac 1t.
\]
\end{proof}
The last lemma completes the proof of \eqref{boundI0horsOS} when $\tau \cos y_Q>\tau^{\e}$.

\paragraph{\bf $Q$ belongs to a conic neighborhood of the $OS$ axis of aperture $\tau^{-1+\e}$ with small $\e\in (0,1/6)$} When $\tau|\cos y_Q|\leq \tau^{\epsilon}$ for small $\e>0$, the stationary phase in $\omega_Q$ cannot be used anymore. This situation corresponds to observation points $Q$ very close to the $OS$ (or $ON$) axis. If $y_Q$ is near $\frac{\pi}{2}$, i.e. $Q$ is near $ON$, then the phase of \eqref{Fudiezzerd3} is non stationary in $y$ and the contribution of the integral is $O(h^{\infty})$. Let $y_Q$ near $-\frac{\pi}{2}$. We have $\psi(y,\omega, Q_S(r))=(1+r^2+2r\sin y)^{1/2}=\phi(0,-y,0,r)$ and we let
\[
e(y,\tau):=\int e^{-i\tau(\psi(y,\omega,Q)-\phi(0,-y,0,r))}\frac{\phi(0,y_*(r),0,r)}{\psi(y,\omega,Q)}d\omega, \quad \phi(0,y_*(r),0,r)=\sqrt{r^2-1}.
\]
As $\psi(y,\omega,Q)=\phi(0,-y,0,r)+O(\tau^{-1+\epsilon})$, then $|\partial^j_ye(y,\tau)|\leq \tau^{j\epsilon}$ for all $j\geq 1$ and we have
\begin{equation}\label{I0tosplitnew}
I_{\kappa_{\e_0}}(Q,Q_0,\tau)=\frac{i}{4\pi \tau} \int 
\kappa_{\e_0}(y-y_*(s))\frac{e^{-i\tau \phi(0,-y,0,r)}}{\phi(0,y_*(r),0,r)}e(y,\tau) \mathcal{F}(\partial_x u^+)(y,Q_0,\tau) dy,
\end{equation}
 with $\mathcal{F}(\partial_x u^+)(y,Q_0,\tau)$ given in \eqref{u_xomega23}. We now transform $I_{\kappa_{\e_0}}$ into an Airy type integral as before.
 \begin{lemma}
 For $0<\e<1/6$, we have, modulo $O(\tau^{-\infty})$,
 \begin{multline}\label{i0formnew}
I_{\kappa_{\e_0}}(Q,Q_0,\tau)=\frac{e^{i\pi/6}}{8\pi^2}
\frac{\tau^{\frac{4}{3}-\frac 13}}{\sqrt{s^2-1}\sqrt{r^2-1}}\kappa_{2\e_0}(-(y_*(r)+y_*(s)))\int e^{-i\tau(\Gamma_0(y_*(s),\alpha,s)+\Gamma_0(y_*(r),\alpha,r))}\\
\times \kappa_{\e_1}(1-\alpha)\frac{(f_1A(\tau^{2/3}\zeta_0(\alpha))+\tau^{-1/3}f_2 A'(\tau^{2/3}\zeta_0(\alpha)))}{A_+(\tau^{2/3}\zeta_0(\alpha))}d\alpha,
\end{multline}
where the symbols $f_{1,2}$ are asymptotic expansions with small parameter $\tau^{-1+6\e}$ and where $f_1$ is elliptic.
 \end{lemma}
 \begin{proof}
As before, the phase $\phi(0,-y,0,r)$ has a unique degenerate critical point at $-y=y_*(r)=\arcsin \frac 1r$ so the phase of \eqref{I0tosplitnew} reads as $y\alpha-\phi(0,-y,0,s)=-\frac{\sigma^3}{3}-\sigma\zeta_0(\alpha)-\Gamma_0(y_*(r),\alpha,s)$. 
We use Malgrange preparation theorem  for $C^{\infty}$ functions \cite[Thm.7.5.4]{horFIO} to transform the integral in $\sigma$ into a sum of Airy functions : the symbol of $I_{\kappa_{\e_0}}$ depending of $y$ equals $e(y,\tau)b_{\partial}(y,\alpha,\tau)$; making the change of variables $y\rightarrow \sigma $, letting $g(\sigma,\alpha,\tau):=e(y(\sigma),\tau)b_{\partial}(y(\sigma),\alpha,\tau)\frac{dy}{d\sigma}$ and applying \cite[Thm.7.5.4]{horFIO} to $g$ gives, using \cite[(7.5.14)]{horFIO},
\[
g=q(\sigma,\alpha,\tau)(\sigma^2+\zeta_0(\alpha))+\sigma r_2(\alpha,\tau)+r_1(\alpha,\tau),
\]
where $|\partial^j_{\sigma}q|\leq c_j\int (|g|+|g^{(5+j)}|)d\sigma\lesssim \tau^{6\e}$; here $r_1$ and $r_2$ represent the main contributions of $f_1$ and $f_2$, respectively. One integration by parts in $\sigma$ yields 
\[
\int e^{-i\tau(\sigma^3/3+\sigma \zeta_0(\alpha))} q(\sigma,\alpha,\tau)(\sigma^2+\zeta_0(\alpha))d\sigma=\frac{1}{i\tau}\int e^{-i\tau(\sigma^3/3+\sigma \zeta_0(\alpha))} \partial_{\sigma} q(\sigma,\alpha,\tau) d\sigma
\]
and applying again \cite[Thm.7.5.4]{horFIO} to $\partial_{\sigma} q$ this time allows to obtain $r_2^1(\alpha,\tau)$, $r_1^1(\alpha,\tau)$ such that $ \partial_{\sigma} q=q_1(\sigma^2+\zeta_0(\alpha))+\sigma r_2^1+r_1^1$, where now $|\partial^j_\sigma q_1|\leq c_j |\partial^{j+5}(\partial q)| \leq c^2_j |\partial^{j+6+5}g|$, hence $|\partial_{\sigma}q_1|\lesssim \tau^{12\e}$. In the same way we make repeated integrations by parts writing at each step, by Malgrange,
 \[
 \partial_{\sigma}q_k=q_{k+1}(\sigma,\alpha,\tau)(\sigma^2+\zeta_0(\alpha))+\sigma r_2^k(\alpha,\tau)+r_1^k(\alpha,\tau),
 \]
 where $|\partial^j_{\sigma}q_{k+1}|\leq c_j |\partial^{5+j}_{\sigma}(\partial_{\sigma} q_k)|$ which gives, by induction, $|\partial_{\sigma}q_{k+1}|\leq C^k_1 |\partial^{6(k+1)}_{\sigma}g|\leq C_1^k \tau^{6(k+1)\e}$ and where the new constants $C_1^k$ depend on $c_j$, $j\leq 6(k+1)$. As, at each step, the integrations by parts yield a factor $\tau^{-1}$, it follows that, for $6\e<1$, writing $f_j:=r_j+\sum_{k\geq 1} \tau^{-k}r_j^k(\alpha)$ yields \eqref{i0formnew}.
\end{proof}

The difference with the last part of the proof of Theorem \ref{thmCE} is that instead of \eqref{phaGam2} we now have 
\begin{equation}\label{phaGam2sr}
\Gamma_0(y_*(s),\alpha,s)+\Gamma_0(y_*(r),\alpha,r)=\sqrt{s^2-1}+\sqrt{r^2-1}+(y_*(s)+y_*(r)) \alpha +\frac{(1-\alpha)^2}{2}\Big(\frac{1+O(1-\alpha)}{\sqrt{s^2-1}}+\frac{1+O(1-\alpha)}{\sqrt{r^2-1}}\Big),
\end{equation}
hence instead of \eqref{phaGam2}, the phase of \eqref{i0formnew} becomes, for $\beta=\tau^{2/3}\zeta_0(\alpha)$,
 \begin{multline}\label{phaI0CEnew}
-\tau\big(\sqrt{s^2-1}+\sqrt{r^2-1}+y_*(s)+y_*(r)\big)\\
+2^{-1/3}\tau^{1/3}(y_*(s)+y_*(r))\beta-2^{-5/3}\beta^2\Big(\frac{\tau^{-1/3}}{\sqrt{s^2-1}}(1+O(\tau^{-2/3}\beta))+\frac{\tau^{-1/3}}{\sqrt{r^2-1}}(1+O(\tau^{-2/3}\beta))\Big).
 \end{multline}
The main advantage over the case of section \ref{secoutside} is that $y_*(s)+y_*(r)=\arcsin(\frac{\sqrt{s^2-1}+\sqrt{r^2-1})}{rs})\sim \frac{t}{rs}$ is strictly positive. If $|\beta|>2$ is large, then the Airy factor in \eqref{i0formnew} is either exponentially decreasing (for $\beta<-2$) or oscillatory (for $\beta>2$), in which case, after decomposing it as before as a sum of terms with phases $0$ and $+\frac 43i\beta^{4/3}$, we obtain two non stationary phase functions, hence a contribution $|\beta|^{-N}$ for all $N\geq 1$. When $|\beta|$ is bounded the Airy factor in \eqref{i0formnew} is a symbol, and we distinguish two cases: 
\begin{itemize}
\item either $\tau^{1/3}(y_*(s)+y_*(r))\sim \tau^{1/3}t/(rs)\lesssim 1$, in which case $\tau^{1/3}/(rs)\lesssim 1/t$ and 
\[
|I_0(Q,Q,\tau)|\lesssim \frac{\tau^{\frac 43-\frac 13-\frac 23}}{\sqrt{s^2-1}\sqrt{r^2-1}}\lesssim \frac 1t,
\]
\item or $\tau^{1/3}(y_*(s)+y_*(r))\sim \tau^{1/3}t/(rs)\sim M\gg 1$, when 
\[
|I_0(Q,Q,\tau)|\lesssim \frac{\tau^{\frac 43-\frac 13-\frac 23}}{\sqrt{s^2-1}\sqrt{r^2-1}}\times M^{-N}\sim (M/t)\times M^{-N},
\]
for all $N\geq 1$, which follows by repeated integrations by parts in $\beta$.
\end{itemize}
This allows to conclude that \eqref{boundI0horsOS} holds and to achieve the proof of Proposition \ref{propI0}.

\subsection{The case $dist(Q_0,B_3(0,1))>c\geq  dist(Q,B_3(0,1))$ }\label{secofoc}

When the source point $Q_0$ is not too close to the boundary, the Melrose and Taylor parametrix from Proposition \ref{MT} can be used for $Q$ near the glancing regime. In this section we consider the situation $\text{dist}(Q_0,\partial\Omega)>c\geq \text{dist}(Q,\partial\Omega)$ for some small, fixed $c>0$, independent of $h$. If $Q_0$ is very close to the boundary and $Q$ is far from a small neighborhood of it, using the symmetry of the Green function with respect to $Q$ and $Q_0$ allows to replace $Q$ by $Q_0$ and reduce the analysis to our case.
Let $Q_0=Q_N(s)$ as before, for some $s>1+c$ and let $Q=(x,y,\omega)\in \Omega$ with $0\leq x\leq c$. Let $u$ solve \eqref{WEOC} and let $u^+:=1_{t>0}u$ as in Section \ref{secgeneral}. By Lemma \ref{lemtimelocaliz}, we are reduced to obtaining dispersive bounds for $u^+(Q,Q_0,t)$ for $t$ such that $\text{dist}(Q_0,\partial\Omega)+\text{dist}(Q,\partial\Omega)\leq t\leq \text{dist}(Q_0,\partial\Omega)+\text{dist}(Q,\partial\Omega)+C_0$, where $C_0$ depends only on the diameter of the obstacle. Write the Fourier transform in time of $u^+$, $\mathcal{F}(u^+)(Q,Q_0,\tau)$, for $Q=(x,y,\omega)$ near $\mathcal{C}^{3,0}_{Q_0}$, modulo a smoothing operator, under the form
\begin{multline}\label{formu+h}
\mathcal{F}(u^+)(Q,Q_0,\tau)=\frac{\tau}{2\pi}\int  e^{i\tau y\alpha} \Big(aA(\tau^{2/3}\zeta(x,\alpha))+b\tau^{-1/3}A'(\tau^{2/3}\zeta(x,\alpha))\Big)\hat{F}_{\tau}(\tau\alpha) d\alpha\\
-\frac{\tau}{2\pi}\int  e^{i\tau y\alpha} \Big(aA_+(\tau^{2/3}\zeta(x,\alpha))+ b \tau^{-1/3}A'_+(\tau^{2/3}\zeta(x,\alpha))\Big)\frac{A(\tau^{2/3}\zeta_0(\alpha))}{A_+(\tau^{2/3}\zeta_0(\alpha))}\hat{F}_{\tau}(\tau\alpha) d\alpha,
\end{multline}
where $\hat{F}_{\tau}$ is defined by \eqref{FourierformFball} (with $d_1=d=3$).
The first integral of \eqref{formu+h} is $w_{in}(Q,Q_0,\tau)$ for which the usual dispersion estimate holds, hence it remains to handle the second line of \eqref{formu+h}. 
\begin{lemma}
There exists a constant $C>0$ independent of $Q_0,Q$ such that, for all $Q=(x,y,\omega)$ in a small, fixed neighborhood of $\mathcal{C}^{3,0}_{Q_0}$ such that $\text{dist}(Q,\partial\Omega)\leq c$, $x\geq 0$, $\omega\in [0,2\pi]$, and for all $t$ such that $\text{dist}(Q_0,\partial\Omega)+\text{dist}(Q,\partial\Omega)\leq t\leq \text{dist}(Q_0,\partial\Omega)+\text{dist}(Q,\partial\Omega)+C_0$, the following holds
\begin{equation}\label{toestimu+h}
\Big|\frac{1}{2\pi}\int  e^{i\tau y\alpha} \Big(aA_+(\tau^{2/3}\zeta(x,\alpha))+ b \tau^{-1/3}A'_+(\tau^{2/3}\zeta(x,\alpha))\Big)\frac{A(\tau^{2/3}\zeta_0(\alpha))}{A_+(\tau^{2/3}\zeta_0(\alpha))}\hat{F}_{\tau}(\tau\alpha) d\alpha\Big|\leq \frac{C}{t}.
\end{equation}
Moreover, for $\chi\in C^{\infty}_0((\frac 12,2))$, $h\in (0,1)$ and for $Q$ and $t$ as above, $|\chi(hD_t)u^+(Q,Q_0,t)|\lesssim \frac{1}{h^2t}$.
\end{lemma}
\begin{proof}
The second statement of the lemma follows from \eqref{toestimu+h} and \eqref{formu+h}. %
As $\text{dist}(Q,\partial\Omega)\leq c$ is small, we let $s\leq t\leq s+C_0+c$. Introducing \eqref{FourierformFball} in
\eqref{toestimu+h}, the last inequality becomes  
\begin{multline}\label{lefttoprove}
\Big|\frac{e^{i\tau (y-y_*(s)-\sqrt{s^2-1})}}{2\pi}\int e^{-i\tau((y-y_*(s))(1-\alpha)+\frac{(1-\alpha)^2}{2\sqrt{s^2-1}}(1+O(1-\alpha)))}\kappa_{\e_1}(1-\alpha) f(\alpha,s,\tau)\\
\times \tau^{2/3}\frac{A(\tau^{2/3}\zeta_0(\alpha))}{A_+(\tau^{2/3}\zeta_0(\alpha))} \Big(aA_+(\tau^{2/3}\zeta(x,\alpha))+b\tau^{-1/3}A_+'(\tau^{2/3}\zeta(x,\alpha))\Big)d\alpha\Big|\leq C\frac{\sqrt{s^2-1}}{t},
\end{multline}
where $\kappa_{\e_1}$ and $f$ are as in Lemma \ref{lemmaF}.
For $t\in [s,s+C_0+c],$ the quotient $\frac{\sqrt{s^2-1}}{t}$ is always larger than a constant depending only on $c$: indeed, if $1+c<s$ is bounded, then so is $t$ and $\frac{\sqrt{s^2-1}}{t}\gtrsim c$ ; if $s$ is large then $\sqrt{s^2-1}\sim s\sim t$ and $\frac{\sqrt{s^2-1}}{t}\gtrsim 1$. As $c$ has been fixed, it will be enough to prove that the integral in \eqref{lefttoprove} is bounded by a constant independent of $(x,y)$. Let $-\tau^{2/3}\zeta_0(\alpha)=\beta$ as before, then $|\beta|\leq 2\e_1 \tau^{2/3}$ on the support of $\kappa_{\e_1}$ and we write $\alpha=\alpha(\beta,\tau)$ as in \eqref{eqtalphabet}.
Let also
\[
\tilde a(x,y,\beta,\tau):= f(\alpha(\beta,\tau),s,\tau)\Big(a+b\tau^{-1/3}\frac{A'+}{A_+}(\tau^{2/3}\zeta(x,\alpha(\beta,\tau)))\Big)\times (\tau^{2/3}\frac{d\alpha}{d\beta}),
\]
with $\frac{A'_+}{A_+}$ defined \eqref{eq:Phi+}, then $\tilde a $ is a symbol of order $0$ w.r.t $\beta$. The integral in \eqref{lefttoprove} becomes 
\begin{equation}\label{intJchichij}
\int e^{-i\tau^{1/3}((y-y_*(s))\beta+\frac{\tau^{-2/3}\beta^2}{2\sqrt{s^2-1}}(1+O(\tau^{-2/3}\beta)))}\kappa_{\e_1}(1-\alpha(\beta,\tau))\tilde a(x\,y,\alpha(\beta),\tau) A_+(\tau^{2/3}\zeta(x,\alpha(\beta)))\frac{A(-\beta)}{A_+(-\beta)} d\beta.
\end{equation}
As $\zeta(x,\alpha)=\zeta_0(\alpha)+x\partial_x\zeta(0,\alpha)+O(x^2)$ then $\tau^{2/3}\zeta(x,\alpha(\beta))=-(\beta+(\tau^{2/3}x)(-\partial_x\zeta(0,\alpha)+O(x)))$, where $-\partial_x\zeta(0,\alpha)=2^{1/3}$ and $0\leq x\leq c$. 
For $|\beta|\lesssim 1$, all the factors are bounded and we conclude. For $\beta\ll -1$, the Airy factor $A(-\beta)$ is exponentially decreasing and $|A_+(\tau^{2/3}\zeta(x,\alpha(\beta)))/A_+(-\beta)|\leq 1$ and we conclude. We are left with the case $1<\beta\leq 2\e_1\tau^{2/3}$, when all the Airy terms have oscillatory behavior. As $A(-\beta)/A_+(-\beta)=e^{i\pi/3}+e^{-i\pi/3}e^{2i\mu(\beta)}$, the integral \eqref{intJchichij} may be split in two parts with phase functions given by, respectively,
\begin{multline}\label{phaseJpmk}
\varphi_{+}(\beta,x,y,\tau)=-\tau^{1/3}(y-y_*(s))\beta-\frac{\tau^{-1/3}\beta^2}{2\sqrt{s^2-1}}(1+O(\tau^{-2/3}\beta)) - \mu(\tau^{2/3}\zeta(x,\alpha(\beta))),\\
\varphi_{-}(\beta,x,y,\tau)=-\tau^{1/3}(y-y_*(s))\beta-\frac{\tau^{-1/3}\beta^2}{2\sqrt{s^2-1}}(1+O(\tau^{-2/3}\beta)) - \mu(\tau^{2/3}\zeta(x,\alpha(\beta)))+2\mu(\beta),
\end{multline}
where $\mu(w)= \frac 23(-w)^{3/2}(1+O((-w)^{-3/2}))$ has been defined in \eqref{eq:Phi+}. In both cases the symbol equals
\begin{equation}\label{symbJk}
e^{\pm i\pi/3}
\kappa_{\e_1}(1-\alpha(\beta,\tau))\tilde a(x,y,\beta,\tau)\Sigma(\tau^{2/3}\zeta(x,\alpha(\beta))),
\end{equation}
where $\Sigma(w)= \frac{(-w)^{-1/4}}{2\sqrt{\pi}}(1+O((-w)^{-3/2})$ has been defined in \eqref{eq:Phi+}. 

Let $X:=2^{1/3}\tau^{2/3}x$ and $Y:=\tau^{1/3}(y_*(s)-y)$ where $|y-y_*(s)|<1$ is small, hence $|Y| < \tau^{1/3}$, $0\leq X \leq 2^{1/3} \tau^{2/3}c$. With these notations, the critical points of the phase $\varphi_{+}(\beta,x,y,\tau)$ satisfy
\[
\sqrt{2}(X+\beta)^{1/2}(1+O(\tau^{-2/3}(X+\beta)))=Y-\frac{\tau^{-1/3}\beta}{\sqrt{s^2-1}}(1+O(\tau^{-2/3}\beta)).
\]
There exists a unique solution on the support of the symbol, denoted $\beta_c$, which is a non-degenerate critical point. Write $\beta=\beta_cw$ : for $w\notin [1/4,4]$ the phase is non stationary, while for $w\sim 1$, the stationary phase applies as $\partial^2_{w}\varphi_{+}(\beta_cw,x,y,\tau)|_{w=1}\sim\frac{\beta_c^2}{(\beta_c+X)^{1/2}}$ and yields a bound for the corresponding part of \eqref{intJchichij} of the form
\[
  {\beta_c}{(X+\beta_c)^{-1/4}}\times 
  {(\partial^2_{w}\varphi^{+}(\beta_c,x,y,\tau))^{-1/2}}\sim 1,
\]
where the factor $\beta_c$ comes from the change of variable $\beta\rightarrow w$ and the factor $(X+\beta_c)^{1/4}$ comes from the form of $\Sigma$. 
The phase $\varphi_{-}(\beta,x,y,\tau)$ is stationary when
\[
\sqrt{2}(X+\beta)^{1/2}(1+O(\tau^{-2/3}(X+\beta)))=Y-\frac{\tau^{-1/3}\beta}{\sqrt{s^2-1}}(1+O(\tau^{-2/3}\beta)) +2\sqrt{\beta}(1+O(\tau^{-2/3}\beta)),
\]
which may have up to two solutions $\sqrt{\beta^{\pm}_{c}}(1+O(\sqrt{\beta^{\pm}_c}\tau^{-1/3}))=-Y \pm \sqrt{Y^2/2+X}$ depending on the position of the observation point $Q$ : if $Y>0$ or if $-1\leq Y\leq 0$ then only $\beta^{+}_c$ may be such that $\beta_c^{+}>1$ when $X>Y^2/2+2Y+1$. 
When $-Y>1$, there may be two solutions, corresponding to $\pm$ signs, when $X\leq Y^2/2+2Y+1$ : to have coalescence, i.e. $\beta^-_c=\beta^+_c$, we must have $Y^2/2+X=0$, which doesn't hold when $-Y>1$. Hence for $\beta>1$ there are no degenerate critical points. We achieve the proof of \eqref{toestimu+h} exactly as in the previous case, setting $\beta=\beta_c^{\pm}w$, and applying, near each $\beta_c^{\pm}$ the usual stationary phase with $|\partial^2_{w}\varphi_{-}(\beta^{\pm}_c,x,y,\tau)|\gtrsim (\beta_c^{\pm})^{3/2}$, which yields a bound for the corresponding part of \eqref{intJchichij} of the form
\[
\sum_{\pm}\frac{\beta^{\pm}_c}{(X+\beta^{\pm}_c)^{1/4}}\times \frac{1}{\sqrt{\partial^2_{w}\varphi^{-}(\beta^{\pm}_c,x,y,\tau)}}\lesssim 1.
\]
\end{proof}

\subsection{The case $c>dist(Q_0,B_3(0,1)), dist(Q,B_3(0,1))$}\label{secctb}
In this section we use the Laplacian in spherical coordinates \eqref{laplaballrS2} (with $d=d_1=3$ and $d_2=0$) to construct the solution to the wave equation in terms of zonal functions as the arguments from the previous sections no longer apply (especially when $s-1\leq \tau^{-2/3}$). We denote by $\mathcal{R}(Q,Q_0,\tau)$ the outgoing solution to $(\tau^2+\Delta)w=\delta_{Q_0}$, $w|_{\partial\Omega}=0$, then the solution $u(Q,Q_0,t)$ to \eqref{WE} with initial data $(u_0,u_1)=(\delta_{Q_0},0)$ reads as
\begin{equation}\label{flotlinwave1}
u(Q,Q_0,t)=\int_0^{+\infty} e^{it\tau}  \mathcal{R}(Q,Q_0,\tau)\frac{d\tau}{\pi}.
\end{equation}
In the following we obtain an explicit representation for $\mathcal{R}(Q,Q_0,\tau)$ in terms of spherical harmonics and then prove dispersive bounds for \eqref{flotlinwave1} when $Q$ and $Q_0$ are close to $\partial\Omega$. First, we recall some well known some facts about the eigenfunctions on the sphere $\mathbb{S}^2$. In the spherical coordinates $(r,\varphi,\omega)$ introduced in \eqref{polarcoord} and for $d=3$, the Laplace operator has the form $\Delta=\frac{\partial^2}{\partial r^2}+\frac{2}{r}\frac{\partial}{\partial r}+\frac{1}{r^2}\Delta_{\mathbb{S}^2}$, where $\Delta_{\mathbb{S}^2}$ denotes the Laplacian on the unit sphere $\mathbb{S}^2$. The eigenvalues $(\mu_m)_{m\geq 0}$ of $-\Delta_{\mathbb{S}^2}$ are $\mu^2_m=m(m+1)$, for all $m\geq 0$ and have multiplicity $2m+1$.
The corresponding eigenfunctions are usual spherical harmonics, denoted $(Y_{m,j})_{m\geq 0, j\in\{1,..,2m+1\}}$, which form an $L^2(\mathbb{S}^2)$ - orthonormal basis of the eigenspaces of $-\Delta_{\mathbb{S}^2}$. For $m\geq 0$, consider the zonal functions
\begin{equation}\label{zonefct}
Z^m_{(\varphi,\omega)}((\varphi',\omega')): =\sum_{j=1}^{2m+1}Y_{m,j}(\varphi,\omega)\overline{Y_{m,j}(\varphi',\omega')}
=\frac{2m+1}{4\pi}P_m \Big(\cos(\sphericalangle ((\varphi,\omega),(\varphi',\omega')) \Big),
\end{equation}
where $P_m$ is the Legendre polynomial of degree $m$ and where $\sphericalangle ((\varphi,\omega),(\varphi',\omega'))$ is the angle between two vectors $(r,\varphi,\omega)$ and $(r',\varphi',\omega')$, given by $\cos(\sphericalangle (\varphi,\omega),(\varphi',\omega')) =\cos\varphi\cos\varphi'+\sin\varphi\sin\varphi'\cos(\omega-\omega')$.
For $m\geq 0$ and $(\varphi,\omega),(\varphi',\omega')\in\mathbb{S}^2$, $|Z^m_{(\varphi,\omega)}((\varphi',\omega'))|\leq \frac{2m+1}{4\pi}$. By separation of variables we obtain : 

\begin{lemma}
The outgoing solution $\mathcal{R}(Q,Q_0,\tau)$ to $(\tau^2+\Delta)w=\delta_{Q_0}$, $w|_{\partial\Omega}=0$ has the following explicit representation :
\begin{equation}\label{formR}
\mathcal{R}(Q,Q_0,\tau)=\sum_{m\in \N}G_{m+1/2}(|Q|,|Q_0|,\tau)Z^m_{\frac{Q}{|Q|}}(\frac{Q_0}{|Q_0|}),
\end{equation}
where $Z^m$ are given in \eqref{zonefct} and where, for $\nu\in \{m+\frac 12, m\in \mathbb{N}\}$,
 for $\tau\in \mathbb{R}$ and $r,s\geq 1$ we set
  \begin{equation}\label{greenfctGnu}
  G_{\nu}(r,s,\tau)=\frac{\pi}{2i\sqrt{rs}}\Big(J_{\nu}(s\tau)-\frac{J_{\nu}(\tau)}{H^{(1)}_{\nu}(\tau)}H^{(1)}_{\nu}(s\tau)\Big)H^{(1)}_{\nu}(r\tau).
  \end{equation}
For $\nu\in\{m+\frac 12,m\in \mathbb{N}\}$, $J_{\nu}$, $H^{(1)}_{\nu}$ denote the Bessel and Hankel functions defined in section \ref{secHank}.
\end{lemma}
Using \eqref{flotlinwave1} and \eqref{formR}, yields $\chi(hD_t)u(Q,Q_0,t)= \int_0^{\infty}e^{it\tau}\chi(h\tau)\tau I(Q,Q_0,\tau)$, where
\begin{equation}\label{allcasessmalldistdefI}
\tau I(Q,Q_0,\tau):= \sum_{\nu=m+1/2, m\in \mathbb{N}}Z^{\nu-1/2}_{\frac{Q}{|Q|}}(\frac{Q_0}{|Q_0|}) G_{\nu}(|Q|,|Q_0|,\tau)\,\text{ for } |Q|\geq |Q_0|,
\end{equation}
with $G_{\nu}(|Q|,|Q_0|,\tau)=G_{\nu}(r,s,\tau)$ defined in \eqref{greenfctGnu}; for $|Q|\leq |Q_0|$ replace $G_{\nu}(|Q|,|Q_0|,\tau)$ by $G_{\nu}(|Q_0|,|Q|,\tau)$ (and notice that for $Q\in\partial\Omega$ we have $r=1$ and $G_{\nu}(1,s,\tau)=0$). We may assume without loss of generality that $|Q|\geq |Q_0|$: in the following we let $1+c>r=|Q|\geq |Q_0|=s\geq 1$. As $r,s$ are small, the behavior of $G_{\nu}(r,s,\tau)$ will  depend on the size of $\frac{\tau}{\nu}$, that may be less than, nearly equal to or larger than $1$: these regimes correspond, respectively, to the transverse, diffractive and elliptic ones. %
\begin{lemma}\label{3.26}
There exists a constant $C>0$ such that, for all $1+c\geq r,s \geq 1$ and all $1\gtrsim t>h$, 
\begin{equation}\label{allcasessmalldist}
\int_0^{\infty}e^{it\tau}\chi(h\tau)\tau I(Q,Q_0,\tau)d\tau\leq \frac{C}{h^2t}, \quad \chi \in C^{\infty}_0((1/2,2)).
\end{equation}
\end{lemma}
\begin{proof}
Notice first that the sum over $m$ involving only $J_{m+1/2}(s\tau)H^{(1)}_{m+1/2}(r\tau)$ from $G_{m+1/2}$ coincides with the incoming wave (this can be seen using the additional formulas for Bessel functions of first and second kind \cite[(10.1.45), (10.1.46)]{Ab}) and therefore satisfies the usual dispersive bounds. 
Let $\chi_0\in C^{\infty}_0((-2,2))$ be valued in $[0,1]$ and equal to $1$ on $[-1,1]$. Define, for some small $1\gg \epsilon>0$,
\begin{equation}\label{defI+pmSH}
I_{\chi_1,\chi_2}(Q,Q_0,\tau):=\tau^{-1}\sum_{\nu-1/2\in \mathbb{N}}\chi_1\Big(\frac{\varepsilon\tau}{\nu}\Big)\chi_2\Big(\frac{\varepsilon\nu}{\tau}\Big)Z^{\nu-1/2}_{\frac{Q}{|Q|}}(\frac{Q_0}{|Q_0|})\frac{\pi}{2i\sqrt{rs}}\frac{J_{\nu}(\tau)}{H^{(1)}_{\nu}(\tau)}H^{(1)}_{\nu}(s\tau)H^{(1)}_{\nu}(r\tau).
\end{equation}
for $\chi_1\in\{\chi_0,1-\chi_0\}$ and $\chi_2\in\{1,\chi_0,1-\chi_0\}$. It will be enough to show that $|I_{\chi_1,\chi_2}(Q,Q_0,\tau)|$ is uniformly bounded by a constant independent of $Q,Q_0,\tau$.
First we deal with $I_{1-\chi_0,1}(Q,Q_0,\tau)$ which correspond to the sum for $\frac{\tau}{\nu}\geq \frac{1}{\varepsilon}\gg 1$. As $r,s>1$, we also have $r\tau/\nu, s\tau/\nu\geq \frac{1}{\epsilon}\gg 1$. 
Using \eqref{Ai}, we obtain
\begin{equation}\label{I+pmSHdetail}
|I_{1-\chi_0,1}(Q,Q_0,\tau)|\leq \tau^{-1}\sum_{\nu-1/2\in \mathbb{N}}(1-\chi_0)\Big(\frac{\varepsilon\tau}{\nu}\Big)\Big|Z^{\nu-1/2}_{\frac{Q}{|Q|}}\big(\frac{Q_0}{|Q_0|}\big)\Big|\frac{1}{\sqrt{rs}}\frac{1}{\sqrt{\pi s\tau}}\frac{1}{\sqrt{\pi r\tau}}\Big|\frac{J_{\nu}(\tau)}{H^{(1)}_{\nu}(\tau)}\Big|.
\end{equation}
As $|Z^{\nu-1/2}_{\frac{Q}{|Q|}}(\frac{Q_0}{|Q_0|})|\leq \frac{\nu}{2\pi}$, the sum over $\nu\leq \varepsilon \tau$ and for $1\leq s\leq r\leq 1+\varepsilon_0$ is bounded by 
\begin{equation}\label{transversesmalldistS1}
|I_{1-\chi_0,1}(Q,Q_0,\tau)|\leq \frac{2\tau^{-2}}{\pi rs}\sum_{1/2\leq \nu\leq \varepsilon\tau}\Big|Z^{\nu-1/2}_{\frac{Q}{|Q|}}\big(\frac{Q_0}{|Q_0|}\big)\Big|
\lesssim \frac{2\tau^{-2}}{\pi rs}\Big|\sum_{1/2\leq \nu\leq \varepsilon \tau}\frac{\nu}{2\pi}\Big|\lesssim \frac{\varepsilon^2}{rs}\leq 1.
\end{equation}
We are left with $I_{\chi_0,1}=I_{\chi_0,\chi_0}+I_{\chi_0,1-\chi_0}$. We start with the main part $I_{\chi_0,\chi_0}$ which corresponds to the sum over $\frac{\varepsilon}{2}\leq \frac{\tau}{\nu}\leq \frac{2}{\varepsilon}$. We use \eqref{Aiiicj}, \eqref{HAiiicj} with $s\tau=\nu\rho_s$, $r\tau=\nu\rho_r$ and $J_{\nu}=\text{Re}(H_{\nu}^{(1)})$, to obtain
\begin{multline}
|I_{\chi_0,\chi_0}(Q,Q_0,\tau)|\lesssim_{\epsilon} \tau^{-1}\sum_{\nu\sim \tau, \nu -1/2\in \mathbb{N}}\nu^{-2/3}\chi_0\Big(\frac{\varepsilon\tau}{\nu}\Big) \chi_0\Big(\frac{\varepsilon\nu}{\tau}\Big)\Big|Z^{\nu-1/2}_{\frac{Q}{|Q|}}\big(\frac{Q_0}{|Q_0|}\big)\Big|\\\times  |A_+(\nu^{2/3}\tilde\zeta(\rho_s))| |A_+(\nu^{2/3}\tilde\zeta(\rho_r))|\Big|\frac{A(\nu^{2/3}\tilde\zeta(\rho_1))}{A_+(\nu^{2/3}\tilde\zeta(\rho_1))}\Big|\\
\lesssim_{\epsilon} \tau^{-1}\sum_{\nu\sim \tau} \frac{\nu^{1/3}}{2\pi}|A_+(\nu^{2/3}\tilde\zeta(\rho_s))| |A_+(\nu^{2/3}\tilde\zeta(\rho_r))|\Big|\frac{A(\nu^{2/3}\tilde\zeta(\rho_1))}{A_+(\nu^{2/3}\tilde\zeta(\rho_1))}\Big|,
\end{multline}
where in this section the sign $\lesssim_{\epsilon}$ means inequality with a factor which depends only on $\epsilon$. Using the Cauchy-Schwartz inequality, in order to bound $I_{\chi_0,\chi_0}$ it will be enough to obtain uniform bounds for 
\begin{equation}\label{sum2chi0}
 \tau^{-1}\sum_{\nu\sim \tau} \frac{\nu^{1/3}}{2\pi}|A^2_+(\nu^{2/3}\tilde\zeta(\rho_s))| \Big|\frac{A(\nu^{2/3}\tilde\zeta(\rho_1))}{A_+(\nu^{2/3}\tilde\zeta(\rho_1))}\Big|.
\end{equation}
If $|\nu^{2/3}\tilde\zeta(\rho_s)|\lesssim 1$, using that $\tilde\zeta(\rho)\sim |\rho-1|$ for $\rho$ near $1$, it follows that  $|\rho_s-1|=|s\tau/\nu-1|\lesssim \nu^{-2/3}$, hence  $|\nu-s\tau|\lesssim \nu^{1/3}\lesssim_\epsilon \tau^{1/3}$, so the sum over $\nu$ contains only $\sim_\epsilon \tau^{1/3}$ terms (instead of $\tau$ terms). 
As we always have $|A(z)/A_+(z)|\leq 2$, it follows that, in this regime, the sum in \eqref{sum2chi0} can be bounded by
\[
2\tau^{-1}\sum_{\nu=s\tau-c_1(\epsilon)\tau^{1/3}}^{s\tau+c_1(\epsilon)\tau^{1/3}}\frac{\nu^{1/3}}{2\pi}\sim_\epsilon \tau^{-1+4/3-2/3}\sim_{\epsilon} \tau^{-1/3},\quad \text{ as } 1\leq s<1+c.
\]
If $\nu^{2/3}\tilde\zeta(\rho_s)> 1$ then $\rho_s<1$ and then the Airy factor $A^2_+(\nu^{2/3}\tilde\zeta(\rho_s))$ is exponentially increasing; at the same time the factor $ \Big|\frac{A(\nu^{2/3}\tilde\zeta(\rho_1))}{A_+(\nu^{2/3}\tilde\zeta(\rho_1))}\Big|$ is exponentially decreasing as, using that $\partial_{\rho}\tilde\zeta|_{\rho=1}=-2^{1/3}$ and $s>1$, we have $\tilde\zeta(\rho_s)\leq \tilde\zeta(\rho_1)$. As a consequence, the Airy product in \eqref{sum2chi0} is exponentially decreasing like $e^{-\frac 43 \nu \tilde\zeta^{3/2}(\rho_1)+\frac 43 \nu \tilde\zeta^{3/2}(\rho_s)}$ which allows to conclude. We are left with the situation when $\nu^{2/3}\tilde\zeta(\rho_s)\leq -1$ when $\rho_s>1+c_0\nu^{-2/3}$ for some $c_0>0$ which implies $\nu< s\tau-c_0(\epsilon/2)^{1/3}\tau^{1/3}$ as $(\epsilon/2) \tau\leq \nu$ on the support of $\chi_0(\frac{\epsilon\tau}{\nu})$. We use $|A(z)/A_+(z)|\leq 2$ and $|A^2_+(\nu^{2/3}\tilde\zeta(\rho_s))|\leq \frac{1}{\sqrt{\nu^{2/3}|\tilde\zeta(\rho_s)|}}$ in order to bound the sum \eqref{sum2chi0} by
\begin{multline}
2 \tau^{-1}\sum_{\nu=\epsilon\tau/2}^{s\tau-c_{\epsilon}\tau^{1/3}} \frac{\nu^{1/3}}{2\pi}\frac{1}{\sqrt{\nu^{2/3}|\tilde\zeta(\rho_s)|}}\lesssim  \frac{1}{\pi\tau}\sum_{\nu=\epsilon\tau/2}^{s\tau-c_{\epsilon}\tau^{1/3}}\frac{1}{\sqrt{|s\tau/\nu-1|}}\leq\frac{1}{\pi\tau}\sum_{\nu=\epsilon\tau}^{s\tau-c_{\epsilon}\tau^{1/3}}\frac{\sqrt{\frac{\nu}{s\tau}}}{\sqrt{1-\frac{\nu}{s\tau}}} \\
\lesssim \frac{1}{\pi}\int_{\epsilon/s}^{1-c_{\epsilon}\tau^{-2/3}/s}\frac{\sqrt{x}}{\sqrt{1-x}}dx=\frac{1}{\pi}\int_{\sqrt{c_{\epsilon} \tau^{-2/3}/s}}^{\sqrt{1-\epsilon^2/s^2}}\frac{\sqrt{1-w^2}}{w} w dw<\frac{1}{\pi},
\end{multline}
where $c_{\epsilon}=c_0(\epsilon/2)^{1/3}$.
In order to bound $I_{\chi_0,1-\chi_0}$ that corresponds to $\nu>\tau/\epsilon$, we use \eqref{Ainularge} as follows
\[
|I_{\chi_0,1-\chi_0}(Q,Q_0,\tau)|\leq \tau^{-1}\sum_{\nu>\tau/\epsilon}\frac{\nu}{2\pi}\times \frac{1}{2\pi\nu}\big(\frac{e\tau}{2\nu}\big)^{2\nu}\big(\frac{er\tau}{2\nu}\big)^{-\nu}\big(\frac{es\tau}{2\nu}\big)^{-\nu}\lesssim \tau^{-1}\sum_{\nu>\tau/\epsilon}\frac{1}{(rs)^{\nu}}
\]
and as (at least) $r>1$, this concludes the proof of Lemma \ref{3.26}.
\end{proof}

\subsection{The small frequencies case}\label{secHelm}
Let $\Theta$ be a compact obstacle in $\mathbb{R}^3$ with smooth boundary of class $C^2$ and let $\Omega=\mathbb{R}^3\setminus \Theta$. Let $\Delta$ be the Laplace operator on $\Omega$ with Dirichlet boundary conditions. In this section the geometry of the domain $\Omega=\mathbb{R}^3\setminus \Theta$ will be of no importance: we only need to assume that $\Theta$ is compact and that $\partial\Theta$ is of class $C^2$. Let $\tilde \chi\in C_{0}^\infty((0,4))$ valued in $[0,1]$, $\tilde\chi |_{[\frac 12,2]}=1$, $\tilde\chi(0)=0$.
\begin{lemma}\label{lemwaveSF}
Let $u(Q,Q_0,t)$ denote the solution to the wave equation \eqref{WE} with $u_0=\delta_{Q_0}$, $u_1=0$ with Dirichlet boundary condition on $\Omega$.
There exists a constant $C>0$ independent of $Q,Q_0$ such that
\begin{equation}\label{dispSFWEtoprove}
\sup_{Q\in \Omega}\Big|\tilde\chi(\sqrt{-\Delta})u(Q,Q_0,t)\Big|\leq \frac{C}{|t|}.
\end{equation}
\end{lemma}
Before the proof of the Lemma \ref{lemwaveSF} we recall some basic properties of solutions to the Helmholtz equation in $\Omega$ with positive wave number $\tau>0$ (see \cite[Chapter 3]{ColKre83}). The Helmholtz equation reads as follows
\begin{equation}\label{resolvantegen}
(\Delta_D+\tau^2)\mathcal{R}(.,\tau)=\delta_{Q_0} \text{ in } \Omega, \quad \mathcal{R}(.,\tau)|_{\partial\Omega}=0,
\end{equation}
satisfying the Sommerfeld radiation condition
\begin{equation}\label{sommercond}
\partial \mathcal{R}(Q,Q_0,\tau)/\partial r-i\tau\mathcal{R}(Q,Q_0,\tau)=O(|Q|^{-1}), \text{ when }|Q|\rightarrow \infty,
\end{equation}
uniformly for all directions $Q/|Q|$. 
The solution $R(.,\tau)$ to \eqref{resolvantegen} has the following form
\begin{equation}\label{Resform}
\mathcal{R}(Q,Q_0,\tau)=E(Q-Q_0,\tau)-\int_{\tilde P\in\partial\Omega}E(\tilde P-Q_0,\tau)g_{\tilde P}(Q,\tau)d\sigma(\tilde P),
\end{equation}
where $E(z, \tau): =\frac{1}{4\pi}\frac{e^{i\tau |z|}}{|z|}$ denotes the fundamental solution of the Helmholtz equation in $\mathbb{R}^3$,
and where, for $\tilde P\in\partial\Omega$, $g_{\tilde P}(Q,\tau)$ is the solution to the exterior Dirichlet problem for the Helmholtz equation 
\begin{equation}\label{extDirprob}
(\Delta+\tau^2)g_{\tilde P}=0 \text{ for } Q\in\Omega,  \quad g_{\tilde P}( P,\tau)\Big|_{ P\in\partial\Omega}=\delta_{\tilde P=P},\quad g \text{ radiating}.
\end{equation}

\begin{prop}(\cite[Theorem 3.21]{ColKre83})
Given a data $g_0$ on the boundary $\partial\Omega$, %
the exterior Dirichlet problem for the Helmholtz equation
\begin{equation}\label{extDirpbgen}
(\Delta+\tau^2)g=0  \text{ in } \mathbb{R}^3\setminus\Theta, \text{ g radiating},
\end{equation}
has an unique solution which satisfies $g|_{\partial\Omega}=g_0$ on the boundary $\partial\Omega$. \end{prop}
The boundary-value problem \eqref{extDirpbgen} can be reduced to an integral equation of the second kind which is uniquely solvable for all wavelengths $\tau>0$ by seeking the solution in the form of combined double and single-layer potentials. We briefly recall this construction: given an integrable function $\varphi$, the single and double layered operators $S$ and $K$ are defined as follows %
\[
S_{\tau}(\varphi)(y):=2\int_{\partial\Omega}E(y-\tilde y,\tau)\varphi(\tilde y)d\sigma(\tilde y), \quad 
K_{\tau}(\varphi)(y):=2\int_{\partial\Omega}\frac{\partial E(y-\tilde y,\tau)}{\partial \nu(\tilde y)}\varphi(\tilde y)d\sigma(\tilde y),\quad y\in \partial\Omega,
\]
where $\nu$ is the unit normal vector to the boundary $\partial\Omega$ directed into the exterior of $\Theta$. In the so-called layer approach, the solution to the exterior Dirichlet problem \eqref{extDirpbgen} is sought (as in \cite[3.49]{ColKre83}) in the form of acoustic surface potentials, with a density $\varphi$, as follows:
\begin{equation}\label{formpotg}
g(z)=\int_{\partial\Omega} \Big(\frac{\partial E(z-y,\tau)}{\partial \nu(y)}-iE(z-y,\tau)\Big)\varphi(y)d\sigma(y),\quad z\in \Omega.
\end{equation}
The potential $g$ given in \eqref{formpotg} in $\Omega=\mathbb{R}^3\setminus\Theta$  solves the exterior Dirichlet problem \eqref{extDirpbgen} provided the density is a solution of the integral equation (as in \cite[(3.51)]{ColKre83})
\begin{equation}\label{inteqgg0}
\varphi(y)+K_{\tau}(\varphi)(y)-i S_{\tau}(\varphi)(y)=2g_0(y).
\end{equation}
\begin{lemma}\label{leminvoper}
If $\partial\Omega$ is of class $C^2$, them for fixed $\tau>0$, the operator $K_{\tau}-iS_{\tau}$ is compact. Moreover $(I+K_{\tau}-i S_{\tau})^{-1}:L^2(\partial\Omega)\rightarrow L^2(\partial\Omega)$ exists and is bounded.
\end{lemma}
\begin{proof}
For $y\in \partial\Omega$, $\frac{\partial E(y-\tilde y,\tau)}{\partial \nu(\tilde y)}=\frac{1}{4\pi}e^{i\tau |y-\tilde y|}\Big[\frac{i\tau}{|y-\tilde y|}\frac{<y-\tilde y, \nu(\tilde y)>}{|y-\tilde y|}-\frac{<y-\tilde y, \nu(\tilde y)>}{|y-\tilde y|^3}\Big]$. As the boundary is of class $C^2$ then $|<y-\tilde y, \nu(\tilde y)>|\lesssim |y-\tilde y|^2$, and therefore the kernel of the operator $K_{\tau}-iS_{\tau}$ is bounded by $\frac{C}{|y-\tilde y|}\in L^1(\partial\Omega)$ for some constant $C>0$. According to \cite[Theorem 3.33]{ColKre83}, the operator $I+K_{\tau}-iS_{\tau}$ is injective by the Riesz theory for operator equations of the second kind (see \cite[Theorem 1.16]{ColKre83}).
\end{proof}

\begin{proof}(of Lemma \ref{lemwaveSF})
Let $t$ be fixed, large enough (as otherwise the estimate becomes trivial). The next lemma rewrites  $\cos(\sqrt{-\Delta})(\delta_{Q_0})$ in terms of the Green function for the Helmholtz equation.
\begin{lemma}\label{lemVBwave}
The wave flow reads $u(Q,Q_0,t)=\int_0^{+\infty} e^{it\tau}  \mathcal{R}(Q,Q_0,\tau)\frac{d\tau}{\pi}$,
where $\mathcal{R}(Q,Q_0,\tau)$ is the outgoing solution to \eqref{resolvantegen}
satisfying the Sommerfeld radiation condition \eqref{sommercond}
uniformly for all $Q/|Q|$. 
\end{lemma} 
The solution $\mathcal{R}(.,\tau)$ to \eqref{resolvantegen} is given by \eqref{Resform}, 
where, for $\tilde P\in\partial\Omega$, $g_{\tilde P}(Q,\tau)$ is the solution to the exterior Dirichlet problem for the Helmholtz equation \eqref{extDirprob}.
The free wave flow $\int_0^{+\infty} e^{it\tau} E(Q-Q_0,\tau)\frac{d\tau}{\pi}$ satisfies the usual dispersion estimate, and we focus on the second term in the formula defining $\mathcal{R}(.,\tau)$. As in \eqref{formpotg}, $g_{\tilde P}(Q,\tau)$ reads as
\begin{equation}\label{gPtildeform}
 g_{\tilde P}(Q,\tau)=\int_{P\in \partial\Omega} \Big(\frac{\partial E(Q- P,\tau)}{\partial \nu(P)}-i E(Q- P,\tau)\Big)\varphi_{\tilde P}( P,\tau)d\sigma(P), \quad Q\in \Omega,
\end{equation}
where the density $\varphi_{\tilde P}(P)$ is a solution to the following integral equation
\begin{equation}\label{phiPPtildeform}
\varphi_{\tilde P}(P)+K_{\tau}(\varphi_{\tilde P})(P)-iS_{\tau}(\varphi_{\tilde P})(P)=2\delta_{\tilde P=P}.
\end{equation}
By finite speed of propagation, we can replace $\tilde\chi(\sqrt{-\Delta_D})$ from \eqref{dispSFWEtoprove} by a cutoff $\chi^{\#}(D_t)$, where $\chi^{\#}\in C^{\infty}((0,4))$ equals $1$ on the support of $\tilde\chi$ and is such that $\chi^{\#}(0)=0$. We are left to prove that $\exists C>0$ independent of $Q,Q_0$ such that, for $|t|>1$ sufficiently large, the following holds
\begin{equation}\label{Waveinteqestim}
\Big|\int_{\mathbb{R}_{+},(\tilde P,P) \in\partial\Omega^{2} } e^{it\tau} \chi^{\#}(\tau)E(\tilde P-Q_0,\tau) \Big(\frac{\partial E(Q- P,\tau)}{\partial \nu(P)}-i E(Q- P,\tau)\Big)\varphi_{\tilde P}( P,\tau)d\sigma(P)d\sigma (\tilde P)\frac{d\tau}{\pi}\Big|
\leq \frac{C}{|t|}.
\end{equation}
Proving \eqref{Waveinteqestim} requires knowledge of $\varphi_{\tilde P}(P,\tau)$. %
From the integral equation \eqref{phiPPtildeform} satisfied by $\varphi_{\tilde P}$ and Lemma \ref{leminvoper} it follows that $\varphi_{\tilde P}=(I+\Psi_{\tau})(2\delta_{\tilde P})$, where $I+\Psi_{\tau}=(I+K_{\tau}-iS_{\tau})^{-1}$. From the proof of Lemma \ref{leminvoper} it follows that $K_{\tau}-iS_{\tau}$ is a pseudo-differential operator of order $-1$, hence $\Psi_{\tau}$ is also a pseudo-differential operator of order $-1$. Therefore there exists $C>0$ such that for all $\tilde P,P\in\partial\Omega$
\[
\varphi_{\tilde P}(P,\tau)=2\delta_{\tilde P=P}+\Psi_{\tau}(2\delta_{\tilde P})(P),\text{ where, for } \tau \text{ bounded}, |\Psi_{\tau}(2\delta_{\tilde P})(P)|\leq {C}{|P-\tilde P|^{-1}}\,.
\]
For $P, \tilde P\in \partial\Omega$ and for $\tau$ bounded, the phase $\tau |P-\tilde P|$ of $\Psi_{\tau}(2\delta_{P})(\tilde P)$ is bounded. The phase function in \eqref{Waveinteqestim} equals
$\tau(t+|\tilde P-Q_0|+|P-Q|)$ with $|t|$ sufficiently large: if $|t+|\tilde P-Q_0|+|P-Q||\geq |t|/10$, we can perform integrations by parts with respect to $\tau$ and obtain a contribution $O(|t|^{-\infty})$; if $|t+|\tilde P-Q_0|+| P-Q||\leq |t|/10$, we can either have $|\tilde P-Q_0|>|t|/10$ for all $\tilde P\in \partial \Omega$ or $|P-Q|>|t|/10$ for all $P\in \partial\Omega$ (since $\Theta$ is bounded). In both cases, due to the factors $|\tilde P-Q_0|^{-1}|P-Q|^{-1}$ in the integral \eqref{Waveinteqestim}, we obtain (at least) a factor $|t|^{-1}$ and it remains to prove that the integrals with respect to $\tau, \tilde P,P$ are bounded. This follows immediately using $|\int f(z)dz|\leq \int |f(z)|dz$. 
As $\tau\sim 1$ on the support of $\chi$, it remains to prove that, uniformly with respect to $\tau$, the following holds:
\[
\int_{\tilde P\in \partial\Omega}\int_{P \in \partial\Omega}\Sigma(Q,P,t)\frac{\varphi_{\tilde P}(P,\tau)}{|\tilde P-Q_0|| P-Q|} d\sigma(P) d\sigma \tilde P<\frac{C}{|t|},
\]
where, as $\frac{\partial E(Q- P,\tau)}{\partial \nu(P)}=e^{i\tau |P-Q|}\Big[i\tau \frac{<Q-P,\nu(P)>}{|Q-P|^2}- \frac{<Q-P,\nu(P)>}{|Q-P|^3}\Big]$ and  $\varphi_{\tilde P}(P,\tau)=2\delta_{\tilde P=P}+\Psi_{\tau}(2\delta_{\tilde P})(P)$, 
\[
\Sigma (Q,P,\tau)=\chi^{\#}(\tau) \Big[-\frac{<Q-P, \nu(P)>}{|Q-P|^2}+i\tau \frac{<Q-P, \nu(P)>}{|Q-P|}-i\Big].
\]
As $\partial\Omega$ is $C^2$, it follows that $|\Sigma(Q, P,\tau)|\leq C$ for some constant $C$ independent of $Q, P,\tau$. We split the last integral into two parts, corresponding to the two contributions of $\varphi_{\tilde P}(P,\tau)$. As $|\Psi_{\tau}(2\delta_{\tilde P})(P)|\lesssim \frac{1}{|P-\tilde P|}$, we left with proving that there exists some constant $C>0$ independent of $Q,Q_0$ such that 
\[
\int_{\tilde P\in\partial \Omega}\int_{P \in\partial \Omega}\frac{1}{|Q- P|}\frac{1}{|Q_0-\tilde P|}\Big(\delta_{P=\tilde P}+\frac{1}{|P-\tilde P|}\Big) d\sigma(P) d\sigma (\tilde P)<\frac{C}{|t|}\,.
\]
Notice that we can bound either $\frac{1}{|Q- P|}$ by $10/|t|$ for every $P\in \partial\Omega$ or we can bound  $\frac{1}{|Q_0-\tilde P|}$ by $10/|t|$ for every $\tilde P\in \partial\Omega$. As the boundary is $2$-dimensional, the remaining inequality holds true.
\end{proof}

\section{Proof of Theorem \ref{thmdisp3D} for the Schr\"odinger equation}\label{Sec-Schro}
\subsection{The high-frequency regime}
In this section, we prove Theorem \ref{thmdisp3D} for the Schr\"odinger equation 
in the high-frequency case, building on arguments from Section \ref{secdispext3D}. 
Let $\tilde \chi\in C_{0}^\infty((0,4))$ valued in $[0,1]$, equal to $1$ on $[\frac 12,2]$, $\tilde\chi(0)=0$ and fix $\chi_1\in C^{\infty}([\frac 34,\frac 32])$ such that $\tilde\chi(|\xi|)+\sum_{j\geq 1}\chi_1(|\xi|/2^j)=1$. Denote $\chi_{j+1}(\tau)=\chi_1(2^{-j}\tau)$ for $j\geq 1$. On the domain $\Omega= \mathbb{R}^3\setminus B_3(0,1)$, one has the spectral resolution of the Dirichlet Laplacian $\Delta$, and we may define a smooth spectral projection $\chi_j(\sqrt{-\Delta})$ as a bounded operator on $L^2$. Let $v$ the solution to \eqref{SchE} with $v|_{t=0}=v_0=\delta_{Q_0}$, $Q_0\in\Omega$. To estimate the $L^{\infty}$ norm of $v$ we study separately $\tilde\chi(\sqrt{-\Delta})v$ and $\sum_{j\geq 1}\chi_{j+1}(\sqrt{-\Delta})v$. In this section we deal with the regime of large frequencies $(1-\tilde\chi)(\sqrt{-\Delta})v=\sum_{j\geq 1}\chi_{j+1}(\sqrt{-\Delta})v$ and prove the following:

\begin{lemma}
There exists a uniform constant $C$ such that, for all $t>0$, $Q,Q_0\in \Omega$,
\[
\Big|(1-\tilde\chi)(\sqrt{-\Delta_D})v(Q,Q_0,t)\Big|\leq {C}{t^{-3/2}}.
\]
\end{lemma}
\begin{proof}
Fix $t>0$. The proof uses in an essentially way our previous result on the wave equation, together with the Kana\"i transform, which is the following subordination formula,
\begin{equation}\label{eqKanaiSchrod}
v(Q,Q_0,t)=\frac{2}{t^{1/2}}\int_0^{\infty} e^{-i\frac{T^2}{4t}}u(Q,Q_0,T)dT,
\end{equation}
where $u(Q,Q_0,T)=\cos(T\sqrt{-\triangle})(\delta_{Q_0})$ is the solution of \eqref{WEOC} for $d_1=3$, $d_2=0$. Let ${H}(Q,Q_0,T):=\frac{\sin(T\sqrt{-\Delta})}{\sqrt{-\Delta}}(\delta_{Q_0})$, then $u(Q,Q_0,T)=\partial_T H(Q,Q_0,T)$. 
For $Q\in \Omega$, using \eqref{Gbarform}, we obtain
\begin{equation}\label{defHWE}
H(Q,Q_0,T)=\frac{\delta(T-|Q-Q_0|)}{4\pi |Q-Q_0|}-\int e^{iT\tau} I(Q,Q_0,\tau)d\tau,
\end{equation}
where $I(Q,Q_0,\tau)=\mathcal{F}(u^{\#})(Q,Q_0,\tau)/\tau$. Recall that in Proposition \ref{propI0} we set $I_{\kappa_{\e_0}}(Q,Q_0,\tau):=\mathcal{F}(u^{\#}_{\kappa_{\e_0}})(Q,Q_0,\tau)/\tau$, with $u^{\#}$ defined in \eqref{Gbarform} (and $u^{\#}_{\kappa}$ as in \eqref{vhformnewchi}) and then proved in \eqref{boundI0horsOS} that $|I_{\kappa_{\e_0}}(Q,Q_0,\tau)|\lesssim 1/T$ uniformly with respect to $Q,Q_0$ and for all $T\in [T_*(Q,Q_0),T_*(Q,Q_0)+C_0]$ where $T_*(Q,Q_0)=d(Q,\partial\Omega)+d(Q_0,\partial\Omega)$ and where $C_0$ is a uniform constant depending only on the diameter of the obstacle $B_3(0,1)$. One integration by parts in $T$  in \eqref{eqKanaiSchrod} yields
\begin{equation}\label{formvSchrod}
v(Q,Q_0,t)=\frac{i}{t^{3/2}}\int_0^{\infty} e^{-i\frac{T^2}{4t}}T H(Q,Q_0,T)dT.
\end{equation}
Introducing \eqref{defHWE} into \eqref{formvSchrod} yields two terms: the free Schrödinger flow, obtained from the first term in \eqref{defHWE}, does satisfy the estimates \eqref{dispschrodrd}, so we are left with the second term in \eqref{defHWE}. We denote
\begin{equation}\label{formvSchroddiez}
v^{\#}(Q,Q_0,t)=\frac{i}{t^{3/2}}\int_0^{\infty} e^{-i\frac{T^2}{4t}}T \int e^{iT\tau} I(Q,Q_0,\tau)d\tau dT.
\end{equation}
From finite speed of propagation for waves, it follows that, at fixed $Q_0,Q$, it will be enough to consider only $T\in [T_*(Q,Q_0), T_*(Q,Q_0)+C_0]$ (as for $T$ outside this set, one has rapid decay). Writing
\[
(1-\tilde\chi)(\sqrt{-\Delta})v^{\#}(Q,Q_0,t)=\frac{i}{t^{3/2}}\sum_{j\geq 1}\int_0^{\infty}e^{-i\frac{T^2}{4t}}T \int e^{iT\tau} \chi_{j+1}(\sqrt{-\Delta})I(Q,Q_0,\tau)d\tau dT,
\]
it remains to prove that there exists $C>0$ independent of $Q,Q_0,t$ such that 
\[
\Big|\sum_{j\geq 1}\int_{T^*(Q,Q_0)}^{T^*(Q,Q_0)+C_0}e^{-i\frac{T^2}{4t}}T\int e^{iT\tau}\chi_{j+1}(\sqrt{-\Delta}) I(Q,Q_0,\tau)d\tau dT\Big|\leq C.
\] 
By finite speed of propagation of the wave flow, we can introduce a cutoff $\chi\in C^{\infty}_0((\frac 12, 2))$, $\chi=1$ on the support of $\chi_1$ such that $\chi_1(2^{-j}\sqrt{-\Delta})(1-\chi(2^{-j}D_T))u^{\#}=O(2^{-jN})$ for all $N\geq 1$, therefore it remains to prove uniform bounds for the sum over $j\geq 1$ with $\chi_{j+1}(\sqrt{-\Delta})$ replaced by $\chi_1(2^{-j}D_T):=\chi_{j+1}(D_T)$.
In the following we prove that there exists a constant $C>0$ independent of $Q,Q_0$ such that
\begin{equation}\label{estpourSchrod}
\Big|\sum_{j\geq 1}\int_{T^*(Q,Q_0)}^{T^*(Q,Q_0)+C_0}e^{-i\frac{T^2}{4t}}T \int e^{iT\tau}\chi_1(2^{-j}\tau) I(Q,Q_0,\tau)d\tau dT\Big|\leq C.
\end{equation}
Let $Q\in \Omega$ and write $I=I_{\kappa_{\e_0}}+I_{1-\kappa_{\e_0}}$, where we set $\chi_1(h_j\tau)I_{\kappa}(Q,Q_0,\tau)=\mathcal{F}(u^{\#}_{h_j,\kappa})(Q,Q_0,\tau)/\tau$, with $h_j=2^{-j}$, $\kappa\in \{\kappa_{\e_0}, 1-\kappa_{\e_0}\}$ is as in Definition \ref{dfnkappa} and $u^{\#}_{h_j,\kappa}$ is given in \eqref{vhformnewchi}. As noticed in the proof of Lemma \ref{lemtrans},  
$I_{1-\kappa_{\e_0}}(Q,Q_0,\tau)$ has an unique critical point $P(Q,Q_0)$, which is non-degenerate (if $Q\in \Omega^+_{Q_0}$, then $P(Q,Q_0)$ is the point on the boundary through which passes the geodesic from $Q_0$ to $Q$, and if $Q\notin \Omega^+_{Q_0}$, $P(Q,Q_0)$ is the intersection point between $QQ_0$ and $\partial\Omega$). Let $T_c(Q,Q_0):=|P(Q,Q_0)-Q|+|P(Q,Q_0)-Q_0|$, then $I_{1-\kappa_{\e_0}}(Q,Q_0,\tau)$ reads as an integral with phase function $-i\tau T_c(Q,Q_0)$ and with classical symbol of order $0$ in $\tau$ and satisfies $|TI_{1-\kappa_{\e_0}}(Q,Q_0,\tau)|\lesssim 1$ for all $T\in [T^*(Q,Q_0), T^*(Q,Q_0)+C_0]$.
The phase function $-T^2/(4t)+\tau(T-T_c(Q,Q_0))$ of \eqref{estpourSchrod} has critical points $T=T_c(Q,Q_0)$ and $\tau=T_c(Q,Q_0)/(2t)$. Let $j_c$ be such that $T_c(Q,Q_0)/(2t)\in [2^{j_c-1}, 2^{j_c+1}]$: for $j\neq  j_c$ and for $\tau$ on the support of $\chi_{j+1}$, the phase is non-stationary with respect to either $\tau$ or $T$ and the sum over $j$ of all the contributions obtained by applying the non stationary phase lemma is uniformly bounded. For $j=j_c$, make the change of variables $\tau=\tilde \tau T_c(Q,Q_0)/(2t)$. Let also $T:=\tilde T T_c(Q,Q_0)$, then for $\tilde T\notin [1/2,2]$ the phase is non-stationary in $\tilde T$. It remains to prove uniform bounds for
\begin{equation}\label{estpourSchrodbis}
\frac{T^2_c(Q,Q_0)}{4t}\Big|\int_{1/2}^{2}e^{i\frac{T^2_c(Q,Q_0)}{2t}(-\frac{\tilde T^2}{2}+\tilde\tau(\tilde T-1))} \chi_1(2^{-j_c}\frac{T_c(Q,Q_0)}{2t}\tilde\tau)T_c(Q,Q_0)\tilde T I_{1-\kappa_{\e_0}}(Q,Q_0,\frac{T_c(Q,Q_0)}{2t}\tilde \tau)d\tilde \tau d\tilde T\Big|,
\end{equation}
which follow by usual stationary phase using that $T_c(Q,Q_0) I_{1-\kappa_{\e_0}}(Q,Q_0,\tilde \tau)$ is uniformly bounded. 

We proceed in a very similar way with $I=I_{\kappa_{\e_0}}$, at least as long as $\text{dist}(Q_0,\partial\Omega)\geq c>0$. Using the notations of Sections \ref{secffb}, \ref{secofoc} with $Q_0=Q_N(s)$ and $Q=(r,y,\omega)$, the phase function of $I_{\kappa_{\e_0}}(Q,Q_0,\tau)$ equals $-i\tau(\sqrt{r^2-1}+\sqrt{s^2-1}+\arcsin(1/r)+\arcsin(1/s))$. 
In this case we let $T_c(r,s):=\sqrt{r^2-1}+\sqrt{s^2-1}+\arcsin(1/r)+\arcsin(1/s)$ and proceed as with $I_{1-\kappa_{\e_0}}$, using again the fact that $T_c(r,s)I_{\kappa_{\e_0}}(Q,Q_0,\tau)$ is uniformly bounded as proved in \eqref{boundI0horsOS}. When $\text{dist}(Q_0,\partial\Omega)>c>\text{dist}(Q,\partial\Omega)$, the phase function of $I_{\kappa_{\e_0}}$ may be equal to $\varphi^{\pm}(\beta^{\pm}_c,\cdot)$ defined in \eqref{phaseJpmk}, hence there might be at most four critical points $T_c$, that are dealt with exactly as before. When $\text{dist}(Q_0,\partial\Omega),\text{dist}(Q,\partial\Omega)\leq c$, 
we combine the previous arguments with those of Section \ref{secctb}: by Kanaï transform, we also have
\begin{equation}\label{VB}
v(Q,Q_0,t)=\int_0^{+\infty} e^{it\tau^2} \mathcal{R}(Q,Q_0,\tau)\frac{d\tau}{\pi},
\end{equation}
with $\mathcal{R}(Q,Q_0,\tau)$ as in \eqref{flotlinwave1} and we have to prove $\Big|(1-\tilde\chi)(\sqrt{-\Delta_D})v(Q,Q_0,t)\Big|\leq \frac{C}{t^{3/2}}$ for some uniform constant $C$; taking $\chi$ equal to $1$ on the support of $\chi_1$, this reduces to obtaining uniform bounds for  
\begin{multline}\label{tildeRschrod}
t^{3/2}\int_0^{+\infty} e^{it\tau^2}(\sum_{j\geq 1}\chi(2^{-j}\tau^2)) \mathcal{R}(Q,Q_0,\tau)\frac{d\tau}{\pi}= \\
t^{3/2}\sum_{\nu=m+1/2, m\in\mathbb{N}}Z^{\nu-1/2}_{\frac{Q}{|Q|}}(\frac{Q_0}{|Q_0|}) \int_0^{\infty}e^{it\tau^2}(\sum \chi_j(\tau^2))G_{\nu}(|Q|,|Q_0|,\tau)\frac{d\tau}{\pi}.
\end{multline}
We can now proceed exactly as in Section \ref{secctb}, taking advantage of the fact that the second derivative w.r.t. $\tau$ of our new phase is always large. 
This completes the proof of Theorem \ref{thmdisp3D}.
\end{proof}

\subsection{The small frequency case} \label{sect3DH}
We obtain sharp dispersion bounds for the Schrödinger equations in $\Omega=\mathbb{R}^3\setminus \Theta$, where $\Theta$ is any compact domain with smooth and strictly geodesically convex boundary. We use the notations and results from section \ref{secHelm}. Let $\tilde\chi\in C_{0}^\infty((0,4))$, $\tilde\chi=1$ on $[\frac 12,2]$, $\tilde\chi(0)=0$. 
\begin{lemma}\label{lemSchrodSF}
Let $v(Q,Q_0,t)$ denote the solution to the Schr\"odinger equation \eqref{SchE}, with $v_0=\delta_{Q_0}$ and Dirichlet boundary condition on $\Omega$. 
There exists a constant $C>0$ independent of $Q,Q_0$ such that
\[
\|\tilde\chi(\sqrt{-\Delta})v(Q,Q_0,t)\|_{L^{\infty}(\Omega)}\leq {C}{|t|^{-3/2}}\,,\quad \forall Q\in \Omega.
\]
\end{lemma}
\begin{proof} (of Lemma \ref{lemSchrodSF})
Fix $|t|>1$ large enough, as otherwise the estimate is trivial.
Replacing \eqref{Resform} in formula \eqref{VB} yields two terms. 
As the free Schr\"odinger flow $\int_0^{+\infty} e^{it\tau^2} E(Q-Q_0,\tau)\frac{d\tau}{\pi}$ satisfies the usual dispersive estimates, we focus on the second term in the formula \eqref{Resform} defining $R(.,\tau)$, where $g_{\tilde P}(Q,\tau)$ is of the form \eqref{gPtildeform} with density $\varphi_{\tilde P}(P)$ solution to \eqref{phiPPtildeform}. Replacing $\tilde\chi(\sqrt{-\Delta})$ by $\chi^{\#}(\tau)$, where $\chi^{\#}\in C^{\infty}_0((0,4))$ is equal to $1$ on the support of $\tilde \chi$ (and $\chi^{\#}(0)=0$), we are left to prove that $\exists C>0$ independent of $Q,Q_0$ such that, for $|t|>1$ large, the following holds
\begin{equation}\label{Schrodinteqestim}
\Big|\int_{\mathbb{R}_{+}, (\tilde P,P) \in \partial\Omega^{2}} e^{it\tau^2} \chi^{\#}(\tau)E(\tilde P-Q_0,\tau)  \Big(\frac{\partial E(Q- P,\tau)}{\partial \nu(P)}-i E(Q- P,\tau)\Big)\varphi_{\tilde P}( P,\tau)d\sigma(P)d\sigma (\tilde P)\frac{d\tau}{\pi}\Big|
\leq \frac{C}{|t|^{3/2}}.
\end{equation}
Let $\Psi_{\tau}$ be the pseudo-differential operator of order $-1$ such that $I+\Psi_{\tau}=(I+K_{\tau}-iS_{\tau})^{-1}$. Then $\varphi_{\tilde P}(P,\tau)=2\delta_{\tilde P=P}+\Psi_{\tau}(2\delta_{\tilde P})(P)$ where $|\Psi_{\tau}(2\delta_{\tilde P})(P)|\leq \frac{C}{|P-\tilde P|}$ for some constant $C>0$ and $\tau>0$ bounded. As such, for bounded $\tau$, the phase $\tau |P-\tilde P|$ of $\Psi_{\tau}(2\delta_{P})(\tilde P)$ is bounded for $P, \tilde P\in \partial\Omega$: we may therefore apply stationary phase to the integral \eqref{Schrodinteqestim} for the phase function
$t(\tau^2+\tau\frac{|P-Q_0|+|\tilde P-Q|}{t})$ whose critical point is given by $\tau_c=-\frac{|P-Q_0|+|\tilde P-Q|}{2t}$ and must belong to the support of $\chi^{\#}(\tau)$ which is a fixed, compact subset of $(0,4)$. Our large parameter is $t$ and the second derivative with respect to $\tau$ equals $2t$; we obtain a factor $|t|^{-3/2}$ and it remains to show that
\[
\int_{\tilde P\in \partial\Omega}\int_{P \in \partial\Omega}e^{-i\frac{(|P-Q_0|+|\tilde P-Q|)^2}{4t}}\Sigma_v(P,\tilde P,t)\frac{(|P-Q_0|+|\tilde P-Q|)}{|P-Q_0||\tilde P-Q|} d\sigma(P) d\sigma \tilde P<C,
\]
where the symbol $\Sigma_v$ is an asymptotic expansion with small parameter $1/t$ and main contribution the symbol of $2\delta_{\tilde P=P}+\Psi_{\tau_c}(2\delta_{\tilde P})(P)$. Indeed, in \eqref{Schrodinteqestim} the symbol is
\[
\chi^{\#}(\tau) \tau\Big[-\frac{<Q-P, \nu(P)>}{|Q-P|^3}+i\tau \frac{<Q-P, \nu(P)>}{|Q-P|^2}-\frac{i}{|P-Q|}\Big] \varphi_{\tilde P}(P,\tau),
\]
hence a derivative with respect to $\tau$ only provides factors $|P-\tilde P|$ coming from $\Psi_{\tau}(2\delta_{\tilde P})(P)$; the only factors depending on $Q$ or $Q_0$ are $|Q-\tilde P|^{-1}$ and $|Q_0-P|^{-1}$.
We split the integral \eqref{Schrodinteqestim} into two parts, corresponding to the two contributions $2\delta_{\tilde P=P}$ and $\Psi_{\tau}(2\delta_{\tilde P})(P)$ of $\varphi_{\tilde P}(P,\tau)$. As $|\Psi^{(-1)}_{\tau_c}(2\delta_{\tilde P})(P)|\lesssim \frac{1}{|P-\tilde P|}$, we are left with proving that there exists some constant $C>0$ independent of $Q,Q_0$ such that 
\[
\int_{\tilde P\in \partial\Omega}\int_{P \in \partial\Omega}\Big(\frac{1}{|Q-\tilde P|}+\frac{1}{|Q_0-P|}\Big)\Big(\delta_{P=\tilde P}+\frac{1}{|P-\tilde P|}\Big)\, d\sigma(P) d\sigma (\tilde P)<C.
\]
As the boundary is two dimensional, we have $|P-Q|^{-1}\in L^1(\partial\Omega)$, hence the inequality holds true. This achieves the proof of Lemma \ref{lemSchrodSF} and the one of Theorem \ref{thmdisp3D} for Schrödinger case.
\end{proof}
\subsection{Proof of Theorem \ref{thmCE} for the Schrödinger equation }
\label{sectCE}
Let $d=d_1+d_2$, with $d_1\geq 3$ and $d_2\geq0$. 
Let $v(Q,Q_0,t)$ denote the solution to the Schrödinger equation \ref{SchE} with $v_0=\delta_{Q_0}$ and $\Omega_{d_1,d_2}=(\mathbb{R}^{d_1}\setminus B_{d_1}(0,1))\times \mathbb{R}^{d_2}$, then, from \eqref{VB}, we get that, for
 $\chi\in C_{0}^\infty((0,\infty))$, the following holds
\begin{equation}\label{L1}
\chi(hD_{t})v(Q,Q_0,t)=\int_{0}^\infty
e^{it\tau^2}\chi(h\tau^2)\mathcal{R}(Q,Q_0,\tau)\frac{d\tau}{\pi},
\end{equation}
where $\mathcal{R}$ as in \eqref{flotlinwave1}. Explicit computations show that $\mathcal{R}$ is real and using \eqref{Fourtransglancing}, we have :
\begin{lemma}\label{Lemme-onde-1}
Let $Q_{N,S}(s)$ as in Theorem \ref{thmCE}. For large $\tau$ and $s$ so that $y_*(s)=\arcsin(1/s)\sim \tau^{-1/3}$,
\begin{equation}\label{L3}
\mathcal{R}(Q_{N}(s),Q_{S}(s),\tau)=
\frac{\tau^{d-2+\frac 13-\frac{d_2}{2}}}{\sqrt{s^2-1}^{d_1-1+\frac{d_2}{2}}} \Re \Big(e^{-i\tau(2\sqrt{s^2-1}+2y_*(s))} r(\delta,\tau)\Big),\quad \delta=y_*(s)\tau^{1/3}\sim 1,
\end{equation}
where $r(\delta,\tau)\sim_{\tau^{-1/3}} \sum_{j\geq 0}r_{j}(\delta)\tau^{-j/3}$ is a 
classical symbol for $\delta\sim 1$, and with
$r_{0}(\delta)\not= 0$ for $\delta\sim 1$.
\end{lemma}
Let $0<h_0< 1$ small enough as in Theorem \ref{thmCE} for the wave equation and $\chi\in C^{\infty}_0((1/2,2))$, $\chi=1$ near $1$. Let $0<\tilde h\leq h_0$, then \eqref{L3} holds for $\tau$ on the support of $\chi(\tilde h\tau)$ where we have $\tau\sim 1/\tilde h$. As in \eqref{L1} the cut-off has the form $\chi(h\tau^2)$, we let $h:=\tilde h^2$ where now $0<h<h_0^2$. Setting $\tau=(\sigma/h)^{1/2}$ in the formula \eqref{L1} gives $1\sim h\tau^2\sim \sigma$ on the support of $\chi(h\tau^2)$, $\tau\sim h^{-1/2}$ and 
we 
obtain
\begin{equation}\label{L20}
\frac{1}{\pi}\int_{0}^\infty
e^{it\tau^2}\chi(h\tau^2)\mathcal{R}(Q,Q_0,\tau)d\tau=\frac{1}{\pi h^{1/2}}\int_{0}^\infty e^{\frac{i t\sigma}{h}}\chi(\sigma)\mathcal{R}(Q,Q_0,(\sigma/h)^{1/2}) d\sigma.
\end{equation}
Let $s$ such that $y_*(s)=\arcsin(1/s)=\delta h^{1/6}\sim \delta \tau^{1/3}$, then at
 $Q=Q_{N}(s)$, $Q_0=Q_{S}(s)$ \eqref{L3} becomes
 \begin{multline}\label{L21}
\chi(\tilde h\tau)\mathcal{R}(Q_{N}(s),Q_{S}(s),\tau) = \chi(h^{1/2}(\sigma/h)^{1/2})\mathcal{R}(Q_{N}(s),Q_{S}(s),(\sigma/h)^{1/2})\\
 =\chi(\sigma) \frac{(\sigma/h)^{\frac{d-2+1/3-d_2/2}{2}}}{\sqrt{s^2-1}^{d_1-1+d_2/2}}
Re\Big(e^{-i(\sigma/h)^{1/2} (2\sqrt{s^2-1}+2y_*(s))}r(\delta,(\sigma/h)^{1/2})
\Big),
 \end{multline}
 where $2(y_*(s)+\sqrt{s^2-1})=2(\delta h^{1/6}+\cot(\delta h^{1/6}))$ is the geodesical distance (in $\Omega_{d_1,d_2}$) between $Q_{N}(s)$ and $Q_{S}(s)$. We introduce \eqref{L21} in \eqref{L20} and apply the stationary phase for the phase function 
 \[
t\sigma /h-2\sqrt{\sigma}h^{-1/2}(\delta h^{1/6}+\cot(\delta h^{1/6})), \quad \sigma\sim 1.
\]
At $t=h^{1/3}$, the large parameter is $h^{-2/3}$ and the phase equals 
$\sigma-2\sqrt{\sigma} \Big(\delta h^{1/3}+\cos(\delta h^{1/6})\frac{ h^{1/6}}{\sin(\delta h^{1/6})}\Big)$, where the term in brackets is independent of $\sigma$ and is $\sim \frac{1}{\delta}$, hence the critical point satisfies $\sigma_c\sim 1/\delta^2$. At $t=h^{1/3}$, for $0<h<h_0^2$ and $\arcsin (1/s)=\delta h^{1/6}$ with $\delta\sim 1$, hence for $s\sim h^{-1/6}$, we obtain the desired result, concluding our proof, as the stationary phase yields a factor
 \begin{equation}\label{L4}
 \frac{1}{h^{1/2}}\times \chi(\sigma) \frac{(\sigma/h)^{\frac{d-2+1/3-d_2/2}{2}}}{\sqrt{s^2-1}^{d_1-1+d_2/2}}\times \sqrt{h^{2/3}}|_{\sigma\sim 1/\delta^2, s\sim h^{-1/6}}=\frac{1}{t^{d/2}}\Big|_{t=h^{1/3}}\times h^{-(d_1-3)/6}.
\end{equation}

\section{Appendix}
 
 \subsection{Airy functions} \label{secAi}
For $w\in\C$, the Airy function is defined by $A(w)=\frac{1}{2\pi}\int_{\R}e^{\ii(s^3/3+sw)}ds$.
Let $A_{\pm}(w):=A(e^{\mp2\ii\pi/3}w)$, then $A_-(w)=\bar{A}_+(\bar{w})$, $A(w)=e^{\ii\pi/3}A_+(w)+e^{-\ii\pi/3}A_{-}(w)$ and $A_{\pm}(w),A'_{\pm}(w)$ do not vanish for $w\in\R$, while all the zeros of $Ai(w)$ and $Ai'(w)$ are real and non positive. The following Lemma holds :
 
\begin{lemma}\label{lem:Phi+}
Let $\Sigma(w):=(A_+(w)A_{-}(w))^{1/2}$, then $\Sigma(w)=|A_+(w)|=|A_{-}(w)|$ is real, monotonic increasing in $w$ and nowhere vanishing. We let $\mu(w):=\frac{1}{2\ii}\log(\frac{A_-(w)}{A_+(w)})$ for $w<-1$. Then $A_{\pm}(w)=\Sigma(w)e^{\mp\ii\mu(w)}$ for $w<-1$.
Moreover, for $w<-1$, the following asymptotic expansions hold
\begin{align}\label{eq:Phi+}
\Sigma(w)& \sim_{1/|w|}(-w)^{-\frac14}\sum_{j\geq 0}\sigma_j(-w)^{-\frac{3j}{2}},\,\, \mu(w)\sim_{1/|w|}\frac{2}{3}(-w)^{\frac32}\sum_{j\geq 0}e_j(-w)^{-\frac{3j}{2}},\,\,\sigma_0=\frac{1}{2\sqrt{\pi}}, \,\,e_0=1,\\
\frac{A'_+(w)}{A_+(w)}&\sim_{1/|w|} (-w)^{1/2} \sum_{j\geq 0}d_j (-w)^{-3j/2},\quad d_0=1,\quad 
\frac{A_{-}(w)}{A_{+}(w)}\sim_{1/|w|} e^{2\ii\mu(w)}, \text{ when } w\rightarrow -\infty.
\end{align}
For $w>1$, the Airy function $Ai(w)$ decays exponentially, $A(w)\sim_{1/w} |w|^{-\frac14}e^{-\frac23 w^{3/2}}$, while the functions $A_{\pm}(w)$ grow exponentially $A_{\pm}(w)=\Sigma_{\pm}(w)e^{\frac23w^{3/2}}$, where $\Sigma_{\pm}$ are classical symbols of order $-1/4$ and 
\[
\frac{A(w)}{A_{+}(w)}=e^{i\pi/3}+e^{-i\pi/3}\frac{A_-(w)}{A_{+}(w)}=O(w^{-\infty}).
\]
\end{lemma}

\subsubsection{The Hankel function}\label{secHank}
In \eqref{greenfctGnu}, $H^{(1)}_{\nu}(z)$ is the Hankel function (also known as the Bessel function of the third kind): it solves the Bessel's equation $z^2H''(z)+zH'(z)+(z^2-\nu^2)=0$, $\nu=m+\frac 12$, $m\in\mathbb{N}$; its real and imaginary parts, denoted $J_{\nu}(z)$ and $Y_{\nu}(z)$, are the usual Bessel functions. 
For fixed order $\nu$ and large argument $z$, we have the asymptotics
\begin{equation}\label{Ai}
H^{(1)}_{\nu}(z)=\sqrt{\frac{2}{\pi z}}e^{i(z
-\frac{\pi\nu}{2}-\frac{\pi}{4})}(1+O(1/z)).
\end{equation}
For large order $\nu$ and argument $z=\nu \rho$, the Hankel function has the following uniform asymptotic expansion, derived by R.Langer and F.Olver: (see \cite[9.3.35, 9.3.37]{Ab})
\begin{gather}\label{Aiiicj}
\quad \quad H^{(1)}_{\nu}(\nu \rho)\sim_{1/\nu} 2e^{-i\frac \pi 3}\Big(\frac{4\tilde \zeta(\rho)}{1-\rho^{2}}\Big)^{\frac 14}\Big(\nu^{-\frac 13}A_+(\nu^{\frac 23}\tilde \zeta(\rho))\sum_{j\geq 0}f_j(\tilde\zeta(\rho))\nu^{-2j}+\nu^{-\frac 53}A'_+(\nu^{\frac 23}\tilde \zeta(\rho))\sum_{j\geq 0}g_j(\tilde\zeta(\rho))\nu^{-2j}\Big)\,,\\
\label{HAiiicj}
J_{\nu}(\nu \rho)\sim_{1/\nu} \Big(\frac{4\tilde \zeta(\rho)}{1-\rho^{2}}\Big)^{\frac 14}\Big(\nu^{-\frac 13}A(\nu^{\frac 2 3}\tilde \zeta(\rho))\sum_{j\geq 0}f_j(\tilde\zeta(\rho))\nu^{-2j}+\nu^{-\frac 5 3}A'(\nu^{\frac 2 3}\tilde \zeta(\rho))\sum_{j\geq 0}g_j(\tilde\zeta(\rho))\nu^{-2j}\Big)\,,
\end{gather}
where the function $\tilde \zeta=\tilde \zeta(\rho)$ is defined (see \cite[9.3.38, 9.3.39]{Ab}) as \eqref{tildezet>} for $\rho>1$ and \eqref{tildezet<} for $\rho\leq 1$.
The expansion \eqref{Aiiicj} holds uniformly with respect to $\rho$ in the region $|\arg \rho|<\pi-\epsilon$, where $\epsilon$ is an arbitrarily positive number and in particular for real $\rho$ close to $1$.
As for the coefficients (given in \cite[9.3.40]{Ab}), they are smooth as functions of $\rho$ and $f_0(0)$ is non-vanishing.
When the order is very large compared to the argument, the following expansions hold  for $H^{(1)}_{\nu}(z)=J_{\nu}(z)+iY_{\nu}(z)$ (see \cite[9.3.1]{Ab}):
\begin{equation}\label{Ainularge}
J_{\nu}(z)=\sqrt{\frac{1}{2\pi \nu}}\Big(\frac{ez}{2\nu}\Big)^{\nu}(1+O(1/\nu)), \quad Y_{\nu}(z)= -\sqrt{\frac{1}{2\pi \nu}}\Big(\frac{ez}{2\nu}\Big)^{-\nu}(1+O(1/\nu)).
\end{equation}

\def\cprime{$'$} \def\cprime{$'$}

\end{document}